\documentclass[10pt, leqno]{amsart} 
\usepackage{amssymb,amscd,amsfonts,amsbsy}
\usepackage{latexsym}
\usepackage{exscale}
\usepackage{amsmath,amsthm,amsfonts}
\usepackage{mathrsfs}
\usepackage{xcolor}
\usepackage[colorlinks=true, linkcolor=blue, citecolor=red, urlcolor=red]{hyperref}
\usepackage{esint} 
\usepackage{stmaryrd}
\usepackage{pifont}
\usepackage{bm}

\usepackage[utf8]{inputenc}

\calclayout

\allowdisplaybreaks

\makeatletter
\@namedef{subjclassname@2020}{%
  \textup{2020} Mathematics Subject Classification}
\makeatother

\theoremstyle{plain}
\newtheorem{theorem}[equation]{Theorem}

\newtheorem{lemma}[equation]{Lemma}

\theoremstyle{definition}

\numberwithin{equation}{section}

\def\C{\mathbb{C}}
\def\N{\mathbb{N}}
\def\D{\mathcal{D}}

\def\K{\mathcal{K}}
\def\S{\mathcal{S}}

\def\R{\mathbb{R}}
\def\Rn{\mathbb{R}^n}
\def\Rnn{\R^{n_1} \times \R^{n_2}}

\def\L{\mathbf{L}}
\def\p{\mathbf{p}}
\def\q{\mathbf{q}}
\def\r{\mathbf{r}}
\def\u{\mathbf{u}}
\def\v{\mathbf{v}}
\def\w{\mathbf{w}}

\def\Re{\operatorname{Re}}
\def\Im{\operatorname{Im}}
\def\loc{\operatorname{loc}}

\def\supp{\operatorname{supp}}

\def\BMO{\operatorname{BMO}}
\def\CMO{\operatorname{CMO}}

\def\d{\operatorname{d}}
\def\rd{\operatorname{rd}}

\renewcommand{\emptyset}{\text{\textup{\O}}}

\begin{document}

\title[Multilinear extrapolation of compactness]{Multilinear extrapolation of compactness on mixed-norm spaces}

\author{Mingming Cao}
\address{Mingming Cao\\
Institute for Advanced Study in Mathematics\\ 
Harbin Institute of Technology\\ 
Harbin 150001\\
People's Republic of China}
\email{mingming.cao@hit.edu.cn}

\author{Binyu Du}
\address{Binyu Du\\
Institute for Advanced Study in Mathematics\\ 
Harbin Institute of Technology\\ 
Harbin 150001\\
People's Republic of China} 
\email{binyu$\_$du@163.com}

\author{Honghai Liu}
\address{Honghai Liu\\
School of Mathematics and Information Science\\
Henan Polytechnic University\\
Jiaozuo 454000\\
People's Republic of China} 
\email{hhliu@hpu.edu.cn}

\author{Zengyan Si}
\address{Zengyan Si\\
School of Mathematics and Information Science\\
Henan Polytechnic University\\
Jiaozuo 454000\\
People's Republic of China} 
\email{zengyan@hpu.edu.cn}

\author{K\^{o}z\^{o} Yabuta}
\address{K\^{o}z\^{o} Yabuta\\
Research Center for Mathematics and Data Science\\
Kwansei Gakuin University\\
Gakuen 2-1, Sanda 669-1337, Japan}
\email{kyabuta3@kwansei.ac.jp}

\thanks{The first four authors were supported by the National Natural Science Foundation of China (No. 12571101 and 12526417).}

\year=2026 \month=06 \day=01

\date{\today}




\begin{abstract}
In this paper, we develop the Rubio de Francia extrapolation theorem for the multilinear compactness on mixed-norm Lebesgue spaces. More precisely, if a multilinear operator is bounded on weighted product spaces, then its compactness can be extrapolated from unweighted product spaces to the full range of weighted mixed-norm spaces. This result is mainly based on a multilinear interpolation theorem of compactness on weighted mixed-norm spaces, for which we present a characterization of compactness on mixed-norm spaces and a multilinear interpolation theorem of boundedness on multi-mixed-norm spaces. As applications of the extrapolation theorem, we obtain compactness results for several kinds of bi-parameter operators on weighted mixed-norm spaces, including multilinear bi-parameter Calder\'{o}n--Zygmund operators, multilinear bi-parameter dyadic paraproducts, bilinear bi-parameter continuous paraproducts, and bilinear bi-parameter pseudo-differential operators.
\end{abstract}

\maketitle

\section{Introduction}
\subsection{Motivation and main theorems}
The Rubio de Francia extrapolation theorem stands among the most profound and far-reaching discoveries in harmonic analysis, since it provides a universal mechanism to  extend boundedness, vector-valued inequalities, and compactness from a single exponent-weight pair to the full admissible range of exponents and Muckenhoupt weights classes. Since its creation in the 1980s, this principle has redefined the methodology of weighted norm inequalities, unified the study of linear, multilinear, and multi-parameter operators, and become an indispensable tool in singular integrals theory, function spaces, and partial differential equations. By allowing properties to be “extrapolated” from convenient test cases to all admissible settings, it not only eliminates repetitive arguments and streamlines technically demanding proofs, but also reveals underlying structural unity across seemingly disparate problems. 

The purpose of this paper is to establish the Rubio de Francia extrapolation theorem for multilinear compactness on weighted mixed-norm spaces. Our study is primarily motivated by the following considerations:

(1) The relationship between the compactness and mixed-norm estimates has been captured in \cite[Theorem 2.10]{COY}, which is essential to derive multilinear extrapolation of compactness there. This result and extrapolation of compactness have demonstrated their flexibility in applications (cf. \cite{CIRXY, COY}) although reaching the quasi-Banach range $p \in(0, 1)$ usually turns out to be a major challenge. Thus, these facts naturally motivate us to investigate extrapolation of compactness on mixed-norm spaces. 

(2) We attempt to address an open problem left unresolved in our previous work \cite{CY}. While the first and last authors \cite{CY} obtained weighted compactness of multilinear bi-parameter singular integral operators via a compact variant of T1 type assumptions, it is quite difficult to verify whether the bi-parameter extension of original bilinear examples in \cite{CLSY} fulfill such compact T1 type conditions. Accordingly, we seek to justify weighted compactness of these examples and further  explore whether their compactness persists on mixed-norm spaces.

(3) Mixed-norm estimates have found widespread use in analysis and partial diﬀerential equations. Estimates on mixed-norm spaces with product-type weights can be used to study the regularity of solutions, see \cite{CDL22, DK18, DK21, DKP22, DP21}. Beyond that, they also have numerous applications in harmonic analysis,  for example, multilinear fractional integrals \cite{CS20, CS21, LMV21}, the Kato--Ponce inequalities on mixed-norm spaces \cite{HTW, OW}, bi-parameter paraproducts \cite{BM17}, and multilinear multipliers \cite{BM, BM24}.

(4) Our objective of an extensive research program is to investigate multilinear extrapolation of compactness on abstract Banach function spaces. The extrapolation theory of boundedness on Banach function spaces has been established and found a wealth of applications in \cite{CMM}. However, due to the abstract nature of general Banach function spaces, the approach to compactness developed in \cite{COY} cannot be directly extended to this framework. Therefore, as a first step, we focus on multilinear extrapolation of compactness on the mixed-norm spaces, where we would like to explore some new methods in order to extend these results to general Banach function spaces in the future work.

In order to state our main result we need some conceptions. Given quasi-normed spaces $\mathscr{X}_1, \ldots, \mathscr{X}_m, \mathscr{X}$, an $m$-linear operator $T: \mathscr{X}_1 \times \cdots \times \mathscr{X}_m \to \mathscr{X}$ is said to be {\tt compact} if $T(B_1 \times \cdots \times B_m)$ is {\tt precompact} in $\mathscr{X}$ for all bounded sets $B_j \subset \mathscr{X}_j$, $j=1, \ldots, m$, i.e., $\overline{T(B_1 \times \cdots \times B_m)}$ is a compact subset of $\mathscr{X}$. 

Given $0<p, q \le \infty$ and weights $u$ on $\R^{n_1}$ and $v$ on $\R^{n_2}$, the {\tt mixed-norm space} $L^p(u^p; L^q(v^q))$ consists of all measurable functions $f$ on $\Rnn$ satisfying 
\begin{align*}
\|f\|_{L^p(u^p; L^q(v^q))} 
:= \bigg[\int_{\R^{n_1}} \bigg(\int_{\R^{n_2}} |f(x_1, x_2)|^q v(x_2)^q \, dx_2\bigg)^{\frac{p}{q}} u(x_1)^p \, dx_1\bigg]^{\frac1p} < \infty.
\end{align*}

We develop the Rubio de Francia extrapolation of compactness on mixed-norm spaces below. 

\begin{theorem}\label{thm:RdF-cpt}
Assume that $T$ is an $m$-linear operator such that  
\begin{list}{\rm (\theenumi)}{\usecounter{enumi}\leftmargin=1.2cm \labelwidth=1cm \itemsep=0.2cm \topsep=0.2cm \renewcommand{\theenumi}{\alph{enumi}}}

\item\label{RdFcpt-1} $T$ is bounded from $L^{s_1}(w_1^{s_1}) \times \cdots \times L^{s_m}(w_m^{s_m})$ to $L^s(w^s)$ for some $\vec{s} = (s_1, \ldots, s_m) \in [1, \infty]^m$ and for all $\vec{w} = (w_1, \ldots, w_m) \in A_{\vec{s}}(\Rnn)$, where $\frac1s = \sum_{j=1}^m \frac{1}{s_j}$ and $w = \prod_{j=1}^m w_j$;

\item\label{RdFcpt-2} $T$ is compact from $L^{t_1}(\Rnn) \times \cdots \times L^{t_m}(\Rnn)$ to $L^t(\Rnn)$ for some $\vec{t} = (t_1, \ldots, t_m) \in [1, \infty]^m$, where $\frac1t = \sum_{j=1}^m \frac{1}{t_j} > 0$. 
\end{list} 
Then $T$ is compact from $L^{p_1}(u_1^{p_1}; L^{q_1}(v_1^{q_1})) \times \cdots \times L^{p_m}(u_m^{p_m}; L^{q_m}(v_m^{q_m}))$ to $L^p(u^p; L^q(v^q))$ for all $\vec{p} = (p_1, \ldots, p_m) \in (1, \infty]^m$, for all $\vec{q} = (q_1, \ldots, q_m) \in (1, \infty]^m$, for all $\vec{u} = (u_1, \ldots, u_m) \in A_{\vec{p}}(\R^{n_1})$, and for all $\vec{v} = (v_1, \ldots, v_m) \in A_{\vec{q}}(\R^{n_2})$, where $\frac1p = \sum_{j=1}^m \frac{1}{p_j} > 0$, $\frac1q = \sum_{j=1}^m \frac{1}{q_j} > 0$, $u = \prod_{j=1}^m u_j$, and $v = \prod_{j=1}^m v_j$.  
\end{theorem}

The weights $\vec{u} \in A_{\vec{p}}(\R^{n_1})$ and $\vec{v} \in A_{\vec{q}}(\R^{n_2})$ appearing in Theorem \ref{thm:RdF-cpt} are natural in the context of mixed-norm estimates, since they characterize the weighted boundedness of multilinear strong maximal operators and multilinear bi-parameter Riesz transforms, see  Theorem \ref{thm:ApAq}. On the other hand, mixed-norm estimates involving general weights $w(x_1, x_2)$ are excessively intricate, rendering them inapplicable in practice. In contrast, the separated-variable form of weights $w(x_1, x_2) = u(x_1) \otimes v(x_2)$ enjoy greater practical utility. Such weighted mixed-norm estimates have also been widely applied in the regularity estimates of solutions to elliptic and parabolic equations/systems (cf. \cite{CDL22, DK18, DK21, DKP22, DP21}).

Let us define the multilinear strong maximal operator as
\begin{equation*}
\mathcal{M}_{\mathcal{R}}(\vec{f})(x)
:= \sup_{x \in R \in \mathcal{R}} \prod_{j=1}^m \fint_R |f_j(y_j)| \, dy_j, \quad x \in \Rnn, 
\end{equation*}
where $\mathcal{R}$ is the collection of all rectangles $R = I_1 \times I_2$, where $I_i \subset \R^{n_i}$ is a cube. In the case $m=1$, simply write $M_{\mathcal{R}} = \mathcal{M}_{\mathcal{R}}$. Additionally, the multilinear Riesz transform $\mathcal{R}_j^i$ on $\R^{n_i}$ is defined by 
\begin{align*}
\mathcal{R}_j^i (\vec{f})(x) 
:= \mathrm{p.v. } \int_{\R^{m n_i}}
\frac{\sum_{k=1}^m (x - y_k)_j}{\big(\sum_{k=1}^m |x - y_k|^2 \big)^{\frac{mn_i+1}{2}}} 
f_1(y_1) \cdots f_m(y_m)\, d\vec{y},  
\end{align*}
where $(x-y_k)_j$ is the $j$-th coordinate of $x-y_k \in \R^{n_i}$, $j=1, \ldots, n_i$. Let $\mathcal{R}_{j_1}^1 \otimes \mathcal{R}_{j_2}^2$ denote the multilinear bi-parameter Riesz transform on $\Rnn$.

\begin{theorem}\label{thm:ApAq}
Let $\vec{p} = (p_1, \ldots, p_m) \in (1, \infty]^m$, $\vec{q} = (q_1, \ldots, q_m) \in (1, \infty]^m$, $\frac1p = \sum_{j=1}^m \frac{1}{p_j} > 0$, and $\frac1q = \sum_{j=1}^m \frac{1}{q_j} > 0$. Then the following are equivalent:
\begin{list}{\textup{(\theenumi)}}{\usecounter{enumi}\leftmargin=1cm \labelwidth=1cm \itemsep=0.2cm \topsep=0.2cm \renewcommand{\theenumi}{\arabic{enumi}}}

\item\label{list:uv} $\vec{u} = (u_1, \ldots, u_m) \in A_{\vec{p}}(\R^{n_1})$ and $\vec{v} = (v_1, \ldots, v_m) \in A_{\vec{q}}(\R^{n_2})$;

\item\label{list:M} $\mathcal{M}_{\mathcal{R}} : L^{p_1}(u_1^{p_1}; L^{q_1}(v_1^{q_1})) \times \cdots \times L^{p_m}(u_m^{p_m}; L^{q_m}(v_m^{q_m})) \to L^p(u^p; L^q(v^q))$ boundedly;

\item\label{list:R} $\mathcal{R}_{j_1}^1 \otimes \mathcal{R}_{j_2}^2: L^{p_1}(u_1^{p_1}; L^{q_1}(v_1^{q_1})) \times \cdots \times L^{p_m}(u_m^{p_m}; L^{q_m}(v_m^{q_m})) \to L^p(u^p; L^q(v^q))$ boundedly.
\end{list}
\end{theorem}

\subsection{Proof strategy of Theorem \ref{thm:RdF-cpt}}
Our ultimate goal is to establish the extrapolation theorem for compactness on mixed-norm spaces; prior to this, we begin with the simplest setting to illustrate the  main idea. Suppose that $T$ is bounded on $L^{p_0}(w_0^{p_0})$ and $T$ is compact on $L^{p_1}(w_1^{p_1})$, where $1<p_0 \ne q_0 < \infty$, $w_0^{p_0} \in A_{p_0}$, and $w_1^{p_1} \in A_{p_1}$, then one can obtain the compactness of $T$ on $L^p(w^p)$ by interpolation (cf. \cite[Theorem 9]{CK}), where $\frac1p = \frac{1 - \theta}{p_0} + \frac{\theta}{p_1}$ and $w = w_0^{1 - \theta} w_1^{\theta}$. Moreover, when fixing $p_1$ and $w_1$ while varying $p_0 \in (1, \infty)$, we obtain the weighted compactness for the full range $p \in (1, \infty)$, but the weight $w$ is not arbitrary in this case. Therefore, for any target $p \in (1, \infty)$ and $w \in A_p$, one has to find suitable exponent $p_0 \in (1, \infty)$ and weight $w_0 \in A_{p_0}$ satisfying interpolation condition (cf. eq. \eqref{PST-3}). Our proof is carried out under such idea. This means that to prove Theorem \ref{thm:RdF-cpt}, we first present a multilinear interpolation of compactness on weighted mixed-norm spaces as follows.

\begin{theorem}\label{thm:ICPT}
Let $p_0, q_0, \widetilde{p}_0, \widetilde{q}_0 \in (0, \infty)$ and $p_j, q_j, \widetilde{p}_j, \widetilde{q}_j \in [1, \infty]$, $j = 1, \ldots, m$. Let $u_0^{p_0}, v_0^{q_0} \in A_{\infty}(\R^{n_1})$, $\widetilde{u}_0^{\widetilde{p}_0}, \widetilde{v}_0^{\widetilde{q}_0} \in A_{\infty}(\R^{n_2})$, $u_j$ and $v_j$ be weights on $\R^{n_1}$, $\widetilde{u}_j$ and $\widetilde{v}_j$ be weights on $\R^{n_2}$, $j=1, \ldots, m$. Let $T: \mathscr{S} \times \cdots \times  \mathscr{S} \rightarrow \mathscr{M}$ be an $m$-linear operator, where $\mathscr{S}$ is the collection of all simple functions on $\Rnn$, and $\mathscr{M}$ is the set of all measurable functions on $\Rnn$. 
Assume that 
\begin{align}
\label{eq:ICTP-1} &T \text{ is bounded from $L^{p_1}(u_1^{p_1}; L^{\widetilde{p}_1}(\widetilde{u}_1^{\widetilde{p}_1})) \times \cdots \times L^{p_m}(u_m^{p_m}; L^{\widetilde{p}_m}(\widetilde{u}_m^{\widetilde{p}_m}))$ to $L^{p_0}(u_0^{p_0}; L^{\widetilde{p}_0}(\widetilde{u}_0^{\widetilde{p}_0}))$}, 
\\
\label{eq:ICTP-2} &T \text{ is compact from $L^{q_1}(v_1^{q_1}; L^{\widetilde{q}_1}(\widetilde{v}_1^{\widetilde{q}_1})) \times \cdots \times L^{q_m}(v_m^{q_m}; L^{\widetilde{q}_m}(\widetilde{v}_m^{\widetilde{q}_m}))$ to $L^{q_0}(v_0^{q_0}; L^{\widetilde{q}_0}(\widetilde{v}_0^{\widetilde{q}_0}))$}. 
\end{align}
Then, $T$ can be extended as a compact operator from $L^{r_1}(w_1^{r_1}; L^{\widetilde{r}_1}(\widetilde{w}_1^{\widetilde{r}_1})) \times \cdots \times L^{r_m}(w_m^{r_m}; L^{\widetilde{r}_m}(\widetilde{w}_m^{\widetilde{r}_m}))$ to $L^{r_0}(w_0^{r_0}; L^{\widetilde{r}_0}(\widetilde{w}_0^{\widetilde{r}_0}))$ for all exponents satisfying 
\begin{equation}\label{eq:exp}
\begin{aligned}
\theta \in (0, 1),\quad & \frac{1}{r_j} = \frac{1 - \theta}{p_j} + \frac{\theta}{q_j},
\quad w_j = u_j^{1 - \theta} v_j^{\theta}, \quad j = 0, \ldots, m,
\\
\text{and} \quad & \frac{1}{\widetilde{r}_j} = \frac{1 - \theta}{\widetilde{p}_j} + \frac{\theta}{\widetilde{q}_j}, 
\quad \widetilde{w}_j = \widetilde{u}_j^{1 - \theta} \widetilde{v}_j^{\theta}, \quad j = 0, \ldots, m.
\end{aligned}
\end{equation}
\end{theorem}

To show Theorem \ref{thm:ICPT}, it is necessary to give a clear characterization of the compactness on weighted mixed-norm spaces. Significantly, the result below proves highly effective for establishing weighted compactness in quasi-Banach range.

\begin{theorem}\label{thm:KRQB}
Let $p, q \in (0, \infty)$, $p_0, q_0 \in (1, \infty)$, $u \in A_{p_0}(\R^{n_1})$, $v \in A_{q_0}(\R^{n_2})$,  and $\K \subset L^p(u; L^q(v))$. Let $0 < a < \min\{p/p_0, q/q_0, 1\}$. Then $\K$ is precompact in $L^p(u; L^q(v))$ if and only if the following are satisfied:
\begin{list}{\rm (\theenumi)}{\usecounter{enumi}\leftmargin=1cm \labelwidth=1cm \itemsep=0.2cm \topsep=0.2cm \renewcommand{\theenumi}{\alph{enumi}}}
 
\item\label{KRQB-1} ${\displaystyle \sup_{f \in \K} \|f\|_{L^p(u; L^q(v)))} < \infty}$, 

\item\label{KRQB-2} ${\displaystyle \lim_{A \to \infty} \sup_{f \in \K} 
\|f \mathbf{1}_{B(0, A)^c}\|_{L^p(u; L^q(v))}=0}$, 

\item\label{KRQB-3} ${\displaystyle \lim\limits_{r \to 0} \sup\limits_{f \in \K} 
\bigg\|\bigg[\fint_{B(0, r)} |\tau_y f - f|^a \, dy \bigg]^{\frac1a} \bigg\|_{L^p(u; L^q(v))} = 0}$.
\end{list}
\end{theorem}

Besides, in view of Theorem \ref{thm:KRQB}, to achieve Theorem \ref{thm:ICPT}, one has to work in multiple mixed-norm spaces, which are generalized in the following way.

For any integer $N \ge 1$, let $\p = (p_1, \ldots, p_N)$ be an $N$-tuple, and $\w = (w_1, \ldots, w_N)$ be an $N$-tuple of weights $w_i$ on measure spaces $(\Sigma_k, \mu_k)$, $k = 1, \ldots, N$; then on the product space $(\bm{\Sigma}, \bm{\mu}) := \big(\prod_{k=1}^N \Sigma_k, \, \prod_{k=1}^N \mu_k \big)$, the weighted mixed-norm is defined by
\begin{align*}
\|f\|_{\L^{\p}(\w)} 
:= \bigg[\int_{\Sigma_1} \cdots \bigg(\int_{\Sigma_N} |f(x_1, \ldots, x_N)|^{p_N}  w_N(x_N)\, d\mu_N(x_N) \bigg)^{\frac{p_{N-1}}{p_N}} \cdots \, w_1(x_1) \, d\mu_1(x_1)\bigg]^{\frac{1}{p_1}}.
\end{align*}

A further crucial ingredient in the proof of Theorem \ref{thm:ICPT} is a multilinear interpolation of boundedness on multiple mixed-norm spaces below.

\begin{theorem}\label{thm:IBDD}
Let $\p_0, \q_0 \in (0, \infty)^{N_0}$, $\p_j, \q_j \in [1, \infty]^{N_j}$, and $(\bm{\Sigma}_j, \bm{\mu}_j) := \big(\prod_{k=1}^{N_j} \Sigma_{j, k}, \, \prod_{k=1}^{N_j} \mu_{j, k} \big)$ be a product space, $j = 1, \ldots, m$. Let $T: \mathbf{S}_1 \times \cdots \times  \mathbf{S}_m \rightarrow \mathbf{M}_0$ be an $m$-linear operator, where $\mathbf{S}_j$ is the collection of all simple functions on $(\bm{\Sigma}_j, \bm{\mu}_j)$, and $\mathbf{M}_0$ is the set of all measurable functions on $(\bm{\Sigma}_0, \bm{\mu}_0)$. Assume that 
\begin{align}
\label{eq:IBDD-1} &T \text{ is bounded from $\L^{\p_1}(\u_1^{\p_1}) \times \cdots \times \L^{\p_m}(\u_m^{\p_m})$ to $\L^{\p_0}(\u_0^{\p_0})$ with constant $M_0$}, 
\\
\label{eq:IBDD-2} &T \text{ is bounded from $\L^{\q_1}(\v_1^{\q_1}) \times \cdots \times \L^{\q_m}(\v_m^{\q_m})$ to $\L^{\q_0}(\v_0^{\q_0})$ with constant $M_1$}. 
\end{align}
Then, $T$ can be extended as a bounded operator from $\L^{\r_1}(\w_1^{\r_1}) \times \cdots \times \L^{\r_m}(\w_m^{\r_m})$ to $\L^{\r_0}(\w_0^{\r_0})$ with constant $M_0^{1-\theta} M_1^{\theta}$, for all exponents satisfying
\begin{align}\label{eq:pqr}
\theta \in (0, 1),\quad & \frac{1}{\r_j} = \frac{1 - \theta}{\p_j} + \frac{\theta}{\q_j},
\quad \text{and} \quad \w_j = \u_j^{1 - \theta} \v_j^{\theta}, \quad j = 0, \ldots, m.
\end{align}
\end{theorem}

\subsection{Applications of Theorem \ref{thm:RdF-cpt}}
Let us present several applications of Theorem \ref{thm:RdF-cpt}, which clearly illustrate the effectiveness and flexibility of extrapolation theorem. 

Our first application is multilinear bi-parameter singular integral operators (cf. \cite{CY, LMV20}), for which $T1$ type assumptions are given in \cite{CY}.

\begin{theorem}\label{thm:CZO}
Let $T$ be an $m$-linear bi-parameter singular integral operator. Assume that $T$ satisfies the following hypotheses: 
\begin{list}{\rm (\theenumi)}{\usecounter{enumi}\leftmargin=1.3cm \labelwidth=1cm \itemsep=0.1cm \topsep=0.2cm \renewcommand{\theenumi}{H\arabic{enumi}}}

\item\label{H1} $T$ admits the compact full kernel representation,

\item\label{H2} $T$ admits the compact partial kernel representation,

\item\label{H3} $T$ satisfies the weak compactness property,

\item\label{H4} $T$ satisfies the diagonal $\CMO$ condition,

\item\label{H5} $T$ satisfies the product $\CMO$ condition.
\end{list}
Then $T$ is compact from $L^{p_1}(u_1^{p_1}; L^{q_1}(v_1^{q_1})) \times \cdots \times L^{p_m}(u_m^{p_m}; L^{q_m}(v_m^{q_m}))$ to $L^p(u^p; L^q(v^q))$ for all $\vec{p} = (p_1, \ldots, p_m) \in (1, \infty]^m$, for all $\vec{q} = (q_1, \ldots, q_m) \in (1, \infty]^m$, for all $\vec{u} = (u_1, \ldots, u_m) \in A_{\vec{p}}(\R^{n_1})$, and for all $\vec{v} = (v_1, \ldots, v_m) \in A_{\vec{q}}(\R^{n_2})$, where $\frac1p = \sum_{j=1}^m \frac{1}{p_j} > 0$, $\frac1q = \sum_{j=1}^m \frac{1}{q_j} > 0$, $u = \prod_{j=1}^m u_j$, and $v = \prod_{j=1}^m v_j$.  
\end{theorem}

The second application is multilinear bi-parameter dyadic paraproducts, which play a crucial role in dyadic analysis (cf. \cite{CY, LMV21}). Given a dyadic grid $\D=\D_1 \times \D_2$, a {\tt compact $m$-linear bi-parameter dyadic paraproduct} takes the form   
\begin{align*}
\Pi_{\D}(\vec{f}) 
&:= \sum_{I = I_1 \times I_2 \in \D} a_I \, 
\prod_{j=1}^m \langle f_j, \overline{h}_{j, I_1} \otimes \overline{h}_{j, I_2} \rangle 
\, \overline{h}_{m+1, I_1} \otimes \overline{h}_{m+1, I_2}, 
\end{align*}
where there exist $j_1^0, j_2^0 \in \{1, \ldots, m+1\}$ so that $\overline{h}_{j_1^0, I_1} = h_{I_1}$, $\overline{h}_{j_2^0, I_2} = h_{I_2}$, $\overline{h}_{j, I_1} = \frac{\mathbf{1}_{I_1}}{|I_1|}$ for every $j \ne j_1^0$, and $\overline{h}_{j, I_2} = \frac{\mathbf{1}_{I_2}}{|I_2|}$ for every $j \ne j_2^0$. Moreover, the coefficients $a_I$ satisfy 
\begin{align*}
\sup_U \frac{1}{|U|} \sum_{I \in \D: \, I \subset U} |a_I|^2 \le 1
\quad \text{and} \quad 
\lim_{N \to \infty} \sup_U \frac{1}{|U|} 
\sum_{I \notin \D(N): \, I \subset U} |a_I|^2 
= 0, 
\end{align*}
where the supremum $\sup_U$ is taken over all open sets $U \subset \Rnn$ with $|U| < \infty$.

\begin{theorem}\label{thm:para}
Let $\Pi_{\D}$ be a compact $m$-linear bi-parameter dyadic paraproduct. Then $T$ is compact from $L^{p_1}(u_1^{p_1}; L^{q_1}(v_1^{q_1})) \times \cdots \times L^{p_m}(u_m^{p_m}; L^{q_m}(v_m^{q_m}))$ to $L^p(u^p; L^q(v^q))$ for all $\vec{p} = (p_1, \ldots, p_m) \in (1, \infty]^m$, for all $\vec{q} = (q_1, \ldots, q_m) \in (1, \infty]^m$, for all $\vec{u} = (u_1, \ldots, u_m) \in A_{\vec{p}}(\R^{n_1})$, and for all $\vec{v} = (v_1, \ldots, v_m) \in A_{\vec{q}}(\R^{n_2})$, where $\frac1p = \sum_{j=1}^m \frac{1}{p_j} > 0$, $\frac1q = \sum_{j=1}^m \frac{1}{q_j} > 0$, $u = \prod_{j=1}^m u_j$, and $v = \prod_{j=1}^m v_j$.  
\end{theorem}

Let us proceed to the third application, continuous paraproducts, which arise naturally in the study of $T1$ theorems, such as the linear case \cite{DJ}, the bilinear case \cite{CLSY, Har}, and the bi-parameter setting \cite{Jou}. For each $i=1, 2$, let $\varphi^{(i)}, \psi^{(i)} \in \mathscr{C}_c^{\infty}(\R^{n_i})$ be radial functions such that $\supp \varphi^{(i)} \subset B(0, 1)$, $\int_{\R^{n_i}} \varphi^{(i)} \, dx_i = 1$, $\supp \psi^{(i)} \subset B(0, 1)$, $\int_{\R^{n_i}} \psi^{(i)} \, dx_i = 0$, and $\int_0^{\infty} |\widehat{\psi^{(i)}}(t e_1^{(i)})|^2 \frac{dt}{t} =1$, where $e_1^{(i)} := (1, 0, \ldots, 0)$ is a unit vector in $\R^{n_i}$. For any $t > 0$ and $f \in L_{\loc}^1(\R^{n_i})$, define convolution operators 
\begin{align*}
P_t^{(i)} f := \varphi_t^{(i)} *f 
\quad\text{ and }\quad 
Q_t^{(i)} f := \psi_t^{(i)} * f, 
\end{align*} 
where $\varphi_t^{(i)}(x_i) := t^{-n_i} \varphi^{(i)}(t^{-1} x_i)$ and $\psi_t^{(i)}(x_i) := t^{-n_i} \psi^{(i)}(t^{-1} x_i)$. Then set 
\begin{align*}
P_{t_1, t_2} := P_{t_1}^{(1)} \otimes P_{t_2}^{(2)} 
\quad \text{and} \quad 
Q_{t_1, t_2} := Q_{t_1}^{(1)} \otimes Q_{t_2}^{(2)}.
\end{align*}
Given $b \in \BMO(\Rnn)$, the {\tt bilinear bi-parameter paraproduct} $\pi_b$ is defined by 
\begin{align*}
\pi_b (f, g) 
:= \int_0^{\infty} \int_0^{\infty} Q_{t_1, t_2} 
\big( (Q_{t_1, t_2} b) (P_{t_1, t_2} f) (P_{t_1, t_2} g) \big) \, \frac{dt_1}{t_1} \frac{dt_2}{t_2}. 
\end{align*}

\begin{theorem}\label{thm:Pib}
The following statements hold: 
\begin{list}{\rm (\theenumi)}{\usecounter{enumi}\leftmargin=1.3cm \labelwidth=1cm \itemsep=0.1cm \topsep=0.2cm \renewcommand{\theenumi}{\arabic{enumi}}}

\item\label{Pib-1} For any $b \in \BMO(\Rnn)$, $\pi_b$ is bounded from $L^{r_1}(w_1^{r_1}) \times L^{r_2}(w_2^{r_2})$ to $L^r(w^r)$ for all $r_1, r_2 \in (1, \infty]$ and $(w_1, w_2) \in A_{(r_1, r_2)}(\Rnn)$, where $\frac1r = \frac{1}{r_1} + \frac{1}{r_2} > 0$ and $w = w_1 w_2$. 

\item\label{Pib-2} For any $b \in \CMO(\Rnn)$, $\pi_b$ is compact from $L^{r_1}(w_1^{r_1}) \times L^{r_2}(w_2^{r_2})$ to $L^r(w^r)$ for all $r_1, r_2 \in (1, \infty]$ and $(w_1, w_2) \in A_{(r_1, r_2)}(\Rnn)$, where $\frac1r = \frac{1}{r_1} + \frac{1}{r_2} > 0$ and $w = w_1 w_2$.  

\item\label{Pib-3} For any $b \in \CMO(\Rnn)$, $\pi_b$ is compact from $L^{p_1}(u_1^{p_1}; L^{q_1}(v_1^{q_1})) \times L^{p_2}(u_2^{p_2}; L^{q_2}(v_2^{q_2}))$ to $L^p(u^p; L^q(v^q))$ for all $p_1, p_2, q_1, q_2 \in (1, \infty]$, for all $(u_1, u_2) \in A_{(p_1, p_2)}(\R^{n_1})$, and for all $(v_1, v_2) \in A_{(q_1, q_2)}(\R^{n_2})$, where $\frac1p = \frac{1}{p_1} + \frac{1}{p_2} > 0$, $\frac1q = \frac{1}{q_1} + \frac{1}{q_2} > 0$, $u = u_1 u_2$, and $v = v_1 v_2$.  
\end{list}
\end{theorem}

Our last application is pseudo-differential operators, which have attracted considerable attention in harmonic analysis (cf. \cite{BMNT, DJ, DDZ, DGZ, Jou}). The {\tt bilinear bi-parameter pseudo-differential operator} $T_{\sigma}$ is defined by
\begin{align*}
T_{\sigma}(f, g)(x) := \int_{\Rnn} \int_{\Rnn} \sigma(x, \xi, \eta)
e^{2\pi i x \cdot (\xi + \eta)} \widehat{f}(\xi) \widehat{g}(\eta) \, d\xi \, d\eta,
\end{align*}
for all $x = (x_1, x_2) \in \Rnn$ and $f, g \in \S(\Rnn)$, where $\widehat{f}(\xi) := \int_{\Rnn} e^{-2 \pi i x \cdot \xi} f(x) \,  dx$.

Given $m = (m_1, m_2) \in \R^2$, $\rho = (\rho_1, \rho_2) \in [0, 1]^2$, and $\delta = (\delta_1, \delta_2) \in [0, 1]^2$, we say that $\sigma \in \mathcal{K}_{\rho,\delta}^m(\Rnn)$ if it satisfies 
\begin{align*}
\big|\partial_{x_1}^{\alpha_1} \partial_{x_2}^{\alpha_2} 
\partial_{\xi_1}^{\beta_1} \partial_{\xi_2}^{\beta_2}
\partial_{\eta_1}^{\gamma_1} \partial_{\eta_2}^{\gamma_2} 
\sigma(x, \xi, \eta) \big|
\leq C_{\alpha, \beta, \gamma}(x, \xi, \eta) 
\prod_{i=1}^2 (1+ |\xi_i| + |\eta_i|)^{m_i + \delta_i |\alpha_i| - \rho_i (|\beta_i| + |\gamma_i|)}, 
\end{align*}
for all multi-indices $\alpha$, $\beta$ and $\gamma$, where $C_{\alpha, \beta, \gamma}$ is a bounded function satisfying 
\begin{align*}
\lim_{|x| + |\xi| + |\eta| \to \infty} C_{\alpha,\beta, \gamma}(x, \xi, \eta) = 0.
\end{align*}
We say that $\sigma \in \mathcal{S}_{\rho,\delta}^m(\Rnn)$ if the above function $C_{\alpha,\beta, \gamma}$  depends only on $\alpha$, $\beta$, and $\gamma$.

\begin{theorem}\label{thm:Tsig}
Let $m = (0, 0)$, $\rho = (1, 1)$, and $\delta = (\delta_1, \delta_2) \in [0, 1)^2$. Then the following hold: 
\begin{list}{\rm (\theenumi)}{\usecounter{enumi}\leftmargin=1.3cm \labelwidth=1cm \itemsep=0.1cm \topsep=0.2cm \renewcommand{\theenumi}{\arabic{enumi}}}

\item\label{Tsig-1} For any $\sigma \in \mathcal{S}_{\rho, \delta}^m(\Rnn)$, $T_{\sigma}$ is bounded from $L^{r_1}(w_1^{r_1}) \times L^{r_2}(w_2^{r_2})$ to $L^r(w^r)$ for all $r_1, r_2 \in (1, \infty]$ and $(w_1, w_2) \in A_{(r_1, r_2)}(\Rnn)$, where $\frac1r = \frac{1}{r_1} + \frac{1}{r_2} > 0$ and $w = w_1 w_2$. 

\item\label{Tsig-2} For any $\sigma \in \mathcal{K}_{\rho, \delta}^m(\Rnn)$, $T_{\sigma}$ is compact from $L^{r_1}(w_1^{r_1}) \times L^{r_2}(w_2^{r_2})$ to $L^r(w^r)$ for all $r_1, r_2 \in (1, \infty]$ and $(w_1, w_2) \in A_{(r_1, r_2)}(\Rnn)$, where $\frac1r = \frac{1}{r_1} + \frac{1}{r_2} > 0$ and $w = w_1 w_2$.  

\item\label{Tsig-3} For any $\sigma \in \mathcal{K}_{\rho, \delta}^m(\Rnn)$, $T_{\sigma}$ is compact from $L^{p_1}(u_1^{p_1}; L^{q_1}(v_1^{q_1})) \times L^{p_2}(u_2^{p_2}; L^{q_2}(v_2^{q_2}))$ to $L^p(u^p; L^q(v^q))$ for all $p_1, p_2, q_1, q_2 \in (1, \infty]$, for all $(u_1, u_2) \in A_{(p_1, p_2)}(\R^{n_1})$, and for all $(v_1, v_2) \in A_{(q_1, q_2)}(\R^{n_2})$, where $\frac1p = \frac{1}{p_1} + \frac{1}{p_2} > 0$, $\frac1q = \frac{1}{q_1} + \frac{1}{q_2} > 0$, $u = u_1 u_2$, and $v = v_1 v_2$.  
\end{list}
\end{theorem}

\subsection{Historical background}
The classical Rubio de Francia extrapolation theorem states that if an operator is bounded on $L^{p_0}(w)$ for all weights $w \in A_{p_0}$ at some fixed exponent $p_0 \in [1, \infty)$, then it is bounded on $L^p(w)$ for every $p \in (1, \infty)$ and every $w \in A_p$. This result not only perfectly complemented the weighted theory and the Calder\'{o}n--Zygmund theory, but also immediately simplified and refined weighted estimates for Calder\'{o}n--Zygmund operators \cite{CMP12, DGPP}, and Littlewood--Paley square functions \cite{CMP12, Ler}, vector-valued inequalities \cite{CMP}, the endpoint weak-type inequality \cite{CMP05}, and the weak Muckenhoupt--Wheeden conjecture \cite{LOP08, LOP09}. Further extensions included off-diagonal $L^p$--$L^q$ extrapolation for weights $A_{p, q} \, (p \ne q)$ \cite{HMS}, limited range extrapolation for weights $w \in A_{p/p_-}\cap RH_{p_+/p}$ with $p \in (p_-, p_+)$ \cite{AM}, and two-weight extrapolation \cite{CMP, CP}. These foundations established extrapolation as a core technique, relying on maximal operator boundedness, duality, and the reverse factorization property of $A_p$ weights as key ingredients. 

The extrapolation theory entered a new era with its extension to the multilinear setting. Early multilinear versions treated products weights $A _{p_1} \times \cdots \times A _{p_m}$ (cf. \cite{GM}), but a genuine breakthrough came with the introduction of the multilinear Muckenhoupt classes $A_{\vec{p}}$
\begin{align*}
[\vec{w}]_{A_{\vec{p}}} 
:= \sup_{\text{cube } Q \subset \Rn} 
\langle w^p \rangle_Q^{\frac1p} \prod^m_{j=1} \langle w_j^{-p_j'} \rangle_Q^{\frac{1}{p_j'}}
< \infty,
\end{align*} 
where $w=\prod^m_{j=1}w_j$ and $\frac1p = \sum_{j=1}^m \frac{1}{p_j}$ with $p_1, \ldots, p_m \in [1, \infty]$, which encode the inherent multi-variable coupling of multilinear Calder\'{o}n--Zygmund operators and are strictly larger than product weights (cf. \cite{LOPTT}). Li, Martell, and Ombrosi \cite{LMO} resolved a decade-long open problem by establishing extrapolation directly for the classes $A _{\vec{p}, \vec{r}}$, which accommodate component weights that may be non-locally integrable individually but well-behaved collectively. Beyond the classes $A_{\vec{p}}$, the general classes $A_{\vec{p}, \vec{r}}$ naturally appear in weighted estimates for bilinear Hilbert transforms and operators with restricted boundedness ranges. These advances unified scalar and vector-valued inequalities, extended results into the quasi-Banach range $(0<p<1)$, and provided a unified framework for bilinear Hilbert transforms, bilinear rough singular integrals, and commutators.

Parallel to multilinear developments, extrapolation has been vastly generalized to abstract function spaces, including variable Lebesgue spaces \cite{CW}, Orlicz spaces \cite{CH}, and Banach function spaces \cite{CMM}. Instead of relying on explicit weight conditions, Cao, Mar\'{i}n, and Martell \cite{CMM} used modern formulations to link extrapolation to weighted boundedness of the Hardy--Littlewood maximal operator and its dual, enabling applications to non-doubling measures, spaces of homogeneous type, and metric measure spaces beyond Euclidean spaces. Remarkably, extrapolation built in \cite{CMM} can be used to established the well-posedness of the Dirichlet problem for elliptic equations on the upper-half spaces. Moreover, based on the general framework in \cite{CMM}, endpoint extrapolation for $A_1$ and $A_{\infty}$ weights has been obtained in \cite{CO} and further yielded sharp local exponential decay estimates and mixed weak-type inequalities, which refined classical good-$\lambda$ inequalities and endpoint estimates.

A substantial advance is the emergence of compactness extrapolation, which elevates the theory from boundedness to compactness. Hyt\"{o}nen and Lappas \cite{HL} first established a compact version of extrapolation by showing that compactness at one weighted $L^{p_0}(w_0)$ propagates to all $p \in (1, \infty)$ and $w \in A_p$. It was swiftly extended to the multilinear setting by Cao, Olivo, and Yabuta \cite{COY}, who developed compact extrapolation for $A_{\vec{p}, \vec{r}}$ weights and limited range scales by means of weighted Kolmogorov--Riesz theorems to characterize precompactness in quasi-Banach spaces. Soon after, this methodology was extended by Cao et al. \cite{CIRXY} to spaces of homogeneous type in order to show weighted compactness of Banach-valued multilinear bounded oscillation operators and/or their commutators.

A prominent application of the compact extrapolation theory lies in the study of compact versions of $T1$ theorems. The classical $T1$ theorem \cite{DJ} provides necessary and sufficient conditions for the $L^2$-boundedness of singular integral operators with a standard Calder\'{o}n–Zygmund kernel. Subsequently, it was extended to the bi-parameter setting \cite{Jou} and the multilinear case \cite{CJ, GT, Har}. A compact version of the $T1$ theorem for Calder\'{o}n--Zygmund operators was first formulated by \cite{Vil}. However, the $T1$ theorem to deduce compactness of multilinear Calder\'{o}n--Zygmund operators has been a long-standing open problem. Inspired by the approach to compactness in the bi-parameter scenario \cite{CYY} and the dyadic representation \cite{LMOV}, by means of extrapolation of compactness, Cao, Liu, Si, and Yabuta \cite{CLSY} revealed the equivalence among compact $T1$ type assumptions, (weighted) compactness of bilinear Calder\'{o}n--Zygmund operators, and a compact dyadic representation of a bilinear Calder\'{o}n--Zygmund operator as an average of some compact bilinear dyadic shifts and paraproducts. This work not only exposed the precise dyadic structure behind compactness, but also established extrapolation of endpoint compactness in order to achieve endpoint compactness covering $L^1 \times L^1 \to L^{1/2, \infty}$ and $L^{\infty} \times L^{\infty} \to \CMO$. Building upon these works, Cao and Yabuta \cite{CY} further investigated the weighted compactness of multilinear Calder\'{o}n--Zygmund operators on product spaces, and then Cao et al. \cite{CCLLYZ} extended the relevant compactness theory to singular integrals associated with Zygmund dilations. These studies fully demonstrate the central role of the extrapolation theory of compactness, and the work presented in this paper is a natural extension of the aforementioned research.

The rest of the paper is organized as follows. In Section \ref{sec:KR}, we present the proof of Theorem \ref{thm:KRQB}, which is based on  a weighted Kolmogorov--Riesz theorem in the Banach range (cf. Theorem \ref{thm:pq1infty}). Section \ref{sec:IBMS} is devoted to showing the interpolation of boundedness on mixed-norm spaces (cf. Theorem \ref{thm:IBDD}). Then in Section \ref{sec:INEX}, by means of Theorems \ref{thm:KRQB} and \ref{thm:IBDD}, we first establish the interpolation and extrapolation of compactness on mixed-norm spaces (cf. Theorems \ref{thm:ICPT} and \ref{thm:RdF-cpt}), then characterize the weights class $A_{\vec{p}}(\R^{n_1}) \times A_{\vec{q}}(\R^{n_2})$ via the boundedness of multilinear strong maximal operators and multilinear bi-parameter Riesz transforms on weighted mixed-norm spaces (cf. Theorem \ref{thm:ApAq}). Finally, Section \ref{sec:app} includes some applications of extrapolation (cf. Theorems \ref{thm:CZO}--\ref{thm:Tsig}).

\section{Kolmogorov--Riesz theorems on mixed-norm spaces}\label{sec:KR}
This section focuses on compactness criterions on mixed-norm spaces. 

\subsection{Kolmogorov--Riesz theorems in the Banach range}
\begin{theorem}\label{thm:pq1infty}
Let $p, q \in (1, \infty)$,  $u \in A_{p}(\R^{n_1})$,   $v \in A_{q}(\R^{n_2})$, and $\mathcal{K} \subset L^p(u; L^q(v))$. Then $\mathcal{K}$ is precompact in $L^p(u; L^q(v))$ if and only if the following are satisfied: 
\begin{list}{\textup{(\theenumi)}}{\usecounter{enumi}\leftmargin=1cm \labelwidth=1cm \itemsep=0.2cm 
			\topsep=0.2cm \renewcommand{\theenumi}{\arabic{enumi}}}
			
\item\label{list:pq1infty-1} $\sup\limits_{f \in \mathcal{K}} \|f\|_{L^p(u; L^q(v))} < \infty$,

\item\label{list:pq1infty-2} $\lim\limits_{A \to \infty} \sup\limits_{f \in \mathcal{K}}\|f \mathbf{1}_{{B(0, A)}^c}\|_{L^p(u; L^q(v))} = 0$,

\item\label{list:pq1infty-3} $\lim\limits_{r \to 0}  \sup\limits_{f \in \mathcal{K}}\|f-{\langle f \rangle}_{B(\cdot, r)}\|_{L^p(u; L^q(v))} = 0$.
\end{list}

\end{theorem}

\begin{proof}
We begin with proving the necessity. Assume that $\mathcal{K}$ is precompact in $ L^{p}(u;  L^{q}(v))$, which implies that $\mathcal{K}$ is totally bounded. Given $\varepsilon > 0$, denote 
\begin{align*}
B_\varepsilon(f) 
:= \{g \in L^{p}(u;  L^{q}(v)) : \|g - f\|_{L^{p}(u;  L^{q}(v))} < \varepsilon \}, 
\end{align*}
then there exists a finite number of functions $ \{f_k\}_{k=1}^N \subset  \mathcal{K}$ such that $\mathcal{K} \subseteq \bigcup_{k=1}^{N} B_{\varepsilon}(f_k)$. Therefore, for any $f \in \mathcal{K}$, there exists some $k \in \{1, \dots, N\}$ such that 
\begin{align}\label{eq:ffk}
 \|f - f_k\|_{L^{p}(u;  L^{q}(v))} < \varepsilon.
\end{align}
Hence, 
\begin{align*}
\|f\|_{L^{p}(u;  L^{q}(v))} 
\le \|f-f_k\|_{L^{p}(u;  L^{q}(v))} + \|f_k\|_{L^{p}(u;  L^{q}(v))}
\le \varepsilon + \max\limits_{1 \le k \le N} \|f_k\|_{L^{p}(u;  L^{q}(v))},
\end{align*}
which justifies the condition \eqref{list:pq1infty-1}. Note that 
\begin{align}\label{eq:uAp}
u \in A_{p}(\R^{n_1}) \Longrightarrow u, u^{1-p'} \in L^1_{\loc}(\R^{n_1}),
\end{align}
and
\begin{align}\label{eq:vAq}
v \in A_{q}(\R^{n_2}) \Longrightarrow v, v^{1-q'} \in L^1_{\loc}(\R^{n_2}).
\end{align}
By Lemma \ref{lem:dense}, $\mathscr{C}^\infty_c(\R^{n_1} \times \R^{n_2})$ is dense in $L^{p}(u;  L^{q}(v))$. Thus, there exists $g_k \in \mathscr{C}^\infty_c(\R^{n_1} \times \R^{n_2})$ such that 
\begin{align*}
\|f_k-g_k\|_{L^{p}(u;  L^{q}(v))}  \le \varepsilon, \quad 1 \le k \le N,
\end{align*}
which along with \eqref{eq:ffk} gives
 \begin{align}\label{eq:fgk}
\|f-g_k\|_{L^{p}(u;  L^{q}(v))} 
\le \|f-f_k\|_{L^{p}(u;  L^{q}(v))} + \|f_k-g_k\|_{L^{p}(u;  L^{q}(v))} 
\lesssim \varepsilon. 
\end{align}
Suppose $\supp{g_k} \subset B(0, A_k) $ for each $k=1,\dots ,N$, and $A_0:=\max{\{A_1, \dots ,A_N\}}$. Then for any $A \ge A_0$, the estimates \eqref{eq:ffk} and \eqref{eq:fgk} lead to
\begin{align*}
\|f \mathbf{1}_{{B(0, A)}^c}\|_{L^{p}(u;  L^{q}(v))} 
\le \|f-g_k\|_{L^{p}(u;  L^{q}(v))} + \|g_k \mathbf{1}_{{B(0, A)}^c}\|_{L^{p}(u;  L^{q}(v))} 
\lesssim \varepsilon.  
\end{align*}
This shows the condition \eqref{list:pq1infty-2}. To proceed, we split 
 \begin{align}\label{eq:ffB}
\|f-{\langle f \rangle}_{B(\cdot, r)}\|_{L^{p}(u;  L^{q}(v))} 
&\le \|f-g_k\|_{L^{p}(u;  L^{q}(v))} 
+ \|g_k-{\langle g_k \rangle}_{B(\cdot, r)}\|_{L^{p}(u;  L^{q}(v))} 
\\ \nonumber 
&\quad+ \|{\langle g_k \rangle}_{B(\cdot, r)} - {\langle f \rangle}_{B(\cdot, r)}\|_{L^{p}(u;  L^{q}(v))}. 
\end{align}
Denote $B_i(r) := \{x_i \in \R^{n_i} : |x_i| < r \}$, $i=1, 2$. Since $g_k \in \mathscr{C}^\infty_c(\R^{n_1} \times \R^{n_2})$, there exists some $r_0 > 0 $ so that
 \begin{align*}
\sup\limits_{|x-y| < r_0}|g_k(x)-g_k(y)| 
\le \varepsilon \, u(B_1(2A_0))^{-\frac1p} \, v(B_2(2A_0))^{-\frac1q},
\end{align*}
which yields that for any $0 < r < \min\{r_0, A_0\}$,
 \begin{align}\label{eq:gkgkB}
&\|g_k - \langle g_k \rangle_{B(\cdot, r)}\|_{L^p(u;  L^q(v))} 
\\ \nonumber 
&\le \bigg\| \mathbf{1}_{B_1(2A_0)} \otimes \mathbf{1}_{B_2(2A_0)}
\fint_{B(0, r)} |g_k(\cdot) - g_k(\cdot +y)| \, dy \bigg\|_{L^{p}(u;  L^{q}(v))}  
\le \varepsilon. 
\end{align}
In addition,
 \begin{align}\label{gkbx}
|{\langle g_k \rangle}_{B(x, r)}-{\langle f \rangle}_{B(x, r)}| 
\le \fint_{B(x, r)} |f-g_k| \,dy \le M_{\mathcal{R}}(f - g_k)(x), 
\end{align}
and 
\begin{align}\label{eq:Mh}
\|M_{\mathcal{R}}h\|_{L^s(\mu; L^t(\nu))} \lesssim \|h\|_{L^s(\mu; L^t(\nu))}, 
\end{align}
for all $s, t \in (1, \infty)$ and $(\mu, \nu) \in A_s(\R^{n_1}) \times A_t(\R^{n_2})$, where the latter was shown in  \cite[Theorem 1 and Lemma 3]{Kur}. Collecting \eqref{eq:fgk}, \eqref{gkbx}, and \eqref{eq:Mh}, we obtain 
\begin{align}\label{eq:gkBfB}
\|{\langle g_k \rangle}_{B(\cdot, r)}-{\langle f \rangle}_{B(\cdot, r)}\|_{L^{p}(u;  L^{q}(v))} 
\lesssim \|f - g_k \|_{L^{p}(u;  L^{q}(v))}
\lesssim \varepsilon.
\end{align}
As a consequence of \eqref{eq:ffk}, \eqref{eq:ffB}, \eqref{eq:gkgkB}, and \eqref{eq:gkBfB}, there holds for any $r \in (0, r_0)$, 
 \begin{align*}
\|f-{\langle f \rangle}_{B(\cdot, r)}\|_{L^{p}(u;  L^{q}(v))} \lesssim \varepsilon, 
\quad \text{ uniformly in  } f \in \mathcal{K}. 
\end{align*}
This verifies the condition \eqref{list:pq1infty-3}.

To show the sufficiency, assume the conditions  \eqref{list:pq1infty-1}, \eqref{list:pq1infty-2}, and \eqref{list:pq1infty-3} hold. Given $\varepsilon >0$, the conditions \eqref{list:pq1infty-2} and \eqref{list:pq1infty-3} imply that there exist $A >0$ and $r>0$ such that 
\begin{align}\label{eq:bc}
\|f \mathbf{1}_{{B(0, A)}^c}\|_{L^{p}(u;  L^{q}(v))} \le \varepsilon 
\quad\text{and}\quad 
\|f - {\langle f \rangle}_{B(\cdot, r)}\|_{L^{p}(u;  L^{q}(v))} \le \varepsilon, 
\quad \forall f \in \mathcal{K}. 
\end{align}
For any $x \in \overline{B(0, A)}$, note that 
\begin{align}\label{eq:BB0}
B := B(x, r) \subset B(0, 2A+r) =: B_0 \subset B_0^1 \times B_0^2, 
\end{align}
where $B_0^i = \{y_i \in \R^{n_i}: |y_i| < 2A+r\}$. By H\"older's inequality, 
\begin{align}\label{eq:fB}
|{\langle f \rangle}_{B(x, r)}| 
&\le \frac{1}{|B(0, r)|}  \int_{{\R^{n_1}} \times {\R^{n_2}} } |f \, (u^{\frac1p} \otimes v^{\frac1q})| 
| \mathbf{1}_{B}(u^{-\frac{1}{p}} \otimes v^{-\frac{1}{q}})| \,dy
\\ \nonumber 
&\lesssim r^{-n_1-n_2}\|f(u^{\frac{1}{p}} \otimes v^{\frac{1}{q}})\|_{L^{p}( L^{q})} 
\|(\mathbf{1}_{B_0^1} \otimes \mathbf{1}_{B_0^2})(u^{-\frac{1}{p}} \otimes v^{-\frac{1}{q}})\|_{L^{p'}( L^{q'})}
\\ \nonumber 
&=r^{-n_1-n_2} \|f\|_{L^{p}(u;  L^{q}(v))} 
\big[u^{1-p'}(B_0^1) \big]^{\frac{1}{p'}} \big[v^{1-q'}(B_0^2) \big]^{\frac{1}{q'}}. 
\end{align}
By the fact $u \in A_p(\R^{n_1})$ and $v \in A_q(\R^{n_2})$, there exist some $\kappa_1 \in (1, p)$ and $\kappa_2 \in (1, q)$ so that $u \in  A_{\kappa_1}(\R^{n_1})$ and $v \in A_{\kappa_2}(\R^{n_2})$. Thus, by \eqref{eq:uAp} and \eqref{eq:vAq},
\begin{align}\label{eq:uvkappa}
u^{1-\kappa'_1} \in L^1_{\loc}(\R^{n_1}) \quad \text{and} \quad v^{1-\kappa'_2} \in L^1_{\loc}(\R^{n_2}). 
\end{align}
Set $p_0 = \frac{p}{\kappa_1} > 1$ and $q_0 = \frac{q}{\kappa_2} > 1$, and then choose $p_1$ and $q_1$ such that $ \frac{1}{p} + \frac{1}{p_1} = \frac{1}{p_0}$ and $ \frac{1}{q} + \frac{1}{q_1} = \frac{1}{q_0}$. For any $x,y \in \overline{B(0, A)}$, by H\"{o}lder's inequality and \eqref{eq:BB0}, 
\begin{align}\label{eq:fBxfBy}
& |\langle f \rangle_{B(x, r)} - \langle f \rangle_{B(y, r)}| 
\\ \nonumber 
&\le \frac{1}{|B(0, r)|}  \int_{{\R^{n_1}} \times {\R^{n_2}} } |f(z)| | \mathbf{1}_{B(x, r)}-\mathbf{1}_{B(y, r)}| \,dz
\\ \nonumber 
&\lesssim r^{-n_1-n_2}\|f \mathbf{1}_{B(x, r) \cup B(y, r)}\|_{L^{p_0}( L^{q_0})} 
\|\mathbf{1}_{B(x, r)}-\mathbf{1}_{B(y, r)}\|_{L^{p'_0}( L^{q'_0})}
\\ \nonumber 
&\le r^{-n_1-n_2} \|f \mathbf{1}_{B_0}\|_{L^{p_0}( L^{q_0})} 
\|\mathbf{1}_{B(x, r) \triangle B(y, r)} \|_{L^{p'_0}( L^{q'_0})}
\\ \nonumber 
&\le r^{-n_1-n_2} \|f (u^{\frac{1}{p}} \otimes v^{\frac{1}{q}}) \|_{L^p( L^q)}  
\| (u^{-\frac{1}{p}} \mathbf{1}_{B_0^1}) \otimes (v^{-\frac{1}{q}} \mathbf{1}_{B_0^2}) \|_{L^{p_1}( L^{q_1})} 
\|\mathbf{1}_{B(x, r) \triangle B(y, r)} \|_{L^{p'_0}( L^{q'_0})}
\\ \nonumber 
&= r^{-n_1-n_2} \|f\|_{L^{p}(u;  L^{q}(v))} 
[u^{1-\kappa'_1}(B_0^1)]^{\frac{1}{p_1}}  
[v^{1-\kappa'_2}(B_0^2)]^{\frac{1}{q_1}}  
\|\mathbf{1}_{B(x, r) \triangle B(y, r)} \|_{L^{p'_0}( L^{q'_0})}. 
\end{align}
It follows from the condition \eqref{list:pq1infty-1}, \eqref{eq:uAp}, \eqref{eq:vAq}, and \eqref{eq:fB}--\eqref{eq:fBxfBy} that $\{{\langle f \rangle}_{B(\cdot, r)}\}_{f \in \mathcal{K}}$ is equi-bounded and equi-continous on the closed ball $\overline{B(0, A)}$. By the classical Ascoil-Arzel\`{a} theorem (cf. \cite[p. 85]{Yos}), it is precompact in $\mathscr{C}(\overline{B(0, A)})$. Hence, it is totally bounded in $\mathscr{C}(\overline{B(0, A)})$. Then there exists a finite number of functions $\{f_j\}_{j=1}^N \subset  \mathcal{K} $ such that 
\begin{align*}
\inf\limits_{1 \le j \le N} \sup\limits_{|x| \le A} |{\langle f \rangle}_{B(x, r)} - {\langle f_j \rangle}_{B(x, r)}| 
\le \varepsilon \, u(B_0^1)^{-\frac1p} v(B_0^2)^{-\frac1q}, 
\quad \forall f \in \mathcal{K},  
\end{align*}
which indicates that for each $f \in \mathcal{K}$ there exists $j \in \{1, \dots, N\}$ such that 
\begin{align}\label{eq:fBxfjBx}
\sup\limits_{|x| \le A} |{\langle f \rangle}_{B(x, r)} - {\langle f_j \rangle}_{B(x, r)}| 
\le \varepsilon \, u(B_0^1)^{-\frac1p} v(B_0^2)^{-\frac1q}.  
\end{align}
Therefore, we invoke \eqref{eq:bc} and \eqref{eq:fBxfjBx} to conclude
\begin{align*}
& \|f - f_j \|_{L^p(u;  L^q(v))} 
\\
&\le \|(f-f_j) \mathbf{1}_{B(0, A)}\|_{L^{p}(u;  L^{q}(v))} 
+ \|(f-f_j)\mathbf{1}_{B(0, A)^c}\|_{L^{p}(u;  L^{q}(v))}
\\
&\le \|(f-{\langle f \rangle}_{B(\cdot, r)}) \mathbf{1}_{B(0, A)}\|_{L^{p}(u;  L^{q}(v))} 
+ \|({\langle f \rangle}_{B(\cdot, r)}-{\langle f_j \rangle}_{B(\cdot, r)}) \mathbf{1}_{B(0, A)}\|_{L^{p}(u;  L^{q}(v))} 
\\
&\quad+ \|({\langle f_j \rangle}_{B(\cdot, r)} - f_j) \mathbf{1}_{B(0, A)}\|_{L^{p}(u;  L^{q}(v))} 
+ \|f \mathbf{1}_{B(0, A)^c}\|_{L^{p}(u;  L^{q}(v))} 
+ \|f_j \mathbf{1}_{B(0, A)^c}\|_{L^{p}(u;  L^{q}(v))}
\\
& \le \|f - \langle f \rangle_{B(\cdot, r)}\|_{L^{p}(u;  L^{q}(v))} 
+ \varepsilon \, [u(B_0^1)]^{-\frac{1}{p}} [v(B_0^2)]^{-\frac{1}{q}} 
\|\mathbf{1}_{B_0^1} \otimes \mathbf{1}_{B_0^2}\|_{L^{p}(u;  L^{q}(v))} 
\\
&\quad+ \|{\langle f_j \rangle}_{B(\cdot, r)}-f_j \|_{L^{p}(u;  L^{q}(v))} 
+ \|f \mathbf{1}_{B(0, A)^c}\|_{L^{p}(u;  L^{q}(v))} 
+ \|f_j \mathbf{1}_{B(0, A)^c}\|_{L^{p}(u;  L^{q}(v))}
\\
& \le 5\varepsilon.
\end{align*}
This shows that $\mathcal{K}$ is totally bounded. Hence, $\mathcal{K}$ is precompact in $L^{p}(u;  L^{q}(v))$. 
\end{proof}

\subsection{Kolmogorov--Riesz theorems in the quasi-Banach range}
We are going to present the proof of Theorem \ref{thm:KRQB}. To prove the necessity, assume that $\mathcal{K}$ is precompact in $ L^{p}(u;  L^{q}(v))$. In particular, $\mathcal{K}$ is totally bounded. Given $\varepsilon > 0$, there exists a finite number of functions $ \{f_j\}^N_{j=1}\subset  \mathcal{K}$ such that $\inf\limits_{1 \le j \le N}\|f-f_j\|_{L^{p}(u;  L^{q}(v))} \le \varepsilon$. Let $f \in \mathcal{K}$ be an arbitrary function, then there exists some $j \in \{1,\dots,N\}$ such that 
\begin{align}\label{eq:ffj}
\|f-f_j\|_{L^{p}(u;  L^{q}(v))} \le \varepsilon.  
\end{align}
This gives 
\begin{align*}
\|f\|_{L^p(u; L^{q}(v))} 
\lesssim \|f - f_j\|_{L^{p}(u; L^{q}(v))} + \|f_j\|_{L^{p}(u; L^{q}(v))} 
\le \varepsilon + \max\limits_{1 \le j \le N} \|f_j\|_{L^{p}(u; L^{q}(v))},
\end{align*}
which shows the condition \eqref{KRQB-1}. Since $\mathscr{C}^\infty_c(\R^{n_1} \times \R^{n_2})$ is dense in $L^{p}(u;  L^{q}(v))$ (cf. Lemma \ref{lem:dense}), for any $j \in \{1,\dots,N\}$, there exists $g_j \in \mathscr{C}^\infty_c(\R^{n_1} \times \R^{n_2})$ such that 
\begin{align*}
\|f_j - g_j\|_{L^{p}(u;  L^{q}(v))}  \le \varepsilon, 
\end{align*}
which together with \eqref{eq:ffj} implies 
 \begin{align}\label{eq:fgj}
\|f-g_j\|_{L^{p}(u;  L^{q}(v))} 
\lesssim \|f-f_j\|_{L^{p}(u;  L^{q}(v))} + \|f_j-g_j\|_{L^{p}(u;  L^{q}(v))} 
\lesssim \varepsilon. 
\end{align}
Set $\supp{g_j} \subset B(0, A_j) $ for each $j=1,\dots ,N$, and $A_0:=\max{\{A_1, \dots ,A_N\}}$. Then for any $A \ge A_0$, by \eqref{eq:ffj} and \eqref{eq:fgj}, 
\begin{align*}
\|f \mathbf{1}_{{B(0, A)}^c}\|_{L^{p}(u;  L^{q}(v))} 
\lesssim \|f-g_j\|_{L^{p}(u;  L^{q}(v))} + \|g_j \mathbf{1}_{{B(0, A)}^c}\|_{L^{p}(u;  L^{q}(v))} 
\lesssim \varepsilon,  
\end{align*}
which justifies the condition \eqref{KRQB-2}. Denote $B_i(2A_0) := \{x_i \in \R^{n_i}: |x_i|<2A_0 \}$, $i=1, 2$. By the condition $g_j \in \mathscr{C}^\infty_c(\R^{n_1} \times \R^{n_2})$, there exists some $r_0 > 0 $ so that
 \begin{align}\label{eq:gjxgjy}
 \sup\limits_{|x-y| < r_0}|g_j(x) - g_j(y)| 
 \le \varepsilon \, u(B_1(2A_0))^{-\frac{1}{p}} v(B_2(2A_0))^{-\frac{1}{q}}. 
\end{align}
To proceed, we split
 \begin{align}\label{eq:i1i2i3}
\mathcal{I}& 
:= \bigg\| \bigg( \fint_{B(0, r)} |\tau_y f-f |^a \,dy \bigg)^{\frac{1}{a}} \bigg\|_{L^{p}(u;  L^{q}(v))} 
\\ \nonumber 
&\lesssim \bigg\| \bigg( \fint_{B(0, r)} |\tau_y f-\tau_y g_j |^a \,dy \bigg)^{\frac{1}{a}} \bigg\|_{L^{p}(u;  L^{q}(v))} 
\\ \nonumber 
& \quad +  \bigg\| \bigg(\fint_{B(0, r)} |\tau_y g_j-g_j  |^a \,dy \bigg)^{\frac{1}{a}} \bigg\|_{L^{p}(u;  L^{q}(v))} 
\\ \nonumber 
& \quad + \bigg\| \bigg( \fint_{B(0, r)} | g_j-f |^a \,dy \bigg)^{\frac{1}{a}} \bigg\|_{L^{p}(u;  L^{q}(v))}
\\ \nonumber 
& =: \mathcal{I}_1 + \mathcal{I}_2 + \mathcal{I}_3. 
\end{align}
Since $u \in A_{p_0}(\R^{n_1}) \subset A_{\frac{p}{a}}(\R^{n_1})$ and $v \in A_{q_0}(\R^{n_2}) \subset A_{\frac{q}{a}}(\R^{n_2})$, the inequalities \eqref{eq:Mh} and  \eqref{eq:fgj} imply
 \begin{align}\label{eq:i1}
\mathcal{I}_1 
& \le \|M_{\mathcal{R}}(|f - g_j|^a)^{\frac{1}{a}}\|_{L^{p}(u;  L^{q}(v))} 
= \|M_{\mathcal{R}}(|f - g_j|^a)\|^{\frac{1}{a}}_{L^{\frac{p}{a}}(u;  L^{\frac{q}{a}}(v))} 
\\ \nonumber 
&\lesssim  \||f - g_j|^a\|^{\frac1a}_{L^{\frac{p}{a}}(u; L^{\frac{q}{a}}(v))} 
= \|f - g_j\|_{L^p(u; L^q(v))} 
\lesssim \varepsilon. 
\end{align}
For all $0<r<\min \{r_0, A_0\}$, using \eqref{eq:gjxgjy} and that $\supp ( \tau_y g_j-g_j) \subset B(0, 2A_0)$ for any $|y| < r$, we have 
 \begin{align}\label{eq:i2}
\mathcal{I}_2 \le \|\mathbf{1}_{B_1(0, 2A_0)} \otimes \mathbf{1}_{B_2(0, 2A_0)}\|_{L^{p}(u;  L^{q}(v))} \sup\limits_{|x-y|<r_0}|g_j(x)-g_j(y)| \le \varepsilon. 
\end{align}
Additionally, \eqref{eq:fgj} implies 
 \begin{align}\label{eq:i3}
\mathcal{I}_3 = \|g_j-f \|_{L^{p}(u;  L^{q}(v))}  \lesssim 2\varepsilon. 
\end{align}
Now gathering \eqref{eq:i1i2i3}--\eqref{eq:i3}, we conclude that 
 \begin{align*}
 \mathcal{I} \lesssim 5\varepsilon, \quad \text{ for all } 0<r<\min\{r_0, A_0\}. 
\end{align*}
This shows the condition \eqref{KRQB-3}.

To show the sufficiency, assume that the conditions  \eqref{KRQB-1}, \eqref{KRQB-2}, and \eqref{KRQB-3} hold. Before the proof, two assertions are given first. We claim that it suffices to treat the case that $\mathcal{K}$ is a family of non-negative functions. Denote
\begin{align*}
\mathcal{K}^+ :=\{f^+: f \in \mathcal{K} \} 
\quad\text{ and } \quad 
\mathcal{K}^- :=\{f^-: f \in \mathcal{K} \},  
\end{align*}
where $f^+ := (|f| + f)/2$ and $f^- := (|f| - f)/2$. Then for all $f, g \in \mathcal{K}$ and $x \in \Rnn$, 
\begin{align*}
0 \le f^+(x) \le |f(x)|, \quad |f^+(x) - g^+(x)| \le |f(x)-g(x)|,  
\\
0 \le f^-(x) \le |f(x)|, \quad |f^-(x)-g^-(x)| \le |f(x)-g(x)|, 
\\
\text{ and } |f(x)-g(x)| \le |f^+(x)-g^+(x)| + |f^-(x)-g^-(x)|,   
\end{align*}
which together with $(\tau_h f)^\pm = \tau_h f^\pm$ indicates that 
\begin{align*}
\text{\eqref{KRQB-1}--\eqref{KRQB-3} hold for $\mathcal{K}$
$\iff$ \eqref{KRQB-1}--\eqref{KRQB-3} hold for $\mathcal{K}^+$ and $\mathcal{K}^-$}, 
\end{align*}
and 
\begin{align*}
\mathcal{K} \text{ is precompact in } L^{p}(u;  L^{q}(v)) 
\iff \mathcal{K}^+ \text{ and } \mathcal{K}^- \text{ are precompact in } L^{p}(u;  L^{q}(v)). 
\end{align*}
Thus, we may assume that $\mathcal{K}$ is a family of non-negative functions. 

We also claim that 
\begin{align}\label{eq:KKa}
\mathcal{K} \text{ is precompact in } L^{p}(u;  L^{q}(v)) 
\iff \mathcal{K}^a \text{ is precompact in } L^{\frac{p}{a}}(u;  L^{\frac{q}{a}}(v)),
\end{align}
where $\mathcal{K}^a := \{f^a: f \in \mathcal{K}\}$. To show \eqref{eq:KKa}, we utilize an elementary calculation (cf. \cite{Tsu}): for any $\alpha \in(0, 1)$, 
\begin{align}\label{eq:alphast}
|s^\alpha - t^\alpha| \le |s - t|^\alpha 
\le \frac{1}{\alpha} \bigg(\frac{s+t}{|s - t|} \bigg)^{1 - \alpha}|s^\alpha - t^\alpha|, \quad \forall s, t>0.  
\end{align}
Therefore for all $f, g \in \mathcal{K}$, 
\begin{align}\label{eq:alphafg}
\|f^a - g^a\|_ {L^{\frac{p}{a}}(u;  L^{\frac{q}{a}}(v))} \le \| |f - g|^a \|_{L^{\frac{p}{a}}(u;  L^{\frac{q}{a}}(v))} = \|f - g\|^a_{L^p(u;  L^q(v))},  
\end{align}
which gives the left-to-right implication in \eqref{eq:KKa}. To prove the reverse, assume that $\mathcal{K}^a$ is precompact in $L^{\frac{p}{a}}(u;  L^{\frac{q}{a}}(v))$. Fix $\varepsilon >0$ and $f, g \in \mathcal{K}$, denote 
 \begin{align*}
E_\varepsilon := \bigg\{ x\in \Rnn: \frac{f(x)+g(x)}{|f(x) - g(x)|} \le \frac{1}{\varepsilon} \bigg\},  
\end{align*}
and
\begin{align}\label{eq:supf}
K_0 
:= \sup_{f\in \mathcal{K}}\|f\|_ {L^p(u;  L^q(v))}
= \sup_{f\in \mathcal{K}}\|f^a\|^{\frac{1}{a}}_{L^{\frac{p}{a}}(u;  L^{\frac{q}{a}}(v))}.
\end{align}
Invoking \eqref{eq:alphast}, we obtain 
\begin{align}\label{eq:fgalpha}
\|f - g\|_ {L^p(u;  L^q(v))} 
&\lesssim \| |f - g|^a \mathbf{1}_{E_\varepsilon}\|^{\frac{1}{a}}_{L^{\frac{p}{a}}(u; L^{\frac{q}{a}}(v))}
+\|(f - g) \mathbf{1}_{E_\varepsilon^c}\|_ {L^p(u; L^q(v))}
\\ \nonumber
&\lesssim \varepsilon^{a-1}\| |f^a - g^a| \mathbf{1}_{E_\varepsilon}\|^{\frac{1}{a}}_{L^{\frac{p}{a}}(u; L^{\frac{q}{a}}(v))}
+ \varepsilon \|(f+g) \mathbf{1}_{E_\varepsilon^c}\|_{L^p(u; L^q(v))}
\\ \nonumber
&\lesssim \varepsilon^{a-1}\|f^a - g^a\|^{\frac{1}{a}}_{L^{\frac{p}{a}}(u; L^{\frac{q}{a}}(v))}
+ \varepsilon \|f^a\|^{\frac{1}{a}}_{L^{\frac{p}{a}}(u; L^{\frac{q}{a}}(v))}
+ \varepsilon \|g^a\|^{\frac{1}{a}}_{L^{\frac{p}{a}}(u; L^{\frac{q}{a}}(v))}
\\ \nonumber
&\le \varepsilon^{a-1}\|f^a - g^a\|^{\frac{1}{a}}_{L^{\frac{p}{a}}(u; L^{\frac{q}{a}}(v))}
+ 2\varepsilon K_0.
\end{align}
Let $\{f_j\} \subset \mathcal{K}$ be an arbitrary sequence of functions. Then by the precompactness of $\mathcal{K}^a$, there exists a relabeled Cauchy subsequence of $\{f^a_j\}$. Therefore, given $\varepsilon >0$, there exists an integer $N=N(\varepsilon)$ such that for all $i, j \ge N$, 
\begin{align}\label{eq:fiafja}
\|f_i^a - f_j^a\|_{L^{\frac{p}{a}}(u;  L^{\frac{q}{a}}(v))} \le \varepsilon^a.
\end{align}
Fix $i, j\ge N$. The inequalities \eqref{eq:fgalpha} and \eqref{eq:fiafja} imply
 \begin{align*}
\|f_i - f_j\|_{L^p(u;  L^q(v))} \lesssim \varepsilon^{a - 1}\|f_i^a - f_j^a\|^{\frac{1}{a}}_{L^{\frac{p}{a}}(u;  L^{\frac{q}{a}}(v))}+ 2\varepsilon K_0
\le \varepsilon^a+2\varepsilon K_0, 
\end{align*}
which asserts that $\{f_j\}$ is a Cauchy sequence in $\mathcal{K} \subset L^p(u;  L^q(v))$. Hence, $\eqref{eq:KKa}$ holds.

Now we assume that the conditions \eqref{KRQB-1}--\eqref{KRQB-3} hold and $\mathcal{K}$ is a family of non-negative functions. It follows from \eqref{eq:alphafg} that
\begin{align*}
|f^a(x) - \langle f^a \rangle_{B(x, r)}| 
\le \fint_{B(0, r)} | \tau_y f(x)-f(x)|^a \, dy.
\end{align*}
Thus, the conditions \eqref{KRQB-1}--\eqref{KRQB-3} imply  
\begin{align}\label{eq:paqa1}
\sup_{f \in \mathcal{K}}\|f^a\|_{L^{\frac{p}{a}}(u;  L^{\frac{q}{a}}(v))}<\infty, 
\end{align}
\begin{align}\label{eq:paqa2}
\lim_{A \to \infty}\sup_{f \in \mathcal{K}}\|f^a \mathbf{1}_{B(0, A)^c}\|_{L^{\frac{p}{a}}(u;  L^{\frac{q}{a}}(v))} = 0, 
\end{align}
\begin{align}\label{eq:paqa3}
\lim_{r \to 0}\sup_{f \in \mathcal{K}}\|f^a - \langle f^a \rangle_{{B(x, r)}}\|_{L^{\frac{p}{a}}(u;  L^{\frac{q}{a}}(v))} = 0.
\end{align}
Note that $1 < p_0 < \frac{p}{a}$, $1 < q_0 < \frac{q}{a}$, $u \in A_{p_0}(\R^{n_1}) \subset A_{\frac{p}{a}}(\R^{n_1})$, and $v \in A_{q_0}(\R^{n_2}) \subset A_{\frac{q}{a}}(\R^{n_2})$. Then by \eqref{eq:paqa1}--\eqref{eq:paqa3} and Theorem \ref{thm:pq1infty}, $\mathcal{K}^a$ is precompact in $L^{\frac{p}{a}}(u;  L^{\frac{q}{a}}(v))$. In view of \eqref{eq:KKa}, $\mathcal{K}$ is precompact in $L^p(u; L^q(v))$. 
\qed

\begin{lemma}\label{lem:dense} 
Let  $u \in L^1_{\loc}(\R^{n_1})$ and $v \in L^1_{\loc}(\R^{n_2})$. Then $\mathscr{C}^{\infty}_c(\R^{n_1} \times \R^{n_2})$ is dense in $L^{p}(u;  L^{q}(v))$ for all $p, q \in (0, \infty)$.
\end{lemma}

\begin{proof}
Let $\varepsilon > 0$ and $f \in L^p(u;  L^q(v))$. Then there exists a simple function $g$ of the form $g = \sum_{j=1}^{j_0} c_j \mathbf{1}_{A_j}$ such that $\|f - g\|_{L^p(u;  L^q(v))} < \varepsilon$, where $j_0 \in \N$, $c_j$ is a constant, $A_j = A_j^1 \times A_j^2$, $A_j^i$ is a measurable set in $\R^{n_i}$, and further $A_j^i \subset Q_i$ with $Q_i$ being a cube in $\R^{n_i}$, $i=1, 2$. Hence, in order to show $\mathscr{C}^\infty_c(\R^{n_1} \times \R^{n_2})$ is dense in $L^{p}(u;  L^{q}(v))$, it suffices to check that for each $\mathbf{1}_{A_j}$, there exists $h \in \mathscr{C}^\infty_c(\R^{n_1} \times \R^{n_2})$ such that 
\begin{align}\label{eq:1Ajh}
\|\mathbf{1}_{A_j} - h\|_{L^p(u;  L^q(v))} < \varepsilon.  
\end{align}
For each $i=1, 2$, let $\varphi^{(i)} \in \mathscr{C}^{\infty}_c(\R^{n_i})$ be a non-negative function satisfying $ \int_{\R^{n_i}} \varphi^{(i)} \, dx_i = 1$ and $\supp \varphi^{(i)} \subset B(0, 1)\subset \R^{n_i}$. Set $\varphi_t^{(i)}(x_i) := t^{-n_i} \varphi^{(i)}(x_i/t)$, then for sufficiently small $t>0$, 
\begin{align*}
h_t := (\varphi_t^{(1)} * \mathbf{1}_{A_j^1}) \otimes (\varphi_t^{(2)} * \mathbf{1}_{A_j^2}) \in \mathscr{C}^{\infty}_c(\R^{n_1} \times \R^{n_2}) 
\quad \text{and} \quad 
\supp h_t \subset 2Q_1 \times 2Q_2.
\end{align*}
Observe that
\begin{align}\label{eq:1Aj}
\| \mathbf{1}_{A_j} - h_t \|_{L^p(u; L^q(v))}
& \lesssim \big\| \big(\mathbf{1}_{A_j^1} - \varphi_t^{(1)} * \mathbf{1}_{A_j^1}\big) \mathbf{1}_{A_j^2}\big\|_{L^p(u; L^q(v))}
\\ \nonumber
& \qquad + \big\|\big(\mathbf{1}_{A_j^2} - \varphi_t^{(2)} * \mathbf{1}_{A_j^2}\big)\varphi_t^{(1)} * \mathbf{1}_{A_j^1}\big\|_{L^p(u; L^q(v))} 
\\ \nonumber
&= \|\mathbf{1}_{A_j^1} - \varphi_t^{(1)} * \mathbf{1}_{A_j^1}\|_{L^p(u)} \|\mathbf{1}_{A_j^2}\|_{L^q(v)}
\\ \nonumber
& \qquad + \|\mathbf{1}_{A_j^2} - \varphi_t^{(2)} * \mathbf{1}_{A_j^2}\|_{L^q(v)} \|\varphi_t^{(1)} * \mathbf{1}_{A_j^1}\|_{L^p(u)}.
\end{align}
The fact $\mathbf{1}_{A_j^1} \in L^1(\R^{n_1})$ implies $\lim\limits_{t \to 0} \|\mathbf{1}_{A_j^1} - \varphi_t^{(1)} \ast \mathbf{1}_{A_j^1}\|_{L^1(\R^{n_1})}=0$. Note that 
\begin{align*}
\|\mathbf{1}_{A_j^1} - \varphi_t^{(1)} * \mathbf{1}_{A_j^1}\|_{L^{\infty}(\R^{n_1})} \le 2 
\quad\text{and}\quad 
\supp(\mathbf{1}_{A_j^1} - \varphi_t^{(1)} * \mathbf{1}_{A_j^1}) \subset 2Q_1.
\end{align*} 
Then there exists a sequence $\{t_k\}$ tending to $0$ as $k \to \infty$ such that 
\begin{align*}
\lim_{k \to \infty} \big[\mathbf{1}_{A_j^1} - \varphi_{t_k}^{(1)} \ast \mathbf{1}_{A_j^1}(x_1) \big] = 0, \quad \text{a.e. } x_1 \in 2Q_1.
\end{align*}
Thus, by the fact $u \in L_{\loc}^1(\R^{n_1})$ and Lebesgue's domination convergence theorem,
\begin{align}\label{eq:1Aj1}
\lim_{k \to \infty} \|\mathbf{1}_{A_j^1} - \varphi_{t_k}^{(1)} * \mathbf{1}_{A_j^1}\|_{L^p(u)} = 0.
\end{align}  
Similarly, taking a subsequence if necessary, we have 
\begin{align}\label{eq:1Aj2}
\lim_{k \to \infty} \|\mathbf{1}_{A_j^2} - \varphi_{t_k}^{(2)} * \mathbf{1}_{A_j^2}\|_{L^q(v)}=0.
\end{align}  
In view of \eqref{eq:1Aj}--\eqref{eq:1Aj2}, when $k$ is large enough, the function $h = h_{t_k}$ satisfies \eqref{eq:1Ajh}. This completes the proof.
\end{proof}

\section{Interpolation of boundedness on mixed-norm spaces}\label{sec:IBMS}
This section is devoted to showing Theorem \ref{thm:IBDD}. To this end, we combine the ideas in \cite{BP} and \cite{Stein}.

Our aim is to prove the following inequality
\begin{align}\label{eq:TMM}
\|T(f_1, \ldots, f_m)\|_{\L^{\r_0}(\w_0^{\r_0})}
\le M_0^{1-\theta} M_1^{\theta} \prod_{j=1}^m \|f_j\|_{\L^{\r_j}(\w_j^{\r_j})},
\end{align}
for all functions $f_j \in \L^{\r_j}(\w_j^{\r_j})$, $j=1, \ldots, m$. Due to technical reasons, we will complete the proof of \eqref{eq:TMM} step by step by considering the following three scenarios: 
\begin{enumerate}
\item all $f_j$, $\u_j$, and $\v_j$ are simple functions;

\item all $f_j$ are simple functions, and all $\u_j, \v_j$ are general weights;

\item all $f_j$, $\u_j$, and $\v_j$ are general functions.
\end{enumerate}

\medskip
\noindent
{\bf Step 1: simple functions and weights.} 
Suppose that all $f_j$, $\u_j$, and $\v_j$ are simple functions. 
We may assume $\|f_j\|_{L^{\r_j}(\w_j)} = 1$, $j=1, \ldots, m$. If we set 
\begin{equation*}
\widetilde{f}_j := f_j \otimes_{k=1}^{N_j} w_{j, k} 
\quad \text{and} \quad 
\widetilde{f}_j = |\widetilde{f}_j| e^{is_j}, \quad j=1, \ldots, m, 
\end{equation*} 
it is enough to prove
\begin{equation}\label{eq:tfww}
\big\| T\big(\widetilde{f}_1 \otimes_{k=1}^{N_1} w_{j, k}^{-1}, \ldots, \widetilde{f}_m \otimes_{k=1}^{N_m} w_{j, k}^{-1}\big) 
\otimes_{k=1}^{N_0} w_{j, k} \big\|_{\L^{\r_0}(\bm{\mu}_0)} 
\le M_0^{1-\theta} M_1^{\theta}.
\end{equation}

Pick a large integer $\tau \in \mathbb{N}$ so that $\tau > \max\big\{ \frac{1}{p_{0, k}}, \frac{1}{q_{0, k}}: k = 1, \ldots, N_0 \big\}$. By the definition of $\r_0$, we see that $\tau > \max\big\{\frac{1}{r_{0, k}}: k=1, \ldots, N_0 \big\}$. By \cite[Theorem~1]{BP}, there holds
\begin{align}\label{eq:Tf-dual}
\|T(\vec{\mathfrak{f}})\|_{\L^{\r_0}(\bm{\mu}_0)}^{\frac{1}{\tau}}
= \big\| |T(\vec{\mathfrak{f}})|^{\frac{1}{\tau}} \big\|_{\L^{\tau \r_0}(\bm{\mu}_0)}
= \sup_{g} \int_{\bm{\Sigma}_0} |T(\vec{\mathfrak{f}})(x)|^{\frac{1}{\tau}} \, g(x) \, d\bm{\mu}_0(x),
\end{align}
where the supremum is taken over all nonnegative simple functions $g \in \L^{(\tau \r_0)'}(\bm{\mu}_0)$ with $\|g\|_{\L^{(\tau \r_0)'}(\bm{\mu}_0)} = 1$. Fix such a function $g$. Then define  
\begin{equation*}
\Phi_\ell(z) := \int_{\bm{\Sigma}_0} |\phi_\ell(z)|^{\frac{1}{\tau}}\, d\bm{\mu}_0, \quad z \in \C, \, \ell \in \N,
\end{equation*}
where 
\begin{align*}
& \phi_{\ell}(z) := e^{\tau (z^2-1)/\ell} M_0^{z-1} M_1^{-z} \, T(\vec{F}_z) \otimes_{k=1}^{N_0} (u_{0, k}^{1-z} v_{0, k}^{z}) \, G_z^{\tau}, 
\\ 
& F_{z, j} := |\widetilde{f}_j|^{\frac{r_{j, N_j}}{r_{j, N_j}(z)}} e^{is_j} \otimes_{k=1}^{N_j} (u_{j, k}^{z-1} v_{j, k}^{-z}) \, 
\prod_{\nu=1}^{N_j-1} \| \widetilde{f}_j \|_{L^{r_{j, \nu+1}} \cdots L^{r_{j, N_j}}}^{\frac{r_{j, \nu}}{r_{j, \nu}(z)} - \frac{r_{j, \nu+1}}{r_{j, \nu+1}(z)}}, 
\\
& \frac{1}{\r_j(z)} := \frac{1-z}{\p_j} + \frac{z}{\q_j}, \quad j = 0, 1, \dots, m, 
\\ 
& G_z := g^{\frac{(\tau r_{0, N_0})'}{(\tau r_{0, N_0}(z))'}} \prod_{\nu=1}^{N_0-1} 
\|g\|_{L^{(\tau r_{0, \nu+1})'} \cdots L^{(\tau r_{0, N_0})'}}^{\frac{(\tau r_{0, \nu})'}{(\tau r_{0, \nu}(z))'} - \frac{(\tau r_{0, \nu+1})'}{(\tau r_{0, \nu+1}(z))'}}. 
\end{align*}
We would like to demonstrate 
\begin{align}\label{eq:Phi-ell}
\sup_{z \in S} \Phi_{\ell}(z) \le 1, \quad\forall \ell \in \N,
\end{align} 
where $S := \{z \in \C: 0 < \Re(z) < 1\}$. Thus, $\lim\limits_{\ell \to \infty} \Phi_{\ell}(\theta) \le 1$, which together with \eqref{eq:Tf-dual} gives \eqref{eq:tfww} as desired.

In order to verify \eqref{eq:Phi-ell}, by Hadamard's three lines lemma, it suffices to show that for each $\ell \in \N$,
\begin{align}
\label{Pell-1} & \text{$\Phi_{\ell}(z)$ is subharmonic in $S$}, 
\\
\label{Pell-2} & \text{$\Phi_{\ell}(z)$ is continuous and bounded on $\overline{S}$}, 
\\
\label{Pell-3} & \text{$\Phi_{\ell}(z) \le 1$, if $\Re(z) = 0$ or $\Re(z) = 1$}.
\end{align}
Obviously, $\phi_{\ell}(z)$ is holomorphic in $S$, which indicates $|\phi_{\ell}(z)|^{\frac{1}{\tau}}$ is subharmonic in $S$. Then for any circle $\{z \in \C: |z - z_0| < r\}$ in $S$, there holds
\begin{align}\label{eq:sub}
\frac{1}{2\pi r}\int_{0}^{2\pi}\Phi_{\ell}(re^{i t}-z_0)dt
&=\int_{\bm{\Sigma}_0}\frac{1}{2\pi r}
\int_{0}^{2\pi} |\phi_{\ell}(re^{i t}-z_0)|^{\frac{1}{\tau}}\, dt\, d\bm{\mu}_0
\\ \nonumber
&\ge \int_{\bm{\Sigma}_0} |\phi_{\ell}(z_0)|^{\frac{1}{\tau}}\, d\bm{\mu}_0 
= \Phi_\ell(z_0),
\end{align}
which means $\Phi_{\ell}(z)$ is subharmonic in $S$. By the continuity of $\phi_{\ell}(z)$ on $\overline{S}$, it is easy to see that $\Phi_{\ell}(z)$ is continuous on $\overline{S}$. To show the boundedness of $\Phi_{\ell}(z)$ on $\overline{S}$, fixing $z \in S$ and setting $h_j := e^{is_j} \otimes_{k=1}^{N_j} (u_{j, k}^{z-1} v_{j, k}^{-z})$, we have 
\begin{align*}
|\phi_{\ell}(z)|^{\frac{1}{\tau}} 
\lesssim e^{-|\Im(z)|^2/\ell} \sum_{l_0, l_1, \ldots, l_m} 
\big|T(h_1 {\bf 1}_{I_{1, l_1}},\dots, h_m {\bf 1}_{I_{m, l_m}}) \otimes_{k=1}^{N_0} u_{0, k} \big|^{\frac{1}{\tau}} 
{\bf 1}_{I_{0, l_0}},
\end{align*}
which along with with H\"older's inequality and \eqref{eq:IBDD-1} gives
\begin{align*}
\Phi_\ell(z) 
& \lesssim e^{-|\Im z|^2/\ell} \sum_{l_0, l_1, \ldots, l_m}
\|T(h_1 {\bf 1}_{I_{1, l_1}},\dots, h_m 
{\bf 1}_{I_{m, l_m}})\|_{\L^{\p_0}(\u_0)}^{\frac{1}{\tau}}\,  
\|\mathbf{1}_{I_{0, l_0}}\|_{\L^{(\tau \p_0)'}}
\\
&\lesssim e^{-|\Im z|^2/\ell} \sum_{l_0, l_1, \ldots, l_m}  \prod_{j=1}^{m} 
\|h_j {\bf 1}_{I_{j, l_j}}\|_{\L^{\p_j}(\u_j)}^{\frac{1}{\tau}}
\|\mathbf{1}_{I_{0, l_0}}\|_{\L^{(\tau \p_0)'}}
\\
&\lesssim e^{-|\Im z|^2/\ell}\sum_{l_0, l_1, \ldots, l_m}
\prod_{j=1}^{m} \|\mathbf{1}_{I_{j, l_j}}\|_{\L^{\p_j}}^{\frac{1}{\tau}}
\|\mathbf{1}_{I_{0, l_0}}\|_{\L^{(\tau \p_0)'}}
\\
& \lesssim e^{-|\Im(z)|^2/\ell} 
< \infty.
\end{align*}
This implies the boundedness of $\Phi_\ell(z)$ on $\overline{S}$ and $\lim\limits_{|\Im(z)| \to \infty} \sup\limits_{z \in \overline{S}} \Phi_\ell(z) = 0$.

It remains to prove \eqref{Pell-3}. We claim that for any $y \in \R$,
\begin{align}
\label{GGG}
\|G_{iy}\|_{\L^{(\tau \p_0)'}(\bm{\mu}_0)} = 1,  \quad
& \|G_{1+iy}\|_{\L^{(\tau \q_0)'}(\bm{\mu}_0)} = 1,
\\
\label{FFF}
\|F_{iy, \, j}\|_{\L^{\p_j}(\u_j^{\p_j})} = 1, \quad
& \|F_{1+iy, \, j}\|_{\L^{\q_j}(\v_j^{\q_j})} = 1.
\end{align}
Note that $\Re(\r_j(z)) = \p_j$ if $\Re(z) = 0$, and $\Re(\r_j(z)) = \q_j$ if $\Re(z) = 1$. In the case $N_0 = N_j = 1$, there holds
\begin{align*}
\|G_{iy, \, j}\|_{\L^{(\tau \p_0)'}(\bm{\mu}_0)} 
& = \Big\| g^{\frac{(\tau r_{0, 1})'}{(\tau p_{0, 1})'}} \Big\|_{L^{(\tau p_{0, 1})'}} 
= \|g\|_{L^{(\tau r_{0, 1})'}}^{\frac{(\tau r_{0, 1})'}{(\tau p_{0, 1})'}} 
= 1,
\\
\|G_{1+iy, \, j}\|_{\L^{(\tau \q_0)'}(\bm{\mu}_0)} 
& = \Big\| g^{\frac{(\tau r_{0, 1})'}{(\tau q_{0, 1})'}} \Big\|_{L^{(\tau q_{0, 1})'}} 
= \|g\|_{L^{(\tau r_{0, 1})'}}^{\frac{(\tau r_{0, 1})'}{(\tau q_{0, 1})'}} 
= 1,
\\
\|F_{iy, \, j}\|_{\L^{\p_j}(\u_j^{\p_j})} 
& = \big\| |\widetilde{f}_j|^{\frac{r_{j, 1}}{p_{j, 1}}} u_{j, 1}^{-1}\big\|_{L^{p_{j, 1}}(u_{j, 1}^{p_{j, 1}})}
= \|\widetilde{f}_j\|_{L^{r_{j, 1}}}^{\frac{r_{j, 1}}{p_{j, 1}}} 
= 1,
\\
\|F_{1+iy, \, j}\|_{\L^{\q_j}(\v_j^{\q_j})} 
& = \big\| |\widetilde{f}_j|^{\frac{r_{j, 1}}{q_{j, 1}}} v_{j, 1}^{-1}\big\|_{L^{q_{j, 1}}(v_{j, 1}^{q_{j, 1}})}
= \|\widetilde{f}_j\|_{L^{r_{j, 1}}}^{\frac{r_{j, 1}}{q_{j, 1}}} 
= 1.
\end{align*}
To achieve \eqref{GGG}, we use induction and assume \eqref{GGG} holds for all $N_0$-tuple functions $g$. Let us see what happens to the $(N_0+1)$-tuple function $g$. By definition, 
\begin{align*}
& \|G_{iy, \, j}\|_{\L^{(\tau \p_0)'}(\bm{\mu}_0)} 
\\
& = \bigg\| \Big\| g^{\frac{(\tau r_{0, N_0+1})'}{(\tau p_{0, N_0+1})'}} \Big\|_{L^{(\tau p_{0, N_0+1})'}}  
\prod_{\nu=1}^{N_0} \|g\|_{L^{(\tau r_{0, \nu+1})'} \cdots L^{(\tau r_{0, N_0+1})'}}^{\frac{(\tau r_{0, \nu})'}{(\tau r_{0, \nu}(iy))'} - \frac{(\tau r_{0, \nu+1})'}{(\tau r_{0, \nu+1}(iy))'}} \bigg\|_{L^{(\tau p_{0, 1})'} \cdots L^{(\tau p_{0, N_0})'}}
\\
& = \bigg\| \|g\|_{L^{(\tau r_{0, N_0+1})'}}^{\frac{(\tau r_{0, N_0+1})'}{(\tau p_{0, N_0+1})'}}  
\|g\|_{L^{(\tau r_{0, N_0+1})'}}^{\frac{(\tau r_{0, N_0})'}{(\tau p_{0, N_0})'} - \frac{(\tau r_{0, N_0+1})'}{(\tau p_{0, N_0+1})'}}
\prod_{\nu=1}^{N_0-1} \|g\|_{L^{(\tau r_{0, \nu+1})'} \cdots L^{(\tau r_{0, N_0+1})'}}^{\frac{(\tau r_{0, \nu})'}{(\tau r_{0, \nu}(iy))'} - \frac{(\tau r_{0, \nu+1})'}{(\tau r_{0, \nu+1}(iy))'}} \bigg\|_{L^{(\tau p_{0, 1})'} \cdots L^{(\tau p_{0, N_0})'}}
\\
& = \bigg\| \|g\|_{L^{(\tau r_{0, N_0+1})'}}^{\frac{(\tau r_{0, N_0})'}{(\tau p_{0, N_0})'}}
\prod_{\nu=1}^{N_0-1} \big\| \|g\|_{L^{(\tau r_{0, N_0+1})'}} \big\|_{L^{(\tau r_{0, \nu+1})'} \cdots L^{(\tau r_{0, N_0})'}}^{\frac{(\tau r_{0, \nu})'}{(\tau r_{0, \nu}(iy))'} - \frac{(\tau r_{0, \nu+1})'}{(\tau r_{0, \nu+1}(iy))'}} \bigg\|_{L^{(\tau p_{0, 1})'} \cdots L^{(\tau p_{0, N_0})'}}.
\end{align*}
Then, the induction hypothesis applied to the $N_0$-tuple function $\|g\|_{L^{(\tau r_{0, N_0+1})'}}$ yields
\begin{align*}
\|G_{iy, \, j}\|_{\L^{(\tau \p_0)'}(\bm{\mu}_0)} = 1.
\end{align*}
Analogously, one can show both the second equality in \eqref{GGG} and \eqref{FFF} hold.

Now let us bound $\Phi_\ell(iy)$ and $\Phi_\ell(1+iy)$. By H\"{o}lder's inequality, \eqref{eq:IBDD-1}, \eqref{GGG}, and \eqref{FFF}, we deduce that 
\begin{align}\label{eq:iy}
\Phi_\ell(iy) 
& \leq e^{-y^2/\ell} M_0^{-\frac{1}{\tau}} \int_{\bm{\Sigma}_0}
\big|T(\vec{F}_{iy}) \otimes_{k=1}^{N_0} u_{0, k} \big|^{\frac{1}{\tau}} \, |G_{iy}| \, d\bm{\mu}_0
\\  \nonumber
&\leq M_0^{-\frac{1}{\tau}} \big\| |T(\vec{F}_{iy}) \otimes_{k=1}^{N_0} u_{0, k}|^{\frac{1}{\tau}} \big\|_{\L^{\tau \p_0}(\bm{\mu}_0)} 
\, \|G_{iy}\|_{\L^{(\tau \p_0)'}(\bm{\mu}_0)}
\\  \nonumber
&= M_0^{-\frac{1}{\tau}} \|T(\vec{F}_{iy})\|_{\L^{\p_0}(\u_0^{\p_0})}^{\frac{1}{\tau}} 
\leq \prod_{j=1}^m \|F_{iy, \, j}\|_{\L^{\p_j}(\u_j^{\p_j})}^{\frac{1}{\tau}} 
= 1. 
\end{align}
By the same token, we obtain
\begin{align}\label{eq:iiy}
\Phi_\ell(1+iy) \le 1. 
\end{align}
Consequently, \eqref{Pell-3} follows from \eqref{eq:iy} and \eqref{eq:iiy}. This completes the  proof. 

\medskip
\noindent
{\bf Step 2: simple functions and general weights.}
Suppose that $f_j$ is a simple function on $(\bm{\Sigma}_j, \bm{\mu}_j)$ satisfying $f_j \in \L^{\p_j}(\u_j^{\p_j}) \cap \L^{\q_j}(\v_j^{\p_j})$, $j=1, \ldots, m$. Then there exists measurable sets $\bm{F}_j := F_{j, 1} \times \cdots \times F_{j, N_j} \subset \bm{\Sigma}_j$ such that $\bm{\mu}_j(\bm{F}_j) < \infty$ and $f_j = 0$ on $\bm{\Sigma}_j \setminus \bm{F}_j$, $j=1, \ldots, m$. For $\ell \in \mathbb{N}$, denote 
\begin{align*}
\bm{F}_j^{\ell} 
:= \{x \in \bm{F}_j: 2^{-\ell} \le u_{j, k}(x), v_{j, k}(x) \le 2^{\ell}, k=1, \ldots, N_j \}. 
\end{align*}
Then $f_j^{\ell} := f_j \mathbf{1}_{\bm{F}_j^{\ell}}$ is a simple function on $(\bm{\Sigma}_j, \bm{\mu}_j)$ and $f_j^{\ell} = 0$ on $\bm{\Sigma}_j \setminus \bm{F}_j^{\ell}$. For each $\ell \in \mathbb{N}$, $j = 1, \ldots, m$, and $k=1, \ldots, N_j$, set
\begin{align*}
u_{j, k}^{\ell}(x) := 
\begin{cases}
2^{-\ell} \beta, & \text{if } x \in \bm{F}_j^{\ell} \text{ and } 2^{-\ell} (\beta-1) < u_{j, k}(x) \le 2^{-\ell} \beta, \,  1 \le \beta < 2^{2\ell},
\\
0, & \text{if } x \in \bm{F}_j^{\ell} \text{ and } u_{j, k}(x) \ge 2^{\ell} \text{ or } u_{j, k}(x) < 2^{-\ell},  
\text{ or } x \in \bm{\Sigma}_j \setminus \bm{F}_j^{\ell}, 
\end{cases}
\end{align*}
and
\begin{align*}
u_{0, k}^{\ell}(x) := 
\begin{cases}
(\beta-1) 2^{-\ell}, & \text{if } x \in \bm{F}_j \text{ and } 2^{-\ell} (\beta-1) \le u_{0, k}(x) < 2^{-\ell} \beta, \,  1 \le \beta < 2^{2\ell},
\\
0, & \text{if } x \in \bm{F}_j \text{ and } u_{0, k}(x) \ge 2^{\ell}, \text{ or } x \in \bm{\Sigma}_j \setminus \bm{F}_j.
\end{cases}
\end{align*}
Similarly, we define $v_{j, k}^{\ell}$, $j=0, 1, \ldots, m$. Then it is easy to check that 
\begin{equation}\label{uvFF}
\begin{aligned}
& u_{j, k}(x) \le u_{j, k}^{\ell}(x), \, \,
v_{j, k}(x) \le v_{j, k}^{\ell}(x), \quad x \in \bm{F}_j^{\ell}, \, j=1, \ldots, m,
\\
& u_{j, k}^{\ell}(x) = v_{j, k}^{\ell}(x) = 0, \quad x \in \bm{\Sigma}_j \setminus \bm{F}_j^{\ell}, \, j=1, \ldots, m,
\\
& u_{0, k}^{\ell}(x) \le u_{0, k}(x), \quad v_{0, k}^{\ell}(x) \le v_{0, k}(x), \quad x \in \bm{\Sigma}_0,
\\
& \lim_{\ell \to \infty} u_{j, k}^{\ell} = u_{j, k}, \, 
\lim_{\ell \to \infty} v_{j, k}^{\ell} = v_{j, k}, \quad j=0, 1, \ldots, m.
\end{aligned}
\end{equation}
By Lebesgue's dominated convergence theorem, we see that for each $j=1, \ldots, m$,
\begin{align}\label{ffpq}
\lim_{\ell \to \infty} \|f_j^{\ell} - f_j\|_{\L^{\p_j}(\u_j^{\p_j})} = 0
\quad \text{and} \quad 
\lim_{\ell \to \infty} \|f_j^{\ell} - f_j\|_{\L^{\q_j}(\v_j^{\q_j})} = 0,
\end{align}
which along with \eqref{eq:IBDD-1} and \eqref{eq:IBDD-2} imply that 
\begin{align}
\label{TLF-1}
& \lim_{\ell \to \infty} \|T(f_1^{\ell}, \ldots, f_m^{\ell}) - T(f_1, \ldots, f_m)\|_{\L^{\p_0}(\u_0^{\p_0})} = 0,
\\
\label{TLF-2}
& \lim_{\ell \to \infty} \|T(f_1^{\ell}, \ldots, f_m^{\ell}) - T(f_1, \ldots, f_m)\|_{\L^{\q_0}(\v_0^{\q_0})} = 0.
\end{align}

Let $\u_j^{\ell} := (u_{j, 1}^{\ell}, \ldots, u_{j, N_j}^{\ell})$, $j=0, 1, \ldots, m$. Now, for any simple functions $g_j$ on $(\bm{\Sigma}_j, \bm{\mu}_j)$ with $g_j = 0$ on $\bm{\Sigma}_j \setminus \bm{F}_j^{\ell}$, $j=1, \ldots, m$, it follows from \eqref{eq:IBDD-1}, \eqref{eq:IBDD-2} and \eqref{uvFF} that
\begin{equation*}
\begin{aligned}
& \|T(g_1, \ldots, g_m)\|_{\L^{\p_0}(\u_0^{\ell, \p_0})}
\le \|T(g_1, \ldots, g_m)\|_{\L^{\p_0}(\u_0^{\p_0})}
\\ 
& \le M_0 \prod_{j=1}^m \|g_j\|_{\L^{\p_j}(\u_j^{\p_j})}
\le M_0 \prod_{j=1}^m \|g_j\|_{\L^{\p_j}(\u_j^{\ell, \p_j})},
\end{aligned}
\end{equation*}
and 
\begin{equation*}
\begin{aligned}
& \|T(g_1, \ldots, g_m)\|_{\L^{\q_0}(\v_0^{\ell, \q_0})}
\le \|T(g_1, \ldots, g_m)\|_{\L^{\q_0}(\v_0^{\q_0})}
\\
& \le M_1 \prod_{j=1}^m \|g_j\|_{\L^{\q_j}(\v_j^{\q_j})}
\le M_1 \prod_{j=1}^m \|g_j\|_{\L^{\q_j}(\v_j^{\ell, \q_j})}.
\end{aligned}
\end{equation*}
These mean that the assumptions in Theorem \ref{thm:IBDD} hold for simple weights $\u_j^{\ell}$ and $\v_j^{\ell}$, $j=0, 1, \ldots, m$. Thus, we use the conclusion in {\bf Step 1} to arrive at
\begin{align*}
\|T(g_1, \ldots, g_m)\|_{\L^{\r_0}(\w_0^{\ell, \r_0})}
\le M_0^{1-\theta} M_1^{\theta} \prod_{j=1}^m \|g_j\|_{\L^{\r_j}(\w_j^{\ell, \r_j})}.
\end{align*}
By Lebesgue's dominated convergence theorem and Fatou's lemma, we obtain 
\begin{align}\label{Tgsp}
\|T(g_1, \ldots, g_m)\|_{\L^{\r_0}(\w_0^{\r_0})}
\le M_0^{1-\theta} M_1^{\theta} \prod_{j=1}^m \|g_j\|_{\L^{\r_j}(\w_j^{\r_j})},
\end{align}
for all simple functions $g_j$ on $(\bm{\Sigma}_j, \bm{\mu}_j)$ with $g_j = 0$ on $\bm{\Sigma}_j \setminus \bm{F}_j^{\ell}$, $j=1, \ldots, m$.

We claim that
\begin{align}\label{TfCau}
\{T(f_1^{\ell}, \ldots, f_m^{\ell})\}_{\ell \in \mathbb{N}} \text{ is a Cauchy sequence in } \L^{\r_0}(\w_0^{\r_0}). 
\end{align}
In fact, for any $\ell_2 > \ell_1$, the inequality \eqref{Tgsp} gives
\begin{align*}
& \|T(f_1^{\ell_2}, \ldots, f_m^{\ell_2}) - T(f_1^{\ell_1}, \ldots, f_m^{\ell_1})\|_{\L^{\r_0}(\w_0^{\r_0})}
\\
& \le \sum_{j=1}^m \|T(f_1^{\ell_1}, \ldots, f_{j-1}^{\ell_1}, f_j^{\ell_2} - f_j^{\ell_1}, 
f_{j+1}^{\ell_2}, \ldots, f_m^{\ell_2})\|_{\L^{\r_0}(\w_0^{\r_0})}
\\
& \le M_0^{1-\theta} M_1^{\theta} \sum_{j=1}^m  
\prod_{i=1}^{j-1} \|f_i^{\ell_1}\|_{\L^{\r_i}(\w_i^{\r_i})} 
\|f_j^{\ell_2} - f_j^{\ell_1}\|_{\L^{\r_j}(\w_j^{\r_j})} 
\prod_{i=j+1}^m \|f_i^{\ell_1}\|_{\L^{\r_i}(\w_i^{\r_i})}.
\end{align*}
By H\"{o}lder's inequality, 
\begin{align*}
\|f_j^{\ell_2} - f_j^{\ell_1}\|_{\L^{\r_j}(\w_j^{\r_j})}
\le \|f_j^{\ell_2} - f_j^{\ell_1}\|_{\L^{\p_j}(\u_j^{\p_j})}^{1-\theta}
\|f_j^{\ell_2} - f_j^{\ell_1}\|_{\L^{\q_j}(\v_j^{\q_j})}^{\theta},
\end{align*}
which together with \eqref{ffpq} indicates $\{f_j^{\ell}\}_{\ell \in \mathbb{N}}$ is a Cauchy sequence in $\L^{\r_j}(\w_j^{\r_j})$. Hence, the above has shown \eqref{TfCau}.

Now we combine \eqref{TLF-1}, \eqref{TLF-2}, and \eqref{TfCau} to conclude  
\begin{align*}
\lim_{\ell \to \infty} \|T(f_1^{\ell}, \ldots, f_m^{\ell}) - T(f_1, \ldots, f_m)\|_{\L^{\r_0}(\w_0^{\r_0})} = 0,
\end{align*}
which along with \eqref{Tgsp} yields
\begin{align*}
\|T(f_1, \ldots, f_m)\|_{\L^{\r_0}(\w_0^{\r_0})}
\le M_0^{1-\theta} M_1^{\theta} \prod_{j=1}^m \|f_j\|_{\L^{\r_j}(\w_j^{\r_j})}.
\end{align*}

\medskip
\noindent
{\bf Step 3: general functions and weights.}
To show the general situation $f_j \in \L^{\r_j}(\w_j^{\r_j})$, $j=1, \ldots, m$, we present a density lemma as follows.

\begin{lemma}\label{lem:LpLqdense}
Let $\p, \q, \r \in (0, \infty)^N$, and let $\u$, $\v$, and $\w$ be weights on the product space $(\bm{\Sigma}, \bm{\mu})$. Then 
\begin{align*}
\mathbf{S}_{\p, \q} 
:=\big\{\text{simple functions in } \L^{\p}(\u^{\p}) \cap \L^{\q}(\v^{\q}) \big\} \text{ is dense in } \L^{\r}(\w^{\r}).
\end{align*}
\end{lemma}

\begin{proof}
We claim that 
\begin{align}\label{eq:SSr}
\mathbf{S}_{\p} := \big\{\text{simple functions in } \L^{\p}(\u^{\p}) \big\} \text{ is dense in } \L^{\p}(\u^{\p}),
\end{align}
Indeed, for $f \in \L^{\p}(\u^{\p})$, we may assume that $f \ge 0$ $\bm{\mu}$-a.e.. Let $\varepsilon>0$. Then there exists a simple function $a = \sum_{i=1}^{\ell_0} a_i {\bf 1}_{\bm{E}_i}$ such that $a\le f \otimes_{k=1}^N u_k$ and $\|f \otimes_{k=1}^N u_k - a\|_{\L^{\p}(\bm{\mu})}<\varepsilon/2$, where $a_i > 0$, $\{\bm{E}_i\}_{i=1}^{\ell_0}$ is a disjoint family and $0 < \bm{\mu}(\bm{E}_i) < \infty$. Set $\bm{E} := \bigcup_{i=1}^{\ell_0} \bm{E}_i$. Note that 
\begin{equation}\label{eq:fuE}
\|f \otimes_{k=1}^N u_k\|_{\L^{\p}(\bm{\Sigma} \setminus \bm{E}, \bm{\mu})}
\le \|f \otimes_{k=1}^N u_k - a\|_{\L^{\p}(\bm{\mu})}
< \varepsilon/2.
\end{equation}
In addition, there exist simple functions $\{b_j\}$ such that $\supp(b_j) \subset \bm{E}$ and $\lim\limits_{j \to \infty}b_j(x) = f(x)$ for all $x \in \bm{E}$. Then 
\begin{align*}
\lim_{j \to \infty} \|(f - b_j) \otimes_{k=1}^N u_k\|_{\L^{\p}(\bm{E}, \bm{\mu})}=0, 
\end{align*}
which gives
\begin{align}\label{eq:fbu}
\|(f - b_{j_0}) \otimes_{k=1}^N u_k\|_{\L^{\p}(\bm{E}, \bm{\mu})} < \varepsilon/2, \quad 
\text{for some } j_0 \in \N.  
\end{align}
As a consequence of \eqref{eq:fuE} and \eqref{eq:fbu}, we have 
\begin{equation*}
\|f - b_{j_0}\|_{\L^{\p}(\u^{\p})}
\le \|(f - b_{j_0}) \otimes_{k=1}^N u_k\|_{\L^{\p}(\bm{E}, \bm{\mu})} 
+ \|f \otimes_{k=1}^N u_k\|_{\L^{\p}(\bm{\Sigma} \setminus \bm{E}, \bm{\mu})} 
< \varepsilon,
\end{equation*}
which verifies \eqref{eq:SSr}. 

In light of \eqref{eq:SSr}, to conclude the proof, it suffices to prove that for any $\varepsilon>0$ and for any measurable set $\bm{E}  \subset \bm{\Sigma}$ with $\bm{\mu}(\bm{E}) < \infty$ and $\mathbf{1}_{\bm{E}} \in \L^{\r}(\w^{\r})$, there exists a simple function $h$ 
such that 
\begin{equation}\label{eq:Eaa}
h \in \L^{\p}(\u^{\p}) \cap \L^{\q}(\v^{\q}) 
\quad \text{and} \quad 
\|h - \mathbf{1}_{\bm{E}}\|_{\L^{\r}(\w^{\r})} < \varepsilon. 
\end{equation}
Let $\varepsilon > 0$. The fact $\mathbf{1}_{\bm{E}} \in \L^{\r}(\w^{\r})$ is equivalent to $\otimes_{k=1}^N w_k \in \L^{\r}(\bm{E}, \bm{\mu})$. Then there exists $\delta = \delta(\varepsilon) > 0$ so that 
\begin{equation}\label{eq:EF}
\forall \bm{F} \subset \bm{E}: \bm{\mu}(\bm{F}) < \delta \quad\Longrightarrow\quad 
\|\otimes_{k=1}^N w_k\|_{\L^{\r}(\bm{F}, \, \bm{\mu})} < \varepsilon. 
\end{equation}
Recall that $\bm{\mu}(\bm{E}) < \infty$ and $0 < u_k, v_k < \infty$ $\mu_k$-a.e., $k=1, \ldots, N$. Then there exists $A>0$ such that $\mu(\bm{F}) < \delta$, where
\begin{align*}
\bm{F} := \{x \in \bm{E}: \otimes_{k=1}^N u_k(x) > A\} \cup \{x \in \bm{E}: \otimes_{k=1}^N v_k(x) > A\}.
\end{align*}
Then \eqref{eq:EF} gives $\|\otimes_{k=1}^N w_k\|_{\L^{\r}(\bm{F}, \bm{\mu})} < \varepsilon$. Observe that 
\begin{align*}
\|\mathbf{1}_{\bm{E} \setminus \bm{F}}\|_{\L^{\p}(\u^{\p})} 
+ \|\mathbf{1}_{\bm{E} \setminus \bm{F}}\|_{\L^{\q}(\v^{\q})} 
\le 2A \bm{\mu}(\bm{E})
< \infty.
\end{align*}
Taking $h = \mathbf{1}_{\bm{E} \setminus \bm{F}}$, we conclude that $h \in \L^{\p}(\u^{\p}) \cap \L^{\q}(\v^{\q})$ and 
\begin{align*}
\|h - {\bf 1}_{\bm{E}}\|_{\L^{\r}(\w^{\r})} 
= \|{\bf 1}_{\bm{F}}\|_{\L^{\r}(\w^{\r})}
= \|\otimes_{k=1}^N w_k\|_{\L^{\r}(\bm{F}, \, \bm{\mu})} 
< \varepsilon,
\end{align*}
which coincides with \eqref{eq:Eaa}.  
\end{proof}

Now let $f_j \in \L^{\r_j}(\w_j^{\r_j})$, $j=1, \ldots, m$. By Lemma \ref{lem:LpLqdense}, there exists a sequence of simple functions $\{f_j^{\ell}\}_{\ell \in \mathbb{N}}$ such that 
\begin{align}\label{Tfell-1}
f_j^{\ell} \in \L^{\p_j}(\u_j^{\p_j}) \cap \L^{\q_j}(\v_j^{\q_j}) 
\quad \text{and} \quad 
\lim_{\ell \to \infty} \|f_j^{\ell} - f_j\|_{\L^{\r_j}(\w_j^{\r_j})} = 0.
\end{align}
Invoking the conclusion in {\bf Step 2}, we obtain
\begin{align}\label{Tfell-2}
\|T(f_1^{\ell}, \ldots, f_m^{\ell})\|_{\L^{\r_0}(\w_0^{\r_0})}
\le M_0^{1-\theta} M_1^{\theta} \prod_{j=1}^m \|f_j^{\ell}\|_{\L^{\r_j}(\w_j^{\r_j})},
\end{align}
for all $j=1, \ldots, m$ and $\ell \in \mathbb{N}$. In addition, as shown in \eqref{TfCau}, 
\begin{align*}
\{T(f_1^{\ell}, \ldots, f_m^{\ell})\}_{\ell \in \mathbb{N}} \text{ is a Cauchy sequence in } \L^{\r_0}(\w_0^{\r_0}). 
\end{align*}
Define 
\begin{align}\label{Tfell-3}
T(f_1, \ldots, f_m) 
:= \lim_{\ell \to \infty} T(f_1^{\ell}, \ldots, f_m^{\ell}) \quad \text{in } \L^{\r_0}(\w_0^{\r_0}).
\end{align}
This definition does not depend on the choice of $\{f_j^{\ell}\}_{\ell \in \mathbb{N}}$. Indeed, for another sequence of simple functions $\{g_j^{\ell}\}_{\ell \in \mathbb{N}}$ such that $g_j^{\ell} \in \L^{\p_j}(\u_j^{\p_j}) \cap \L^{\q_j}(\v_j^{\q_j})$ and $\lim\limits_{\ell \to \infty} \|g_j^{\ell} - f_j\|_{\L^{\r_j}(\w_j^{\r_j})} = 0$, there holds 
\begin{align*}
& \|T(f_1^{\ell}, \ldots, f_m^{\ell}) - T(g_1^{\ell}, \ldots, g_m^{\ell})\|_{\L^{\r_0}(\w_0^{\r_0})}
\\
& \le M_0^{1-\theta} M_1^{\theta} \sum_{j=1}^m  
\prod_{i=1}^{j-1} \|f_i^{\ell}\|_{\L^{\r_i}(\w_i^{\r_i})} 
\|f_j^{\ell} - g_j^{\ell}\|_{\L^{\r_j}(\w_j^{\r_j})} 
\prod_{i=j+1}^m \|g_i^{\ell}\|_{\L^{\r_i}(\w_i^{\r_i})},
\end{align*}
and
\begin{align*}
\|f_j^{\ell} - g_j^{\ell}\|_{\L^{\r_j}(\w_j^{\r_j})}
\le \|f_j^{\ell} - f_j\|_{\L^{\r_j}(\w_j^{\r_j})} + \|g_j^{\ell} - f_j\|_{\L^{\r_j}(\w_j^{\r_j})},
\end{align*}
which asserts
\begin{align*}
\lim_{\ell \to \infty} T(f_1^{\ell}, \ldots, f_m^{\ell})
= \lim_{\ell \to \infty} T(g_1^{\ell}, \ldots, g_m^{\ell}) \quad \text{in } \L^{\r_0}(\w_0^{\r_0}).
\end{align*}
Consequently, gathering \eqref{Tfell-1}--\eqref{Tfell-3}, we conclude
\begin{align*}
\|T(f_1, \ldots, f_m)\|_{\L^{\r_0}(\w_0^{\r_0})}
\le M_0^{1-\theta} M_1^{\theta} \prod_{j=1}^m \|f_j\|_{\L^{\r_j}(\w_j^{\r_j})}.
\end{align*}
This completes the proof. 
\qed

\section{Interpolation and extrapolation of compactness on mixed-norm spaces}\label{sec:INEX}
The goal of this section is to demonstrate Theorems \ref{thm:RdF-cpt}, \ref{thm:ApAq}, and \ref{thm:ICPT}.

\subsection{Interpolation of compactness on mixed-norm spaces}
Having established Theorems \ref{thm:KRQB} and \ref{thm:IBDD}, let us verify Theorem \ref{thm:ICPT}.

For convenience, set $\mathbb{T}(\vec{f})(x, y) := \tau_y T(\vec{f})(x)- T(\vec{f})(x)$. Note that $\mathbb{T}$ is an $m$-linear operator. The assumption \eqref{eq:ICTP-1} gives 
\begin{align}\label{eq:WMIP-1} 
\|T(\vec{f})\|_{L^{p_0}(u_0^{p_0}; \, L^{\widetilde{p}_0}(\widetilde{u}_0^{\widetilde{p}_0}))} 
& \leq M_0 \prod_{j=1}^m \|f_j\|_{L^{p_j}(u_j^{p_j}; \, L^{\widetilde{p}_j}(\widetilde{u}_j^{\widetilde{p}_j}))}, 
\end{align} 
which yields 
\begin{align}\label{eq:WMIP-2} 
\|T(\vec{f}) \mathbf{1}_{B(0, A)^c}\|_{L^{p_0}(u_0^{p_0}; \, L^{\widetilde{p}_0}(\widetilde{u}_0^{\widetilde{p}_0}))} 
& \leq M_0 \prod_{j=1}^m \|f_j\|_{L^{p_j}(u_j^{p_j}; \, L^{\widetilde{p}_j}(\widetilde{u}_j^{\widetilde{p}_j}))}, 
\quad \forall A > 0.
\end{align}
Recall that $u_0^{p_0} \in A_{\infty}(\R^{n_1})$ and $\widetilde{u}_0^{\widetilde{p}_0} \in A_{\infty}(\R^{n_2})$. Then there exists $s \in (1, \infty)$ so that $u_0^{p_0} \in A_s(\R^{n_1})$ and $\widetilde{u}_0^{\widetilde{p}_0} \in A_s(\R^{n_2})$. Take $0 < a < \min\{p_0/s, \widetilde{p}_0/s, 1\}$, which indicates $u_0^{p_0} \in A_{p_0/a}(\R^{n_1})$ and $\widetilde{u}_0^{\widetilde{p}_0} \in A_{\widetilde{p}_0/a}(\R^{n_2})$. Thus, 
\begin{align}\label{eq:WMIP-3} 
& \bigg\|\bigg[\fint_{B(0, r)} |\mathbb{T}(\vec{f})|^a \, dy \bigg]^{\frac1a}\bigg\|_{L^{p_0}(u_0^{p_0}; \, L^{\widetilde{p}_0}(\widetilde{u}_0^{\widetilde{p}_0}))} 
\\ \nonumber
& \lesssim \|M_{\mathcal{R}}(|T(\vec{f})|^a)^{\frac1a}\|_{L^{p_0}(u_0^{p_0}; \, L^{\widetilde{p}_0}(\widetilde{u}_0^{\widetilde{p}_0}))} 
+ \|T(\vec{f})\|_{L^{p_0}(u_0^{p_0}; \, L^{\widetilde{p}_0}(\widetilde{u}_0^{\widetilde{p}_0}))} 
\\ \nonumber
& = \|M_{\mathcal{R}}(|T(\vec{f})|^a)\|_{L^{\frac{p_0}{a}}(u_0^{p_0}; \, L^{\frac{\widetilde{p}_0}{a}}(\widetilde{u}_0^{\widetilde{p}_0}))}^{\frac1a}
+ \|T(\vec{f})\|_{L^{p_0}(u_0^{p_0}; \, L^{\widetilde{p}_0}(\widetilde{u}_0^{\widetilde{p}_0}))} 
\\ \nonumber
& \lesssim \|T(\vec{f})\|_{L^{p_0}(u_0^{p_0}; \, L^{\widetilde{p}_0}(\widetilde{u}_0^{\widetilde{p}_0}))} 
\le M_1 \prod_{j=1}^m \|f_j\|_{L^{p_j}(u_j^{p_j}; \, L^{\widetilde{p}_j}(\widetilde{u}_j^{\widetilde{p}_j}))},
\end{align} 
where we have used \eqref{eq:WMIP-1} and the weighted boundedness of the strong maximal operator $M_{\mathcal{R}}$ on mixed-norm spaces (cf. \cite[Theorem 1]{Kur}.

Let $\varepsilon > 0$. By \eqref{eq:ICTP-2} and Theorem \ref{thm:KRQB}, there exists $A_0 = A_0(\varepsilon) > 1$ such that 
\begin{align}
\label{eq:WMIP-4} 
\|T(\vec{f})\|_{L^{q_0}(v_0^{q_0}; \, L^{\widetilde{q}_0}(\widetilde{v}_0^{\widetilde{q}_0}))} 
& \leq M_2 \prod_{j=1}^m \|f_j\|_{L^{q_j}(v_j^{q_j}; \, L^{\widetilde{q}_j}(\widetilde{v}_j^{\widetilde{q}_j}))}, 
\\
\label{eq:WMIP-5}
\|T(\vec{f}) \mathbf{1}_{B(0, A)^c}\|_{L^{q_0}(v_0^{q_0}; \, L^{\widetilde{q}_0}(\widetilde{v}_0^{\widetilde{q}_0}))} 
& \leq \varepsilon \prod_{j=1}^m \|f_j\|_{L^{q_j}(v_j^{q_j}; \, L^{\widetilde{q}_j}(\widetilde{v}_j^{\widetilde{q}_j}))}, 
\quad \forall A > A_0,
\\ 
\label{eq:WMIP-6} \bigg\|\bigg[\fint_{B(0,r)} |\mathbb{T}(\vec{f})|^a \, dy \bigg]^{\frac1a}\bigg\|_{L^{q_0}(v_0^{q_0}; \, L^{\widetilde{q}_0}(\widetilde{v}_0^{\widetilde{q}_0}))} 
& \leq \varepsilon \prod_{j=1}^m \|f_j\|_{L^{q_j}(v_j^{q_j}; \, L^{\widetilde{q}_j}(\widetilde{v}_j^{\widetilde{q}_j}))}, 
\quad \forall 0 < r < A_0^{-1}.
\end{align} 

Now let $r_j$, $\widetilde{r}_j$, $w_j$, and $\widetilde{w}_j$ be in \eqref{eq:exp}, $j=0, \ldots, m$. In view of Theorem \ref{thm:IBDD}, interpolating between \eqref{eq:WMIP-1} and \eqref{eq:WMIP-4} gives
\begin{align}
\label{eq:WMIP-7} 
\|T(\vec{f})\|_{L^{r_0}(w_0^{r_0}; \, L^{\widetilde{r}_0}(\widetilde{w}_0^{\widetilde{r}_0}))} 
& \leq M_0^{1-\theta} M_2^{\theta} \prod_{j=1}^m \|f_j\|_{L^{r_j}(w_j^{r_j}; \, L^{\widetilde{r}_j}(\widetilde{w}_j^{\widetilde{r}_j}))}.
\end{align} 
Likewise, by \eqref{eq:WMIP-2}, \eqref{eq:WMIP-5}, and Theorem \ref{thm:IBDD} applied to $T(\vec{f}) \mathbf{1}_{B(0, A)^c}$, 
\begin{align}\label{eq:WMIP-8} 
\|T(\vec{f}) \mathbf{1}_{B(0, A)^c}\|_{L^{r_0}(w_0^{r_0}; \, L^{\widetilde{r}_0}(\widetilde{w}_0^{\widetilde{r}_0}))} 
& \leq M_0^{1-\theta} \varepsilon^{\theta} \prod_{j=1}^m \|f_j\|_{L^{r_j}(w_j^{r_j}; \, L^{\widetilde{r}_j}(\widetilde{w}_j^{\widetilde{r}_j}))}, 
\quad \forall A > A_0.
\end{align} 
To proceed, we choose
\begin{align*}
\Sigma_{j, k} & = \R^{n_k}, \quad d\mu_{j, k}(x_k) = dx_k, \quad j = 0, \ldots, m, \, k = 1, 2,
\\
\Sigma_{0, 3} & = \Rnn, \quad \text{and} \quad d\mu_{0, 3}(x) = \frac{\mathbf{1}_{B(0, r)}(x)}{|B(0, r)|} dx.
\end{align*}
Hence, invoking \eqref{eq:WMIP-3}, \eqref{eq:WMIP-6}, and Theorem \ref{thm:IBDD} applied to product spaces $(\bm{\Sigma}_j, \bm{\mu}_j)$, $j=0, \ldots, m$, 
we deduce 
\begin{align}
\label{eq:WMIP-9} 
\bigg\|\bigg[\fint_{B(0,r)} |\mathbb{T}(\vec{f})|^a \, dy \bigg]^{\frac1a}\bigg\|_{L^{r_0}(w_0^{r_0}; \, L^{\widetilde{r}_0}(\widetilde{w}_0^{\widetilde{r}_0}))} 
\leq M_1^{1-\theta} \varepsilon^{\theta} \prod_{j=1}^m \|f_j\|_{L^{r_j}(w_j^{r_j}; \, L^{\widetilde{r}_j}(\widetilde{w}_j^{\widetilde{r}_j}))}, 
\end{align} 
for all $0 < r < A_0^{-1}$. 
Therefore, \eqref{eq:WMIP-7}--\eqref{eq:WMIP-9} and Theorem \ref{thm:KRQB} imply the weighted compactness on mixed-norm spaces as desired.
\qed

\subsection{Extrapolation of compactness on mixed-norm spaces}
To prove Theorem \ref{thm:RdF-cpt}, we first present an extrapolation of boundedness on mixed-norm spaces as follows. 

\begin{theorem}\label{thm:RdF-bdd}
Let $\mathcal{F}$ be a collection of $(m+1)$-tuples of non-negative functions. Assume that there exists some $\vec{s} = (s_1, \ldots, s_m) \in [1, \infty]^m$ such that for all $\vec{w} = (w_1, \ldots, w_m) \in A_{\vec{s}}(\Rnn)$,
\begin{align}\label{Rbd-1}
\|f\|_{L^s(w^s)} 
\le C_0 \prod_{j=1}^m \|f_j\|_{L^{s_j}(w_j^{s_j})}, \quad (f, f_1, \ldots, f_m) \in \mathcal{F}, 
\end{align}
where $\frac1s = \sum_{j=1}^m \frac{1}{s_j}$ and $w = \prod_{j=1}^m w_j$. Then for all $\vec{\mathfrak{p}} = (\mathfrak{p}_1, \ldots, \mathfrak{p}_m) \in (1, \infty]^m$, for all $\vec{\mathfrak{q}} = (\mathfrak{q}_1, \ldots, \mathfrak{q}_m) \in (1, \infty]^m$, for all $\vec{u} = (u_1, \ldots, u_m) \in A_{\vec{\mathfrak{p}}}(\R^{n_1})$, and for all $\vec{v} = (v_1, \ldots, v_m) \in A_{\vec{\mathfrak{q}}}(\R^{n_2})$, 
\begin{align}\label{Rbd-2}
\|f\|_{L^{\mathfrak{p}}(u^{\mathfrak{p}}; L^{\mathfrak{q}}(v^{\mathfrak{q}})} 
\le C \prod_{j=1}^m \|f_j\|_{L^{\mathfrak{p}_j}(u_j^{\mathfrak{p}_j}; L^{\mathfrak{q}_j}(v_j^{\mathfrak{q}_j})} , \quad (f, f_1, \ldots, f_m) \in \mathcal{F}, 
\end{align}
where $\frac{1}{\mathfrak{p}} = \sum_{j=1}^m \frac{1}{\mathfrak{p}_j} > 0$, $\frac{1}{\mathfrak{q}} = \sum_{j=1}^m \frac{1}{\mathfrak{q}_j} > 0$, $u = \prod_{j=1}^m u_j$, and $v = \prod_{j=1}^m v_j$. 
\end{theorem}

\begin{proof}
By \eqref{Rbd-1} and extrapolation for multiple weights $A_{\vec{\tau}}(\Rnn)$ (cf. \cite[Theorem 3.12]{LMV21}), we obtain
\begin{align}\label{Rbd-3}
\|f\|_{L^{\tau}(w^{\tau})}
\le C \prod_{j=1}^m \|f_j\|_{L^{\tau_j}(w_j^{\tau_j})}, \quad (f, f_1, \ldots, f_m) \in \mathcal{F}, 
\end{align}
for all $\vec{\tau} = (\tau_1, \ldots, \tau_m) \in (1, \infty]^m$ and for all $\vec{w} = (w_1, \ldots, w_m) \in A_{\vec{\tau}}(\Rnn)$, where $\frac{1}{\tau} = \sum_{j=1}^m \frac{1}{\tau_j} > 0$ and $w = \prod_{j=1}^m w_j$. Now fix $\vec{\mathfrak{q}} = (\mathfrak{q}_1, \ldots, \mathfrak{q}_m) \in (1, \infty]^m$ and $\vec{v} = (v_1, \ldots, v_m) \in A_{\vec{\mathfrak{q}}}(\R^{n_2})$, where $\frac{1}{\mathfrak{q}} = \sum_{j=1}^m \frac{1}{\mathfrak{q}_j} > 0$ and $v = \prod_{j=1}^m v_j$. Given $(f, f_1, \ldots, f_m) \in \mathcal{F}$, define
\begin{align*}
F(x_1) & := \bigg(\int_{\R^{n_2}} |f(x_1, x_2)|^{\mathfrak{q}} \, v(x_2)^q \, dx_2\bigg)^{\frac{1}{\mathfrak{q}}}, 
\\
F_j(x_1) & := \bigg(\int_{\R^{n_2}} |f_j(x_1, x_2)|^{\mathfrak{q}_j} \, v(x_2)^{q_j} \, dx_2\bigg)^{\frac{1}{\mathfrak{q}_j}}, \quad j =1, \ldots, m.
\end{align*}
Set $\mathcal{F}_q := \{(F, F_1, \ldots, F_m): (f, f_1, \ldots, f_m) \in \mathcal{F}\}$. Then for all $\vec{u} \in A_{\vec{\mathfrak{q}}}(\R^{n_1})$, using the fact $\vec{u} \otimes \vec{v} \in A_{\vec{\mathfrak{q}}}(\Rnn)$ and \eqref{Rbd-3}, we obtain for any $(F, F_1, \ldots, F_m) \in \mathcal{F}_q$,
\begin{align*}
\|F u\|_{L^{\mathfrak{q}}(\R^{n_1})} 
= \|f u \otimes v\|_{L^{\mathfrak{q}}(\Rnn)} 
\lesssim \prod_{j=1}^m \|f_j u_j \otimes v_j\|_{L^{\mathfrak{q}_j}(\Rnn)}
= \prod_{j=1}^m \|F_j u_j\|_{L^{\mathfrak{q}_j}(\R^{n_1})}.
\end{align*}
This together with extrapolation for multiple weights $A_{\vec{\mathfrak{p}}}(\Rn)$ (cf. \cite[Theorem 1.1]{LMO}) gives 
\begin{align*}
\|f\|_{L^{\mathfrak{p}}(u^p; L^{\mathfrak{q}}(v^q))} 
= \|F u\|_{L^{\mathfrak{p}}(\R^{n_1})} 
\lesssim \prod_{j=1}^m \|F_j u_j\|_{L^{\mathfrak{p}_j}(\R^{n_1})}
= \prod_{j=1}^m \|f_j\|_{L^{\mathfrak{p}_j}(u_j; L^{\mathfrak{q}_j}(v_j))}
\end{align*}
for all $\vec{\mathfrak{p}} = (\mathfrak{p}_1, \ldots, \mathfrak{p}_m) \in (1, \infty]^m$ and for all $\vec{u} = (u_1, \ldots, u_m) \in A_{\vec{\mathfrak{p}}}(\R^{n_1})$, where $\frac{1}{\mathfrak{p}} = \sum_{j=1}^m \frac{1}{\mathfrak{p}_j} > 0$ and $u = \prod_{j=1}^m u_j$. This justifies \eqref{Rbd-2}.
\end{proof}

Now let us show Theorem \ref{thm:RdF-cpt}. By extrapolation of compactness \cite[Theorem 1.1]{COY}, the assumptions \eqref{RdFcpt-1} and \eqref{RdFcpt-2} imply
\begin{align}\label{PST-1}
T: L^{r_1}(w_1^{r_1}) \times \cdots \times L^{r_m}(w_m^{r_m}) \to L^r(w^r) \text{ compactly}
\end{align} 
for all $\vec{r} = (r_1, \ldots, r_m) \in (1, \infty]^m$ and for all $\vec{w} = (w_1, \ldots, w_m) \in A_{\vec{r}}(\Rnn)$, where $\frac1r = \sum_{j=1}^m \frac{1}{r_j} > 0$ and $w = \prod_{j=1}^m w_j$.

Let $\vec{p} = (p_1, \ldots, p_m) \in (1, \infty]^m$, $\vec{q} = (q_1, \ldots, q_m) \in (1, \infty]^m$, $\vec{u} = (u_1, \ldots, u_m) \in A_{\vec{p}}(\R^{n_1})$, $\vec{v} = (v_1, \ldots, v_m) \in A_{\vec{q}}(\R^{n_2})$, $\frac1p = \sum_{j=1}^m \frac{1}{p_j} > 0$, $\frac1q = \sum_{j=1}^m \frac{1}{q_j} > 0$, $u = \prod_{j=1}^m u_j$, and $v = \prod_{j=1}^m v_j$. Pick $\vec{\nu} = (\nu_1, \ldots, \nu_m) \in A_{\vec{r}}(\R^{n_1})$ and $\vec{\sigma} = (\sigma_1, \ldots, \sigma_m) \in A_{\vec{r}}(\R^{n_2})$, where $r_j \in (1, \infty)$ is a fixed number satisfying $r_j \ne p_j$ and $r_j \ne q_j$, $j=1, \ldots, m$. In view of \eqref{PST-1}, the fact $\vec{\nu} \otimes \vec{\sigma} \in A_{\vec{r}}(\Rnn)$ gives 
\begin{align}\label{PST-2}
T: L^{r_1}(\nu_1; L^{r_1}(\sigma_1)) \times \cdots \times L^{r_m}(\nu_m; L^{r_m}(\sigma_m)) \to L^r(\nu; L^r(\sigma))\text{ compactly}.
\end{align}
We claim that there exist $\theta \in (0, 1)$, $\vec{\mathfrak{s}} = \vec{\mathfrak{s}}(\theta) \in (1, \infty)^m$, $\vec{\mathfrak{t}} = \vec{\mathfrak{t}}(\theta) \in (1, \infty)^m$, $\vec{\mu} \in A_{\vec{\mathfrak{s}}}(\R^{n_1})$, and $\vec{\omega} \in A_{\vec{\mathfrak{t}}}(\R^{n_2})$ such that  
\begin{equation}\label{PST-3}
\begin{aligned}
u_j & = \mu_j^{1 - \theta} \nu_j^{\theta},\quad 
\frac{1}{p_j} = \frac{1 - \theta}{\mathfrak{s}_j} + \frac{\theta}{r_j}, \quad j=1, \ldots, m,
\\
v_j & = \omega_j^{1 - \theta} \sigma_j^{\theta}, \quad
\frac{1}{q_j} = \frac{1 - \theta}{\mathfrak{t}_j} + \frac{\theta}{r_j}, 
\quad j=1, \ldots, m. 
\end{aligned}
\end{equation}
Indeed, the proof is a routine application of the method in the proof of \cite[Lemma C.1]{CLSY}. Although two couples of weights are currently under consideration, the reverse H\"{o}lder inequalities of $A_{\lambda_1}(\R^{n_1})$ and $A_{\lambda_2}(\R^{n_2})$ weights allow us to choose a common and small parameter $\theta \in (0, 1)$ such that \eqref{PST-3} hold. We omit the details. 

Let $\frac{1}{\mathfrak{s}} = \sum_{j=1}^m \frac{1}{\mathfrak{s}_j}$, $\frac{1}{\mathfrak{t}} = \sum_{j=1}^m \frac{1}{\mathfrak{t}_j}$, $\mu = \prod_{j=1}^m \mu_j$, and $\omega = \prod_{j=1}^m \omega_j$. Then it is easy to check that
\begin{equation}\label{PST-4}
u = \mu^{1 - \theta} \nu^{\theta}, \quad 
v = \omega^{1 - \theta} \sigma^{\theta}, \quad
\frac{1}{p} = \frac{1 - \theta}{\mathfrak{s}} + \frac{\theta}{r}, \quad \text{and} \quad
\frac{1}{q} = \frac{1 - \theta}{\mathfrak{t}} + \frac{\theta}{r}. 
\end{equation}
By Theorem \ref{thm:RdF-bdd} applied to $\vec{\mu} \in A_{\vec{\mathfrak{s}}}(\R^{n_1})$ and $\vec{\omega} \in A_{\vec{\mathfrak{t}}}(\R^{n_2})$, the assumption \eqref{RdFcpt-1} implies
\begin{align}\label{PST-5}
T: L^{\mathfrak{s}_1}(\mu_1^{\mathfrak{s}_1}; L^{\mathfrak{t}_1}(\omega_1^{\mathfrak{t}_1})) \times \cdots \times L^{\mathfrak{s}_m}(\mu_m^{\mathfrak{s}_m}; L^{\mathfrak{t}_m}(\omega_m^{\mathfrak{t}_m})) \to L^{\mathfrak{s}}(\mu^{\mathfrak{s}}; L^{\mathfrak{t}}(\omega^{\mathfrak{t}})) 
\text{ boundedly}.
\end{align} 
Thus, it follows from \eqref{PST-2}--\eqref{PST-5} and Theorem \ref{thm:ICPT} that $T$ is compact from $L^{p_1}(u_1^{p_1}; L^{q_1}(v_1^{q_1})) \times \cdots \times L^{p_m}(u_m^{p_m}; L^{q_m}(v_m^{q_m}))$ to $L^p(u^p; L^q(v^q))$. This completes the proof of Theorem \ref{thm:RdF-cpt}.
\qed

\subsection{A characterization of $A_{\vec{p}}(\R^{n_1}) \times A_{\vec{q}}(\R^{n_2})$}
Let us prove Theorem \ref{thm:ApAq}. 

The implication $\eqref{list:uv} \Longrightarrow \eqref{list:M}$ is a consequence of Theorem \ref{thm:RdF-bdd} and \cite [Theorem 2.5]{GLPT}, where the later showed $\mathcal{M}_{\mathcal{R}}: L^{s_1}(w_1^{s_1}) \times \dots \times L^{s_m}(w_m^{s_m}) \to L^{s}(w^s)$ boundedly for all $\vec{s} = (s_1, \ldots, s_m) \in (1, \infty)^m$ and all $\vec{w} = (w_1, \ldots, w_m) \in A_{\vec{s}}(\Rnn)$, where $\frac1s = \sum_{j=1}^m \frac{1}{s_j} $ and $w = \prod_{j=1}^m w_j$. 

Let us show $\eqref{list:M} \Longrightarrow \eqref{list:uv}$. Fix cubes $I_1 \subset \R^{n_1}$ and $I_2 \subset \R^{n_2}$. By definition, for any $x \in I_1 \times I_2$,
\begin{align*}
\mathcal{M}_{\mathcal{R}}(\vec{f})(x) \ge \prod_{j=1}^m \langle |f_j| \rangle_{I_1 \times I_2},
\end{align*}
which together with \eqref{list:M} gives
\begin{align*}
& \prod_{j=1}^m \|f_j\|_{L^{p_j}(u_j^{p_j};  L^{q_j}(v_j^{q_j}))}
\gtrsim \|\mathcal{M}_{\mathcal{R}}(\vec{f})\|_{L^p(u^p;  L^q(v^q))}
\\
& \ge \bigg[ \int_{I_1} \bigg(\int_{I_2} |\mathcal{M}_{\mathcal{R}}(\vec{f})|^q 
v(x_2)^q dx_2\bigg)^{\frac pq} u(x_1)^p \, dx_1 \bigg]^{\frac 1p}
\\ \nonumber 
&\ge \bigg[ \int_{I_1} \bigg(\int_{I_2} \bigg|\prod_{j=1}^m \langle |f_j| \rangle_{I_1 \times I_2} \bigg|^q 
v(x_2)^q dx_2\bigg)^{\frac pq} u(x_1)^p \, dx_1\bigg]^{\frac 1p}
\\ \nonumber 
&= u^p(I_1)^{\frac1p} v^q(I_2)^{\frac1q} \prod_{j=1}^m \langle |f_j| \rangle_{I_1 \times I_2}. 
\end{align*}
Taking 
\begin{align*}
f_j(x) = \mathbf{1}_{I_1}(x_1) \mathbf{1}_{I_2}(x_2) 
 u_j^{-p'_j}(x_1) v_j^{-q'_j}(x_2), 
\end{align*}
we obtain
\begin{align*}
u^p(I_1)^{\frac1p} \, v^q(I_2)^{\frac1q} 
\prod_{j=1}^m \langle u_j^{-p'_j} \rangle_{I_1}  \langle v_j^{-q'_j} \rangle_{I_2} 
\lesssim \prod_{j=1}^m u_j^{-p'_j}(I_1)^{\frac{1}{p_j}} v_j^{-q'_j}(I_2)^{\frac{1}{q_j}},
\end{align*}
which is equivalent to
\begin{align*}
\bigg[ \langle u^p \rangle_{I_1}^{\frac1p} \prod_{j=1}^m \langle u_j^{-p'_j} \rangle_{I_1} \bigg]
\bigg[ \langle v^q \rangle_{I_2}^{\frac1q} \prod_{j=1}^m \langle v_j^{-q'_j} \rangle_{I_2} \bigg]
\lesssim 1. 
\end{align*}
By the fact that $v=\prod_{j=1}^mv_i$ and the Lebesgue differentiation theorem, we have 
\begin{align*}
\langle u^p \rangle_{I_1}^{\frac1p} \prod_{j=1}^m \langle u_j^{-p'_j} \rangle_{I_1}
\lesssim 1, \quad \text{for any cube } I_1 \subset \R^{n_1}. 
\end{align*}
Hence, $\vec{u} \in A_{\vec{p}}(\R^{n_1})$. Similarly, $\vec{v} \in A_{\vec{q}}(\R^{n_2})$.

The implication $\eqref{list:uv} \Longrightarrow \eqref{list:R}$ follows from Theorem \ref{thm:RdF-bdd} and \cite[Theorem 1.2]{LMV21}, since the latter proved $\mathcal{R}_{j_1}^1 \otimes \mathcal{R}_{j_2}^2: L^{s_1}(w_1^{s_1}) \times \cdots \times L^{s_m}(w_m^{s_m}) \to L^{s}(w^s)$ boundedly for all $\vec{s} = (s_1, \ldots, s_m) \in (1, \infty)^m$ and all $\vec{w} = (w_1, \ldots, w_m) \in A_{\vec{s}}(\Rnn)$, where $\frac1s = \sum_{j=1}^m \frac{1}{s_j} $ and $w = \prod_{j=1}^m w_j$.

It remains to verify $\eqref{list:R} \Longrightarrow \eqref{list:uv}$. We follow the strategy in \cite[p. 1249]{LOPTT} and \cite[p. 13]{LMV21}. By definition, 
\begin{align*}
\mathcal{R}_{j_1}^1 \otimes \mathcal{R}_{j_2}^2(\vec{f})(x)
= \mathrm{p.v.} \, \int_{(\R^{n_1} \times \R^{n_2})^m} 
\prod_{i=1}^2 \frac{\sum_{s=1}^m(x_i - y_{s, i})_{j_1}}{(\sum_{s=1}^m |x_i - y_{s, i}|^2)^{\frac{mn_i+1}{2}}} 
\prod_{j=1}^m f_j(y_j) \,d\vec{y},
\end{align*}
where $x = (x_1, x_2), y_s = (y_{s, 1}, y_{s, 2}) \in \R^{n_1} \times \R^{n_2}$. Fix cubes $I_1 \subset \R^{n_1}$ and $I_2 \subset \R^{n_2}$, and set 
\begin{align*}
I'_1 = I_1 + 2 \ell(I_1) e_{j_1}^{(1)} 
\quad \text{and} \quad 
I'_2 = I_2 + 2 \ell(I_2) e_{j_2}^{(2)},
\end{align*} 
where $e_{j_i}^{(i)}$ is the basis of $\R^{n_i}$ with the $j_i$-th component being 1 and other components being 0. Let $x = (x_1, x_2) \in I'_1 \times I'_2$ and $y_s = (y_{s, 1}, y_{s, 2}) \in I_1 \times I_2$. Pick 
\begin{align*}
f_j(x) = \mathbf{1}_{I_1}(x_1) \mathbf{1}_{I_2}(x_2) u_j^{-p'_j}(x_1) v_j^{-q'_j}(x_2).
\end{align*}
Then we have
\begin{align*}
(x_i - y_{s, i})_{j_i} \ge \ell(I_i), \quad |x_i - y_{s, i}| \simeq \ell(I_i), \quad i=1, 2,
\end{align*}
and
\begin{align*}
\mathcal{R}_{j_1}^1 \otimes \mathcal{R}_{j_2}^2 (\vec{f})(x)
\gtrsim \prod_{j=1}^m  \langle u_j^{-p'_j} \rangle_{I_1} \langle v_j^{-q'_j} \rangle_{I_2}.
\end{align*}
Therefore, by \eqref{list:R},
\begin{align*}
& u^p(I'_1)^{\frac1p} \, v^q(I'_2)^{\frac1q}
\prod_{j=1}^m \langle u_j^{-p'_j} \rangle_{I_1} \langle v_j^{-q'_j} \rangle_{I_2}
\lesssim \|\mathcal{R}_{j_1}^1 \otimes \mathcal{R}_{j_2}^2(\vec{f})\|_{ L^{p}(u^p;  L^{q}(v^q))}
\\ 
& \lesssim \prod_{j=1}^m \|f_j\|_{L^{p_j}(u_j^{p_j};  L^{q_j}(v_j^{q_j}))}
= \prod_{j=1}^m u_j^{-p'_j} (I_1)^{\frac{1}{p_j}} \, v_j^{-q'_j} (I_2)^{\frac{1}{q_j}}.
\end{align*}
That is, 
\begin{align*}
\bigg[ \langle u^p\rangle_{I'_1}^{\frac1p} \prod_{j=1}^m \langle u_j^{-p'_j} \rangle_{I_1}^{\frac{1}{p'_j}} \bigg]
\bigg[ \langle v^q\rangle_{I'_2}^{\frac1q} \prod_{j=1}^m \langle v_j^{-q'_j} \rangle_{I_2}^{\frac{1}{q'_j}} \bigg]
\lesssim 1.
\end{align*}
Analogously, 
\begin{align*}
\bigg[ \langle u^p\rangle_{I_1}^{\frac1p} \prod_{j=1}^m \langle u_j^{-p'_j} \rangle_{I'_1}^{\frac{1}{p'_j}} \bigg]
\bigg[ \langle v^q\rangle_{I_2}^{\frac1q} \prod_{j=1}^m \langle v_j^{-q'_j} \rangle_{I'_2}^{\frac{1}{q'_j}} \bigg]
\lesssim 1.
\end{align*}
By the fact that $v=\prod_{j=1}^m v_j$ and the Lebesgue differentiation theorem, we deduce 
\begin{align}\label{eq:I1I'1}
\langle u^p\rangle_{I_1}^{\frac1p} 
\prod_{j=1}^m  \langle u_j^{-p'_j} \rangle_{I'_1}^{\frac{1}{p'_j}}
\lesssim 1, \quad \text{for any cube } I_1,
\end{align}
and 
\begin{align}\label{eq:I'1I1}
\langle u^p\rangle_{I'_1}^{\frac1p} 
\prod_{j=1}^m  \langle u_j^{-p'_j} \rangle_{I_1}^{\frac{1}{p'_j}}
\lesssim 1, \quad \text{for any cube } I_1.
\end{align}
By by H\"older's inequality, 
\begin{align*}
\langle u^{- \frac{p}{mp-1}}\rangle_{I'_1}^{\frac{mp-1}{p}} 
\le \prod_{j=1}^m \langle u_j^{-p'_j}\rangle_{I'_1}^{\frac{1}{p'_j}},
\end{align*}
which along with \eqref{eq:I1I'1} gives
\begin{align}\label{eq:mpx}
\langle u^p\rangle_{I_1}^{\frac1p} 
\langle u^{- \frac{p}{mp-1}}\rangle_{I'_1}^{\frac{mp-1}{p}} 
\lesssim 1.
\end{align}
Since 
\begin{align*}
1 = \langle u^{\frac1m} \, u^{-\frac1m} \rangle_{I'_1}^m
\le \langle u^p\rangle_{I'_1}^{\frac1p}
\langle u^{- \frac{p}{mp-1}}\rangle_{I'_1}^{\frac{mp-1}{p}},
\end{align*}
the inequality \eqref{eq:mpx} leads to
\begin{align}\label{eq:I1x}
\frac{\langle u^p\rangle_{I_1}^{\frac1p}}{\langle u^p\rangle_{I'_1}^{\frac1p}}
= \frac{\langle u^p \rangle_{I_1}^{\frac1p} \langle u^{- \frac{p}{mp-1}} \rangle_{I'_1}^{\frac{mp-1}{p}}}
{\langle u^p \rangle_{I'_1}^{\frac1p} \langle u^{- \frac{p}{mp-1}} \rangle_{I'_1}^{\frac{mp-1}{p}}}
\lesssim 1.
\end{align}
Gathering \eqref{eq:I'1I1} and \eqref{eq:I1x}, we conclude
\begin{align*}
\langle u^p\rangle_{I_1}^{\frac1p} 
\prod_{j=1}^m  \langle u_j^{-p'_j} \rangle_{I_1}^{\frac{1}{p'_j}}
\lesssim \langle u^p\rangle_{I'_1}^{\frac1p} 
\prod_{j=1}^m  \langle u_j^{-p'_j} \rangle_{I_1}^{\frac{1}{p'_j}}
\lesssim 1, \quad \text{for any cube } I_1,
\end{align*}
which implies $\vec{u} \in A_{\vec{p}}(\R^{n_1})$. Similarly, $\vec{v} \in A_{\vec{q}}(\R^{n_2})$.

\section{Applications}\label{sec:app}
In this section, we present several applications of the extrapolation Theorem \ref{thm:RdF-cpt}. It will become clear that proving the compactness of bi-parameter operators is an extremely challenging task, even in the unweighted setting. 

\subsection{Bi-parameter Calder\'{o}n--Zygmund operators}
We begin with showing Theorem \ref{thm:CZO}. In light of \cite[Theorem 1.1]{CY}, $T$ is compact from $L^{r_1}(w_1^{r_1}) \times \cdots \times L^{r_m}(w_m^{r_m})$ to $L^r(w^r)$ for all $\vec{r} = (r_1, \ldots, r_m) \in (1, \infty]^m$ and $\vec{w} = (w_1, \ldots, w_m) \in A_{\vec{r}}(\Rnn)$, where $\frac1r = \sum_{j=1}^m \frac{1}{r_j} > 0$ and $w = \prod_{j=1}^m w_j$. This immediately implies that the assumptions \eqref{RdFcpt-1} and \eqref{RdFcpt-2} in Theorem \ref{thm:RdF-cpt} hold. Thus, Theorem \ref{thm:RdF-cpt} gives the compactness on weighted mixed-norm spaces as desired. \qed

\subsection{Bi-parameter dyadic paraproducts} 
As above, to demonstrate Theorems \ref{thm:para}, it suffices to prove 
\begin{align}\label{eq:paraLr}
\text{$T$ is compact from $L^{r_1}(w_1^{r_1}) \times \cdots \times L^{r_m}(w_m^{r_m})$ to $L^r(w^r)$}
\end{align}
for all $\vec{r} = (r_1, \ldots, r_m) \in (1, \infty]^m$ and $\vec{w} = (w_1, \ldots, w_m) \in A_{\vec{r}}(\Rnn)$, where $\frac1r = \sum_{j=1}^m \frac{1}{r_j} > 0$ and $w = \prod_{j=1}^m w_j$. Indeed, the proof of \eqref{eq:paraLr} is similar to that of $\mathbb{E}_{\omega} \mathbf{F}_{\mathbf{a}_{\omega}}$ in \cite[Theorem 1.5]{CY}, where extrapolation of compactness on product spaces allows us to reduce the matter to the compactness from $L^{2m}(\Rnn) \times \cdots \times L^{2m}(\Rnn)$ to $L^2(\Rnn)$. By means of \cite[Theorem 6.3]{CY}, the later follows from minor modifications of \cite[Section 6.5]{CY}. The details are left to the reader. 
\qed

\subsection{Bi-parameter continuous paraproducts}
Next, we turn to the proof of Theorem \ref{thm:Pib}. For a bilinear operator $T$ acting on functions defined on the product space $\Rnn$, we define the following duals
\begin{align*}
& \langle T(f_1 \otimes f_2, g_1 \otimes g_2), h_1 \otimes h_2 \rangle
\\
& = \langle T^{1*}(h_1 \otimes h_2, g_1 \otimes g_2), f_1 \otimes f_2 \rangle
= \langle T^{2*}(f_1 \otimes f_2, h_1 \otimes h_2), g_1 \otimes g_2 \rangle
\\
& = \langle T^{1*}_1(h_1 \otimes f_2, g_1 \otimes g_2), f_1 \otimes h_2 \rangle
= \langle T^{1*}_2(f_1 \otimes h_2, g_1 \otimes g_2), h_1 \otimes f_2 \rangle
\\
& = \langle T^{2*}_1(f_1 \otimes f_2, h_1 \otimes g_2), g_1 \otimes h_2 \rangle
= \langle T^{2*}_2(f_1 \otimes f_2, g_1 \otimes h_2), h_1 \otimes g_2 \rangle
\\
& = \langle T^{1*, 2*}_{1, 2}(h_1 \otimes f_2, g_1 \otimes h_2), f_1 \otimes g_2 \rangle
= \langle T^{1*, 2*}_{2, 1}(f_1 \otimes h_2, h_1 \otimes g_2), g_1 \otimes f_2 \rangle.
\end{align*}

The following result contains the main ingredients of the proof of Theorem \ref{thm:Pib}. 

\begin{theorem}\label{thm:ppp}
Let $m = (0, 0)$, $\rho = (1, 1)$, and $\delta = (\delta_1, \delta_2) \in [0, 1)^2$. Let $b \in \BMO(\Rnn)$. Then the following hold:
\begin{list}{\rm (\theenumi)}{\usecounter{enumi}\leftmargin=1.3cm \labelwidth=1cm \itemsep=0.1cm \topsep=0.2cm \renewcommand{\theenumi}{\alph{enumi}}}
\item\label{ppp-1} $\pi_b$ admits a full kernel representation;

\item\label{ppp-2} $\pi_b$ admits a partial kernel representation;

\item\label{ppp-3} $\pi_b$ is bounded from $L^4(\Rnn) \times L^4(\Rnn)$ to $L^2(\Rnn)$;

\item\label{ppp-4} $\mathbb{T}(1, 1) \in \BMO(\Rnn)$, where 
\begin{align*}
\mathbb{T} \in \{T, T^{1*}, T^{2*}, T_1^{1*}, T_2^{1*}, T_1^{2*}, T_2^{2*}, T_{1, 2}^{1*, 2*}, T_{2, 1}^{1*, 2*}: T = \pi_b\};
\end{align*}

\item\label{ppp-5} If $b \in \CMO(\Rnn)$, then $\pi_b$ is compact from $L^4(\Rnn) \times L^4(\Rnn)$ to $L^2(\Rnn)$.
\end{list}
\end{theorem}

Let us conclude Theorem \ref{thm:Pib} from Theorem \ref{thm:ppp} as follows. Property \eqref{Pib-1} is a consequence of parts \eqref{ppp-1}--\eqref{ppp-4} and \cite[Theorem 1.2]{LMV21}, where we have used that part \eqref{ppp-3} implies the weak boundedness property and the diagonal $\BMO$ condition. Property \eqref{Pib-2} follows from property \eqref{Pib-1}, part \eqref{ppp-5}, and \cite[Theorem 1.6]{CY}, while property \eqref{Pib-3} follows from property \eqref{Pib-1}, part \eqref{ppp-5}, and Theorem \ref{thm:RdF-cpt}. This shows Theorem \ref{thm:Pib}.

\medskip 
\noindent{\bf Proof of Theorem \ref{thm:ppp}.} Denote
\begin{align*}
K(x, y, z) 
:= \int_0^{\infty} \int_0^{\infty}  K_{t_1, t_2}(x, y, z)  \, \frac{dt_1}{t_1} \frac{dt_2}{t_2},
\end{align*}
where
\begin{align*}
K_{t_1, t_2}(x, y, z) 
:= \int Q_{t_1, t_2} b(u) \, \psi_{t_1}^{(1)} \otimes \psi_{t_2}^{(2)} (x-u) \, 
\varphi_{t_1}^{(1)} \otimes \varphi_{t_2}^{(2)} (u-y) \, 
\varphi_{t_1}^{(1)} \otimes \varphi_{t_2}^{(2)} (u-z) \, du. 
\end{align*}
By the cancellation and compact support of $\psi^{(1)}$ and $\psi^{(2)}$, there holds
\begin{align*}
& \sup_{t_1, t_2 > 0} \|Q_{t_1, t_2} b\|_{L^{\infty}(\Rnn)} 
\le \sup_{\substack{x \in \Rnn \\ t_1, t_2 > 0}} 
|\langle \psi_{t_1}^{(1)} \otimes \psi_{t_2}^{(2)}(x - \cdot), b(\cdot) \rangle|
\\
& \lesssim \sup_{\substack{x \in \Rnn \\ t_1, t_2 > 0}} 
\| \psi_{t_1}^{(1)} \otimes \psi_{t_2}^{(2)}(x-\cdot)\|_{\mathrm{H}^1(\Rnn)} \|b\|_{\BMO(\Rnn)}
\lesssim 1,
\end{align*}
provided that $\psi_{t_1}^{(1)} \otimes \psi_{t_2}^{(2)}$ is a $\mathrm{H}^1(\Rnn)$-atom. Thus, for all multi-indices $\alpha$, $\beta$ and $\gamma$,
\begin{align*}
|\partial_x^{\alpha} \partial_y^{\beta} \partial_z^{\gamma} K_{t_1, t_2}(x, y, z)|
& \lesssim \prod_{i=1}^2 \int_{\R^{n_i}} 
\Big|\partial_{x_i}^{\alpha_i} \psi_{t_i}^{(i)} (x_i - u_i) \, 
\partial_{y_i}^{\beta_i} \varphi_{t_i}^{(i)} (u_i - y_i) \, 
\partial_{z_i}^{\gamma_i} \varphi_{t_i}^{(i)} (u_i - z_i) \Big|\, du_i
\\
& \lesssim \prod_{i=1}^2 \int_{\R^{n_i}} 
\frac{t_i^{-n_i - |\alpha_i|}}{(1 + |x_i - u_i|/t_i)^{4(n_i + 1)}}
\frac{t_i^{-n_i - |\beta_i|}}{(1 + |u_i - y_i|/t_i)^{4(n_i + 1)}}
\\
& \qquad\qquad \times \frac{t_i^{-n_i - |\gamma_i|}}{(1 + |u_i - z_i|/t_i)^{4(n_i + 1)}} \, du_i
\\
& \lesssim \prod_{i=1}^2 \frac{t_i^{-2n_i - |\alpha_i| - |\beta_i| - |\gamma_i|}}{(1 + |x_i - y_i|/t_i)^{2(n_i + 1)} (1 + |x_i - z_i|/t_i)^{2(n_i + 1)}}.
\end{align*}
If we set $A_i := |x_i - y_i| + |x_i - z_i|$, then whenever $|\alpha_i| + |\beta_i| + |\gamma_i| \le 1$, 
\begin{align}\label{eq:pK}
|\partial_x^{\alpha} \partial_y^{\beta} \partial_z^{\gamma} K(x, y, z)|
& \lesssim \prod_{i=1}^2 \bigg[\int_0^{A_i} \frac{t_i^{2 - |\alpha_i| - |\beta_i| - |\gamma_i|}}{A^{2n_i+2}} \, \frac{dt_i}{t_i} 
+ \int_{A_i}^{\infty} t_i^{-2n_i - |\alpha_i| - |\beta_i| - |\gamma_i|} \, \frac{dt_i}{t_i} \bigg]
\\ \nonumber
& \lesssim  \prod_{i=1}^2 A_i^{-2n_i - |\alpha_i| - |\beta_i| - |\gamma_i|},
\end{align}
which implies the full kernel representation.

To justify the partial kernel representation, define
\begin{align*}
K_{f_2, g_2, h_2}(x_1, y_1, z_1) 
& := \int_{\R^{n_2}} \int_{\R^{n_2}} \int_{\R^{n_2}} K(x, y, z) f_2(y_2) \, g_2(z_2) \, h_2(x_2) \, dx_2 \, dy_2 \, dz_2.
\end{align*}
Then for any multi-indices $\alpha_1, \beta_1, \gamma_1$,
\begin{align*}
& \partial_{x_1}^{\alpha_1} \partial_{y_1}^{\beta_1} \partial_{z_1}^{\gamma_1} K_{f_2, g_2, h_2}(x_1, y_1, z_1) 
\\
& = \int_{\R^{n_1+1}_+} \partial_{x_1}^{\alpha_1} \psi_{t_1}^{(1)}(x_1 - u_1) \, 
\partial_{y_1}^{\beta_1} \varphi_{t_1}^{(1)}(u_1 - y_1) \, 
\partial_{z_1}^{\gamma_1} \varphi_{t_1}^{(1)}(u_1 - z_1) 
H(u_1, t_1) \, \frac{du_1 \, dt_1}{t_1},
\end{align*}
where
\begin{align*}
H(u_1, t_1)
:= \int_{\R^{n_2+1}_+} Q_{t_1, t_2} b(u) \, 
P_{t_2}^{(2)} f_2(u_2) \, P_{t_2}^{(2)} g_2(u_2) \, Q_{t_2}^{(2)} h_2(u_2) \, \frac{du_2 \, dt_2}{t_2}.
\end{align*}
To proceed, given $(u_1, t_1) \in \R^{n_1+1}_+$, define a measure
\begin{align*}
d\nu_{u_1}^{t_1}(u_2, t_2) := |Q_{t_1, t_2} b(u)|^2 \, \frac{du_2 \, dt_2}{t_2}.
\end{align*}
If we set $T_{I_i} := I_i \times (0, \ell(I_i)) = \bigcup_{I'_i \in \D_i, I'_i \subset I_i} W_{I'_i}$, then 
\begin{align*}
\nu_{u_1}^{t_1}(T_{I_2}) 
= \sum_{I'_2 \subset I_2} \iint_{W_{I'_2}} |Q_{t_1, t_2} b(u)|^2 \, \frac{du_2 \, dt_2}{t_2},
\end{align*}
and 
\begin{align*}
\iint_{T_{I_1}} \nu_{u_1}^{t_1}(T_{I_2}) \, \frac{du_1 \, dt_1}{t_1}
= \sum_{I'_1 \times I'_2 \subset I_1 \times I_2} 
\iint_{W_{I'_2}} \iint_{W_{I'_1}} |Q_{t_1, t_2} b(u)|^2 \, \frac{du_1 \, dt_1}{t_1} \, \frac{du_2 \, dt_2}{t_2}
\lesssim |I_1| |I_2|,
\end{align*}
where Theorem \ref{thm:Carleson} part \eqref{Car-1} was used in the last step. By Lebesgue's differentiation theorem, there holds
\begin{align*}
\nu_{u_1}^{t_1}(T_{I_2}) 
\lesssim |I_2|, \quad \text{a.e. } (u_1, t_1) \in \R^{n_1+1}_+.
\end{align*}
This means that 
\begin{align}\label{nuCar}
\nu_{u_1}^{t_1} \text{ is a Carleson measure on $\R^{n_2+1}_+$}, \quad \text{a.e. } (u_1, t_1) \in \R^{n_1+1}_+.
\end{align}
If we define the one-parameter square functions $S_i$ and $S'_i$ on $\R^{n_i}$ as
\begin{align*}
S_i h_i(x_i) 
& := \bigg(\int_0^{\infty} |Q_{t_i}^{(i)} h_i(x_i)|^2 \, \frac{dt_i}{t_i} \bigg)^{\frac12}, \quad x_i \in \R^{n_i},
\\
S'_i h_i(x_i) 
& := \bigg(\int_0^{\infty} |P_{t_i}^{(i)} h_i(x_i)|^2 \, \frac{dt_i}{t_i} \bigg)^{\frac12}, \quad x_i \in \R^{n_i},
\end{align*}
then by \cite[p. 389]{DJ},
\begin{align}\label{SS22}
\|S_i h_i\|_{L^2(\R^{n_i})} + \|S'_i h_i\|_{L^2(\R^{n_i})}
\lesssim \|h_i\|_{L^2(\R^{n_i})}, \quad i=1, 2.
\end{align}
We invoke H\"{o}lder's inequality, \eqref{nuCar}, \cite[Theorem 3.3.7]{Gra-2}, and \eqref{SS22} to deduce
\begin{align}\label{Hut}
|H(u_1, t_1)|
&\le \bigg(\int_{\R^{n_2+1}_+}  \, 
|P_{t_2}^{(2)} f_2(u_2) \, P_{t_2}^{(2)} g_2(u_2)|^2 \, 
|Q_{t_1, t_2} b(u)|^2 \, \frac{du_2 \, dt_2}{t_2} \bigg)^{\frac12}
\\ \nonumber
&\quad \times \bigg(\int_{\R^{n_2+1}_+} |Q_{t_2}^{(2)} h_2(u_2)|^2 \, \frac{du_2 \, dt_2}{t_2} \bigg)^{\frac12}
\\ \nonumber
& = \|(P_{t_2}^{(2)} f_2) \, (P_{t_2}^{(2)} g_2)\|_{L^2(\R^{n_2+1}_+, \nu_{u_1}^{t_1})} \|S_2 h_2\|_{L^2(\R^{n_2})}
\\ \nonumber
& \le \|P_{t_2}^{(2)} f_2\|_{L^4(\R^{n_2+1}_+, \nu_{u_1}^{t_1})}
\|P_{t_2}^{(2)} g_2\|_{L^4(\R^{n_2+1}_+, \nu_{u_1}^{t_1})} \|S_2 h_2\|_{L^2(\R^{n_2})}
\\ \nonumber
& \lesssim \|f_2\|_{L^4(\R^{n_2})} \|g_2\|_{L^4(\R^{n_2})} \|h_2\|_{L^2(\R^{n_2})},
\end{align}
where the implicit constants are independent of $u_1$ and $t_1$. Then using \eqref{Hut} and the same estimates as in \eqref{eq:pK}, we obtain
\begin{align*}
\big|\partial_{x_1}^{\alpha_1} \partial_{y_1}^{\beta_1} \partial_{z_1}^{\gamma_1} K_{f_2, g_2, h_2}(x_1, y_1, z_1) \big|
\lesssim \frac{\|f_2\|_{L^4(\R^{n_2})} \|g_2\|_{L^4(\R^{n_2})} \|h_2\|_{L^2(\R^{n_2})}}{(|x_1 - y_1| + |x_1 - z_1|)^{2n_1 + |\alpha_1| + |\beta_1| + |\gamma_1|}},
\end{align*}
which gives the partial kernel representation on the first parameter. Symmetrically, one can show the partial kernel representation on the second parameter. 

 Define bi-parameter square functions as
\begin{align*}
S f(x) & := \bigg(\int_0^{\infty} \int_0^{\infty} |Q_{t_1, t_2} f(x)|^2 \, \frac{dt_1}{t_1} \, \frac{dt_2}{t_2} \bigg)^{\frac12}, \quad x \in \Rnn,
\\
S' f(x) & := \bigg(\int_0^{\infty} \int_0^{\infty} |P_{t_1} \otimes Q_{t_2} f(x)|^2 \, \frac{dt_1}{t_1} \, \frac{dt_2}{t_2} \bigg)^{\frac12}, \quad x \in \Rnn,
\\
S'' f(x) & := \bigg(\int_0^{\infty} \int_0^{\infty} |Q_{t_1} \otimes P_{t_2} f(x)|^2 \, \frac{dt_1}{t_1} \, \frac{dt_2}{t_2} \bigg)^{\frac12}, \quad x \in \Rnn.
\end{align*}
We claim that
\begin{align}\label{eq:square}
\|\mathbb{S}f\|_{L^2(\Rnn)} & \lesssim \|f\|_{L^2(\Rnn)}, \quad \forall \mathbb{S} \in \{S, S', S''\}.
\end{align}
Indeed, \eqref{eq:square} is a direct consequence of \cite[Theorem 2.5]{Mar14}. Let us focus on $S'$. The kernel of $P_{t_1} \otimes Q_{t_2}$ is given by $s_{t_1, t_2}(x, y) := \varphi_{t_1}^{(1)}(x_1 - y_1) \psi_{t_2}^{(2)}(x_2 - y_2)$, which along with the smoothness of $\varphi^{(1)}$ and $\psi^{(2)}$ satisfies \cite[Assumption 2.1]{Mar14}. The inequality \eqref{SS22} immediately implies \cite[Assumption 2.2]{Mar14}. Additionally, 
\begin{align*}
P_{t_1} \otimes Q_{t_2} 1(x) = \int_{\R^{n_2}} \psi^{(2)} \, dy_2 = 0, \quad x \in \Rnn,
\end{align*}
which verifies \cite[Assumption 2.3]{Mar14}. Thus, \eqref{eq:square} is true for $S'$. Similarly, \eqref{eq:square} hold for both $S$ and $S''$.

To continue, let $h \in L^2(\Rnn)$. Then by \eqref{eq:square}, \eqref{eq:BMO-Car}, and \eqref{eq:PCar}, we have
\begin{align*}
|\langle \pi_b(f, g), h \rangle|
& = \bigg|\int_{\R^{n_1+1}_+} \int_{\R^{n_2+1}_+} Q_{t_1, t_2}b \cdot P_{t_1, t_2}f \cdot  P_{t_1, t_2}g \cdot
Q_{t_1, t_2}h \, \frac{dx_1 \, dt_1}{t_1} \, \frac{dx_2 \, dt_2}{t_2} \bigg|
\\
& \le \bigg(\int_{\R^{n_1+1}_+} \int_{\R^{n_2+1}_+} |P_{t_1, t_2}f \cdot  P_{t_1, t_2} g|^2 \,
|Q_{t_1, t_2}b|^2 \, \frac{dx_1 \, dt_1}{t_1} \, \frac{dx_2 \, dt_2}{t_2} \bigg)^{\frac12} 
\\
& \quad \times \bigg(\int_{\R^{n_1+1}_+} \int_{\R^{n_2+1}_+} 
|Q_{t_1, t_2}h(x)|^2 \, \frac{dx_1 \, dt_1}{t_1} \, \frac{dx_2 \, dt_2}{t_2} \bigg)^{\frac12} 
\\
& = \|(P_{t_1, t_2}f) (P_{t_1, t_2}g)\|_{L^2(\R^{n_1+1}_+ \times \R^{n_2+1}_+, \, \mu_b)} \|Sh\|_{L^2(\Rnn)}
\\
& \le \|P_{t_1, t_2} f\|_{L^4(\R^{n_1+1}_+ \times \R^{n_2+1}_+, \, \mu_b)} 
\|P_{t_1, t_2} g\|_{L^4(\R^{n_1+1}_+ \times \R^{n_2+1}_+, \, \mu_b)} \|Sh\|_{L^2(\Rnn)}
\\
& \lesssim \|b\|_{\BMO(\Rnn)} \|f\|_{L^4(\Rnn)} \|g\|_{L^4(\Rnn)} \|h\|_{L^2(\Rnn)},
\end{align*}
which shows 
\begin{align}\label{pbdd-1}
\pi_b : L^4(\Rnn) \times L^4(\Rnn) \to L^2(\Rnn) \quad \text{boundedly}.
\end{align}
By duality, there holds
\begin{align}
\label{pbdd-2}
& \pi_b^{1*} : L^2(\Rnn) \times L^4(\Rnn) \to L^{\frac43}(\Rnn) \quad \text{boundedly},
\\
\label{pbdd-3}
& \pi_b^{2*} : L^4(\Rnn) \times L^2(\Rnn) \to L^{\frac43}(\Rnn) \quad \text{boundedly}.
\end{align}
Analogously to \eqref{pbdd-1}, in view of \eqref{eq:square} and Theorem \ref{thm:Carleson}, we obtain
\begin{align}
\label{pbdd-4}
& (\pi_b)^{1*}_1 : L^4(\Rnn) \times L^4(\Rnn) \to L^2(\Rnn) \quad \text{boundedly},
\\
\label{pbdd-5}
& (\pi_b)^{1*}_2 : L^4(\Rnn) \times L^4(\Rnn) \to L^2(\Rnn) \quad \text{boundedly},
\\
\label{pbdd-6}
& (\pi_b)^{2*}_1 : L^4(\Rnn) \times L^4(\Rnn) \to L^2(\Rnn) \quad \text{boundedly},
\\
\label{pbdd-7}
& (\pi_b)^{2*}_2 : L^4(\Rnn) \times L^4(\Rnn) \to L^2(\Rnn) \quad \text{boundedly}.
\end{align}
Then by duality and the fact that $(\pi_b)^{1*, 2*}_{1, 2} = \big((\pi_b)^{1*}_1\big)^{2*}$ and $(\pi_b)^{1*, 2*}_{2, 1} = \big((\pi_b)^{1*}_2\big)^{2*}$, one has
\begin{align}
\label{pbdd-8}
& (\pi_b)^{1*, 2*}_{1, 2} : L^4(\Rnn) \times L^2(\Rnn) \to L^{\frac43}(\Rnn) \quad \text{boundedly},
\\
\label{pbdd-9}
& (\pi_b)^{1*, 2*}_{2, 1} : L^4(\Rnn) \times L^2(\Rnn) \to L^{\frac43}(\Rnn) \quad \text{boundedly}.
\end{align}
Thus, it follows from \eqref{pbdd-1}--\eqref{pbdd-9} and \cite[Proposition 3.6]{LMV20} that $\mathbb{T}(1, 1) \in \BMO(\Rnn)$, where 
\begin{align*}
\mathbb{T} \in \{T, T^{1*}, T^{2*}, T_1^{1*}, T_2^{1*}, T_1^{2*}, T_2^{2*}, T_{1, 2}^{1*, 2*}, T_{2, 1}^{1*, 2*}: T = \pi_b\}.
\end{align*}

Finally, to demonstrate the compactness of $\pi_b$ from $L^4(\Rnn) \times L^4(\Rnn)$ to $L^2(\Rnn)$, it suffices to prove
\begin{align}\label{def:pbzero}
\lim_{j \to \infty} \|\pi_b(f_j, g_j)\|_{L^2(\Rnn)} = 0,
\end{align}
for any sequence $\{(f_j, g_j)\}_j \subset L^4 \times L^4$ such that either $f_j$ converges weakly to 0 or $g_j$ converges weakly to 0. As done above, there holds
\begin{align*}
\|\pi_b(f_j, g_j)\|
\lesssim \|P_{t_1, t_2} f_j\|_{L^4(\R^{n_1+1}_+ \times \R^{n_2+1}_+, \, \mu_b)} 
\|P_{t_1, t_2} g_j\|_{L^4(\R^{n_1+1}_+ \times \R^{n_2+1}_+, \, \mu_b)},
\end{align*}
which along with Theorem \ref{thm:Carleson} gives \eqref{def:pbzero} as desired. 
\qed

\medskip 
In order to complete the proof of Theorem \ref{thm:ppp}, we need some additional estimates. To this end, let us introduce some conceptions. Let $\D = \D^1 \times \D^2$, where $\D^i$ is a dyadic grid on $\R^{n_i}$, $i=1, 2$. We say that a locally integrable function $b: \Rnn \to \C$ belongs to the {\tt dyadic product $\BMO$ space} $\BMO(\D)$ if
\begin{align*}
\|b\|_{\BMO(\D)}
:= \sup_{\Omega} \bigg(\frac{1}{|\Omega|} 
\sum_{\substack{I_1 \times I_2 \subset \Omega \\ I_1 \times I_2 \in \D}}
|\langle b, h_{I_1} \otimes h_{I_2} \rangle|^2\bigg)^{\frac12}
< \infty,
\end{align*}
where the supremum is taken over all open sets $\Omega \subset \Rnn$ with $|\Omega| < \infty$. Then the {\tt product $\BMO$ space} $\BMO(\Rnn)$ is defined as the collection of all locally integrable functions $b$ satisfying 
\begin{align*}
\|b\|_{\BMO(\Rnn)} 
:= \sup_{\D = \D_1 \times \D_2}  \|b\|_{\BMO(\D)} < \infty.
\end{align*}

We say that $b$ belongs to the {\tt dyadic product $\CMO$ space $\CMO(\D)$} if $b \in \BMO(\D)$ and satisfies 
\begin{align*}
\lim_{N \to \infty} \|P_{\D(N)}^{\perp} b\|_{\BMO(\D)} = 0,  
\end{align*}
where $\D(N) := \D_1(N) \times \D_2(N)$,  $P_{\D(N)}^{\perp} b := b - P_{\D(N)} b$, and 
\begin{align*}
P_{\D(N)} b
:= \sum_{J_1 \times J_2 \in \D(N)} 
\langle b, h_{J_1} \otimes h_{J_2} \rangle \, h_{J_1} \otimes h_{J_2}.
\end{align*}
The {\tt product $\CMO$ space $\CMO(\Rnn)$} is defined as the set of all functions $b \in \BMO(\Rnn)$ satisfying 
\begin{align*}
\lim_{N \to \infty} \sup_{\D} \|P_{\D(N)}^{\perp} b\|_{\BMO(\D)} = 0. 
\end{align*}

Given a measure $\mu$ on $\R^{n_1+1}_+ \times \R^{n_2+1}_+$, we say that $\mu$ is a {\tt Carleson measure} if 
\begin{equation*}
\|\mu\|_{\mathrm{Car}(\R^{n_1+1}_+ \times \R^{n_2+1}_+)}
:= \sup_{\D} \sup_{\Omega} \frac{1}{|\Omega|} 
\sum_{\substack{I_1 \times I_2 \subset \Omega \\ I_1 \times I_2 \in \D}} 
\iint_{W_{I_2}} \iint_{W_{I_1}} \, d\mu(x, t) 
< \infty,
\end{equation*}
where the first supremum is taken over all dyadic grids $\D = \D_1 \times \D_2$ with $\D_i$ being a dyadic grid in $\R^{n_i}$, and the second supremum is taken over all open sets $\Omega \subset \Rnn$ such that $|\Omega| < \infty$ and such that for every $x \in \Omega$ there exists $I_1 \times I_2 \in \D$ so that $x \in I_1 \times I_2 \subset \Omega$, and $W_I = I \times [\ell(I)/2, \ell(I))$ is the Whitney region associated with a cube $I$. Similarly, we say that $\mu$ is a {\tt vanishing Carleson measure} if 
\begin{equation*}
\lim_{N \to \infty} \sup_{\D} \sup_{\Omega} \frac{1}{|\Omega|}
\sum_{\substack{I_1 \times I_2 \subset \Omega \\ I_1 \times I_2 \notin \D(N)}} 
\iint_{W_{I_2}} \iint_{W_{I_1}}  d\mu(x, t) 
= 0.
\end{equation*}

In what follows, given $b \in L^1_{\loc}(\Rnn)$, define a measure $\mu_b$ on $\R^{n_1+1}_+ \times \R^{n_2+1}_+$ as
\begin{align*}
d\mu_b(x, t) := |Q_{t_1, t_2} b(x)|^2 \, \frac{dx_1 \, dt_1}{t_1} \, \frac{dx_2 \, dt_2}{t_2}.
\end{align*}
For each $i=1, 2$, let $\theta^{(i)}$ be a function on $\R^{n_i}$ satisfying 
\begin{align*}
|\theta^{(i)}(x_i)| 
\lesssim (1 + |x_i|)^{-n_i - \alpha_i}, \quad \text{for some } \alpha_i > 0.
\end{align*}
Denote 
\begin{align*}
\theta_{t_i}^{(i)}(x_i) := t_i^{-n_i} \theta_{t_i}^{(i)}(t_i^{-1} x_i) 
\quad\text{ and } \quad 
\Theta_{t_1, t_2} f := (\theta_{t_1}^{(1)} \otimes \theta_{t_2}^{(2)})*f.
\end{align*}

\begin{theorem}\label{thm:Carleson}
The following statements hold:
\begin{list}{\rm (\theenumi)}{\usecounter{enumi}\leftmargin=1.3cm \labelwidth=1cm \itemsep=0.1cm \topsep=0.2cm \renewcommand{\theenumi}{\alph{enumi}}}
\item\label{Car-1} For any $b \in \BMO(\Rnn)$, $\mu_b$ is a Carleson measure on $\R^{n_1+1}_+ \times \R^{n_2+1}_+$.
 
\item\label{Car-2} For any $b \in \CMO(\Rnn)$, $\mu_b$ is a vanishing Carleson measure on $\R^{n_1+1}_+ \times \R^{n_2+1}_+$.

\item\label{Car-3} For any Carleson measure $\mu$ on $\R^{n_1+1}_+ \times \R^{n_2+1}_+$, $\Theta_{t_1, t_2}$ is bounded from $L^p(\Rnn)$ to $L^p(\R^{n_1+1}_+ \times \R^{n_2+1}_+, \mu)$ for all $p \in (1, \infty)$.

\item\label{Car-4} For any vanishing Carleson measure $\mu$ on $\R^{n_1+1}_+ \times \R^{n_2+1}_+$, $\Theta_{t_1, t_2}$ is compact from $L^p(\Rnn)$ to $L^p(\R^{n_1+1}_+ \times \R^{n_2+1}_+, \mu)$ for all $p \in (1, \infty)$.
\end{list}
\end{theorem}

\begin{proof}
Before starting the proof, we present some useful estimates. A combination of \cite[Lemma 2.7]{Hyt14} and \cite[Lemma 2.5]{HLP} yields the following:
\begin{equation}\label{eq:DD}
\begin{aligned}
& \text{there are dyadic grids $\{\D_i^{\alpha_i}\}_{\alpha_i = 1}^{3^{n_i}}$, with the following properties:} 
\\
& \text{for any cube $J_i \subset \R^{n_i}$ and $k \in \mathbb{N}$ one can find a dyadic cube}
\\
& \text{$J_i^{\alpha_i} \in \D_i^{\alpha_i}$ so that $J_i \subset J_i^{\alpha_i}$, $2^k J_i \subset (J_i^{\alpha_i})^{(k)}$, 
and $3 \ell(J_i) < \ell(J_i^{\alpha_i}) \le 6 \ell(J_i)$}.
\end{aligned}
\end{equation}
This gives
\begin{align}\label{eq:WIWI}
W_{I_i} = I_i \times [\ell(I_i)/2, \ell(I_i))
\subset I_i^{\alpha_i} \times [\ell(I_i^{\alpha_i})/12, \ell(I_i^{\alpha_i})/3) =: \widetilde{W}_{I_i^{\alpha_i}}.
\end{align}

Given $J_i \in \D_i$ and $I_i \in \D_i^{\alpha_i}$, observe that 
\begin{align}\label{eq:WQ}
(x_i, t_i) \in W_{J_i} \, \text{ and  }\,
Q_{t_i}^{(i)} h_{I_i}(x_i) \ne 0 
\quad \Longrightarrow \quad I_i \subset (J_i^{\alpha_i})^{(2)}.
\end{align}
Indeed, it is easy to see that $\supp h_{I_i} \subset I_i \in \D_i^{\alpha_i}$ and $\supp \psi_{t_i}^{(i)}(x_i - \cdot) \subset B(x_i, t_i) \subset 3J_i \subset (J_i^{\alpha_i})^{(2)} \in \D_i^{\alpha_i}$ by \eqref{eq:DD} with $k=2$. Thus, if $I_i \cap (J_i^{\alpha_i})^{(2)} = \emptyset$, there holds $Q_{t_i}^{(i)} h_{I_i}(x_i) = 0$. If $(J_i^{\alpha_i})^{(2)} \subsetneq I_i$, then $h_{I_i}$ is constant on $B(x_i, t_i)$, which along with the cancellation property of $\psi^{(i)}$ gives $Q_{t_i}^{(i)} h_{I_i}(x_i) = 0$.

Setting
\begin{align*}
\D^{\alpha} := \D_1^{\alpha_1} \times \D_2^{\alpha_2} 
\quad \text{and} \quad
\Omega^{\alpha}
:= \bigcup_{\substack{I_1 \times I_2 \subset \Omega \\ I_1^{\alpha_1} \times I_2^{\alpha_2} \in \D^{\alpha}}} 
(I_1^{\alpha_1})^{(2)} \times (I_2^{\alpha_2})^{(2)},
\end{align*}
we have $\Omega^{\alpha} \subset \{M_{\mathcal{R}} \mathbf{1}_{\Omega} \ge 4^{-n_1-n_2}\}$, which together with the boundedness of $M_{\mathcal{R}}$ yields 
\begin{align}\label{eq:OmOm}
|\Omega^{\alpha}|
\lesssim \|M_{\mathcal{R}} \mathbf{1}_{\Omega}\|_{L^2(\Rnn)}^2
\lesssim |\Omega|.
\end{align}

To show item \eqref{Car-1}, assume that $b \in \BMO(\Rnn)$. Define
\begin{align*}
P_{\D}^{\Omega} b
:= \sum_{\substack{I_1 \times I_2 \subset \Omega \\ I_1 \times I_2 \in \D}}
\langle b, h_{I_1} \otimes h_{I_2} \rangle h_{I_1} \otimes h_{I_2}.
\end{align*}
It follows from \eqref{eq:WQ} that
\begin{align*}
Q_{t_1, t_2} b(x)
& = \sum_{I_1 \times I_2 \in \D^{\alpha}}
\langle b, h_{I_1} \otimes h_{I_2}  \rangle \, Q_{t_1}^{(1)} h_{I_1} \otimes Q_{t_2}^{(2)} h_{I_2}(x)
\\
& = \sum_{\substack{I_1 \times I_2 \subset \Omega^{\alpha} \\ I_1 \times I_2 \in \D^{\alpha}}}
\langle b, h_{I_1} \otimes h_{I_2}  \rangle \, Q_{t_1}^{(1)} h_{I_1} \otimes Q_{t_2}^{(2)} h_{I_2}(x)
= Q_{t_1, t_2} (P^{\Omega^{\alpha}}_{\D^{\alpha}} b)(x),
\end{align*}
which along with \eqref{eq:square}, \eqref{eq:WIWI}, and \eqref{eq:OmOm} leads to
\begin{align*}
& \sum_{\substack{J_1 \times J_2 \subset \Omega \\ J_1 \times J_2 \in \D}} 
\iint_{W_{J_2}} \iint_{W_{J_1}}  d\mu_b(x, t) 
\le \sum_{\alpha_1 = 1}^{3n_1} \sum_{\alpha_2 = 1}^{3n_2} 
\sum_{\substack{J_1 \times J_2 \subset \Omega \\ J_1 \times J_2 \in \D}} 
\iint_{W_{J_2} \cap \widetilde{W}_{J_2^{\alpha_2}}} \iint_{W_{J_1} \cap \widetilde{W}_{J_1^{\alpha_1}}}  d\mu_b(x, t) 
\\
& \le \sum_{\alpha_1 = 1}^{3n_1} \sum_{\alpha_2 = 1}^{3n_2} 
\sum_{J_1 \times J_2 \in \D} 
\iint_{W_{J_2}} \iint_{W_{J_1}} |Q_{t_1, t_2} (P_{\D^{\alpha}}^{\Omega^{\alpha}} b)(x)|^2 \, 
\frac{dx_1 \, dt_1}{t_1} \, \frac{dx_2 \, dt_2}{t_2} 
\\
&\le \sum_{\alpha_1 = 1}^{3n_1} \sum_{\alpha_2 = 1}^{3n_2}  \|S(P_{\D^{\alpha}}^{\Omega^{\alpha}} b)\|_{L^2(\Rnn)}^2 
\lesssim \sum_{\alpha_1 = 1}^{3n_1} \sum_{\alpha_2 = 1}^{3n_2}  \|P_{\D^{\alpha}}^{\Omega^{\alpha}} b\|_{L^2(\Rnn)}^2 
\\
& = \sum_{\alpha_1 = 1}^{3n_1} \sum_{\alpha_2 = 1}^{3n_2}  
\sum_{\substack{I_1 \times I_2 \subset \Omega^{\alpha} \\ I_1 \times I_2 \in \D^{\alpha}}} 
|\langle b, h_{I_1} \otimes h_{I_2}  \rangle|^2
\lesssim \sum_{\alpha_1 = 1}^{3n_1} \sum_{\alpha_2 = 1}^{3n_2}  |\Omega^{\alpha}| \|b\|_{\BMO(\Rnn)}^2
\\
& \lesssim |\Omega| \|b\|_{\BMO(\Rnn)}^2.
\end{align*}
This shows
\begin{align}\label{eq:BMO-Car}
\|\mu_b\|_{\mathrm{Car}(\R^{n_1+1}_+ \times \R^{n_2+1}_+)} \lesssim \|b\|_{\BMO(\Rnn)}^2.
\end{align}

To prove item \eqref{Car-2}, let $N \ge 4$, $b \in \CMO(\Rnn)$, and $\Omega \subset \Rnn$ be an open set. Let $\lambda \ge 4$ be an absolute constant chosen later. As above, we have
\begin{align*}
\mathscr{F}_N
& := \sum_{\substack{J_1 \times J_2 \notin \D(\lambda N) \\ J_1 \times J_2 \subset \Omega}} 
\iint_{W_{J_2}} \iint_{W_{J_1}} |Q_{t_1, t_2} b(x)|^2 \, \frac{dx_1 \, dt_1}{t_1} \, \frac{dx_2 \, dt_2}{t_2}
\\
& \le \sum_{\alpha_1 = 1}^{3n_1} \sum_{\alpha_2 = 1}^{3n_2} 
\sum_{\substack{J_1 \times J_2 \notin \D(\lambda N) \\ J_1 \times J_2 \subset \Omega}} 
\iint_{W_{J_2}} \iint_{W_{J_1}} \bigg| \sum_{I_1 \times I_2 \in \D^{\alpha}} b_{I_1, I_2}
Q_{t_1, t_2} h_{I_1} \otimes h_{I_2}(x)|^2 \, \frac{dx_1 \, dt_1}{t_1} \, \frac{dx_2 \, dt_2}{t_2},
\end{align*}
where $b_{I_1, I_2} := \langle b, h_{I_1} \otimes h_{I_2} \rangle \mathbf{1}_{\{I_1 \times I_2 \subset \Omega^{\alpha}\}}$. We split 
\begin{align}\label{FNNN}
\mathscr{F}_N
\le \mathscr{F}_N^1 + \mathscr{F}_N^2 + \mathscr{F}_N^3,
\end{align}
where
\begin{align*}
\mathscr{F}_N^1
& := \sum_{\alpha_1 = 1}^{3n_1} \sum_{\alpha_2 = 1}^{3n_2} \sum_{J_1 \times J_2 \in \D} 
\iint_{W_{J_2}} \iint_{W_{J_1}} \bigg| \sum_{I_1 \times I_2 \notin \D^{\alpha}(N)} b_{I_1, I_2}
Q_{t_1, t_2} h_{I_1} \otimes h_{I_2}(x) \bigg|^2 \, \frac{dx_1 \, dt_1}{t_1} \, \frac{dx_2 \, dt_2}{t_2},
\\
\mathscr{F}_N^2
& := \sum_{\alpha_1 = 1}^{3n_1} \sum_{\alpha_2 = 1}^{3n_2} 
\sum_{\substack{J_1 \notin \D_1(\lambda N) \\ J_1 \times J_2 \subset \Omega}} 
\iint_{W_{J_2}} \iint_{W_{J_1}} \bigg| \sum_{I_1 \times I_2 \in \D^{\alpha}(N)} b_{I_1, I_2}
Q_{t_1, t_2} h_{I_1} \otimes h_{I_2}(x) \bigg|^2 \, \frac{dx_1 \, dt_1}{t_1} \, \frac{dx_2 \, dt_2}{t_2},
\\
\mathscr{F}_N^3
& := \sum_{\alpha_1 = 1}^{3n_1} \sum_{\alpha_2 = 1}^{3n_2} 
\sum_{\substack{J_2 \notin \D_2(\lambda N) \\ J_1 \times J_2 \subset \Omega}} 
\iint_{W_{J_2}} \iint_{W_{J_1}} \bigg| \sum_{I_1 \times I_2 \in \D^{\alpha}(N)} b_{I_1, I_2}
Q_{t_1, t_2} h_{I_1} \otimes h_{I_2}(x) \bigg|^2 \, \frac{dx_1 \, dt_1}{t_1} \, \frac{dx_2 \, dt_2}{t_2}.
\end{align*}
The inequality \eqref{eq:square} implies
\begin{align}\label{FNN-1}
\mathscr{F}_N^1
& \le \sum_{\alpha_1 = 1}^{3n_1} \sum_{\alpha_2 = 1}^{3n_2} \|S(P_{{\D^{\alpha}(N)}^c}^{\Omega^{\alpha}} b)\|_{L^2(\Rnn)}^2
\\ \nonumber
& \lesssim \sum_{\alpha_1 = 1}^{3n_1} \sum_{\alpha_2 = 1}^{3n_2} \|P_{{\D^{\alpha}(N)}^c}^{\Omega^{\alpha}} b\|_{L^2(\Rnn)}^2
\\ \nonumber
& = \sum_{\alpha_1 = 1}^{3n_1} \sum_{\alpha_2 = 1}^{3n_2} 
 \sum_{\substack{I_1 \times I_2 \subset \Omega^{\alpha} \\ I_1 \times I_2 \notin \D^{\alpha}(N)}}
|\langle b, h_{I_1} \otimes h_{I_2}  \rangle|^2.
\end{align}
To estimate $\mathscr{F}_N^2$, observe that
\begin{align}\label{eq:JD8N}
J_1 \not\in \D_1(\lambda N) 
\quad \Longrightarrow \quad 
\ell(J_1) > 2^{\lambda N}.
\end{align}
Indeed, if $\ell(J_1) < 2^{-\lambda N}$, then by \eqref{eq:DD} and \eqref{eq:WQ}, 
\begin{align*}
\ell(I_1) \le 4 \ell(J_1^{\alpha_1}) \le 24 \ell(J_1) \le 24 \cdot 2^{- \lambda N} < 2^{-N},
\end{align*} 
which contradicts the fact $I_1 \in \D_1^{\alpha_1}(N)$. If $2^{-\lambda N} \le \ell(J_1) \le 2^{\lambda N}$ and $\rd(J_1, 2^{\lambda N} \mathbb{I}_1) > \lambda N$, then $\d(J_1, 2^{\lambda N} \mathbb{I}_1) > \lambda N \cdot 2^{\lambda N}$ and 
\begin{align*}
\rd(I_1, 2^N \mathbb{I}_1) 
& = \frac{\d(I_1, 2^N \mathbb{I}_1)}{\max\{\ell(I_1), 2^N\}}
\ge \frac{\d((J_1^{\alpha_1})^{(2)}, 2^N \mathbb{I}_1)}{\max\{\ell((J_1^{\alpha_1})^{(2)}), 2^N\}} 
\\
& \ge \frac{\d(J_1, 2^{\lambda N} \mathbb{I}_1) - 4 \ell(J_1^{\alpha_1})}{\max\{4 \ell(J_1^{\alpha_1}), 2^N\}} 
\ge \frac{\lambda N 2^{\lambda N} - 24 \cdot 2^{\lambda N}}{24 \cdot 2^{\lambda N}}
> N,
\end{align*}
provided $\lambda \ge 48$. This also contradicts the fact $I_1 \in \D_1^{\alpha_1}(N)$. Thus, there must hold $\ell(J_1) > 2^{\lambda N}$.

By the cancellation of $h_{I_i}$ and the decay property of $\psi_{t_i}^{(i)}$, we obtain
\begin{align}\label{eq:Qtth}
|Q_{t_1, t_2} h_{I_1} \otimes h_{I_2}(x)|
& = \bigg|\prod_{i=1}^2 \int_{\R^{n_i}} \big[\psi_{t_i}^{(i)}(x_i - y_i) - \psi_{t_i}^{(i)}(x_i - c_{I_i}) \big] h_{I_i}(y_i) \, dy_i \bigg|
\\ \nonumber
& \lesssim \prod_{i=1}^2 \int_{I_i} \frac{1}{t_i^{n_i+1}} \frac{|y_i - c_{I_i}|}{(1+|x_i - y_i|/t_i)^{n_i+1}} |I_i|^{-\frac12} \, dy_i
\\ \nonumber
& \lesssim \prod_{i=1}^2 \frac{\ell(I_i)}{(t_i + |x_i - y_i|)^{n_i+1}} |I_i|^{\frac12}
\\ \nonumber
& \lesssim \prod_{i=1}^2 \frac{\ell(I_i)^{\frac12} \ell(J_i)^{\frac12}}{\big[\ell(I_i) + \ell(J_i) + \d(I_i, J_i) \big]^{n_i+1}} |I_i|^{\frac12}.
\end{align}
Denote 
\begin{align*}
A_{I_i, J_i}
:= \frac{\ell(I_i)^{\frac12} \ell(J_i)^{\frac12}}{\big[\ell(I_i) + \ell(J_i) + \d(I_i, J_i) \big]^{n_i+1}} |I_i|^{\frac12} |J_i|^{\frac12}.
\end{align*}
Examining the proof of \cite[Proposition 6.3]{HM}, one can show
\begin{align*}
\sum_{\substack{I_i \in \D_i, J_i \in \D'_i \\ \ell(J_i) \ge 2^{\kappa} \ell(I_i)}} A_{I_i, J_i} \, x_{I_i} \, y_{J_i} 
\lesssim 2^{-\frac{\kappa}{2}} \bigg(\sum_{I_i \in \D_i} x_{I_i}^2 \bigg)^{\frac12} 
\bigg(\sum_{J_i \in \D'_i} y_{J_i}^2 \bigg)^{\frac12}, \quad \forall \kappa \ge 0,
\end{align*}
where the implicit constant is independent of $\kappa$. The above immediately gives
\begin{align}\label{eq:AIJ}
\sum_{J_i \in \D'_i} \bigg[\sum_{\substack{I_i \in \D_i \\ 2^{\kappa} \ell(I_i) \le \ell(J_i)}} A_{I_i, J_i} \, x_{I_i} \bigg]^2
\lesssim 2^{-\kappa} \sum_{I_i \in \D_i} x_{I_i}^2, \quad \forall \kappa \ge 0.
\end{align}
Now invoking \eqref{eq:JD8N}--\eqref{eq:AIJ}, we conclude
\begin{align*}
\mathscr{F}_N^2
& \lesssim \sum_{\alpha_1, \alpha_2} \sum_{J_2 \in \D_2} \sum_{J_1 \in \D_1}
\bigg[ \sum_{\substack{I_1 \in \D_1^{\alpha_1} \\ 2^N \ell(I_1) \le \ell(J_1)}} 
A_{I_1, J_1} \sum_{I_2 \in \D_2^{\alpha_2}} A_{I_2, J_2} b_{I_1, I_2} \bigg]^2 
\\
& \lesssim 2^{-N} \sum_{\alpha_1, \alpha_2}
\sum_{I_1 \in \D_1^{\alpha_1}} \sum_{J_2 \in \D_2} 
\bigg[ \sum_{I_2 \in \D_2^{\alpha_2}} A_{I_2, J_2} b_{I_1, I_2} \bigg]^2 
\\
& \lesssim 2^{-N} \sum_{\alpha_1, \alpha_2}
\sum_{I_1 \times I_2 \in \D^{\alpha}} |b_{I_1, I_2}|^2 
\\
& = 2^{-N} \sum_{\alpha_1, \alpha_2}
\sum_{\substack{I_1 \times I_2 \subset \Omega^{\alpha} \\ I_1 \times I_2 \in \D^{\alpha}}} 
|\langle b, h_{I_1} \otimes h_{I_2} \rangle|^2
\\
& \le 2^{-N} \sum_{\alpha_1, \alpha_2} |\Omega^{\alpha}| \|b\|_{\BMO(\Rnn)}^2,
\end{align*}
which along with \eqref{eq:OmOm} gives
\begin{align}\label{FNN-2}
\mathscr{F}_N^2
\lesssim 2^{-N} |\Omega| \|b\|_{\BMO(\Rnn)}^2.
\end{align}
Symmetrically, there holds
\begin{align}\label{FNN-3}
\mathscr{F}_N^3
& \lesssim 2^{-N} |\Omega| \|b\|_{\BMO(\Rnn)}^2.
\end{align}
Consequently, it follows from \eqref{eq:OmOm}--\eqref{FNN-1}, \eqref{FNN-2} and \eqref{FNN-3} that 
\begin{align*}
\text{$\mu_b$ is a vanishing Carleson measure on $\R^{n_1+1}_+ \times \R^{n_2+1}_+$}.
\end{align*}

Next, let us demonstrate item \eqref{Car-3}. It is easy to check that 
\begin{align}\label{def:Nf}
\mathcal{N} f(x) := \sup_{\substack{|y_1 - x_1| < 2t_1 \\ |y_2 - x_2| < 2t_2}} |\Theta_{t_1, t_2} f(y)| 
\lesssim M_{\mathcal{R}} f(x), \quad x \in \Rnn. 
\end{align}
Denote 
\begin{align*}
E_s &:= \{((x_1, t_1), (x_2, t_2)) \in \R^{n_1+1}_+ \times \R^{n_2+1}_+: |\Theta_{t_1, t_2} f(x)| > s\}, 
\\
\Omega_s &:= \{x \in \Rnn: \mathcal{N} f(x) > s\}.
\end{align*}
We claim 
\begin{align}\label{EaOa}
E_s \subset \bigcup_{I_1 \times I_2 \subset \Omega_s} W_{I_1} \times W_{I_2}.
\end{align}
Indeed, for any $((x_1, t_1), (x_2, t_2)) \in E_s$, there holds $\sup_{\substack{|z_1 - x_1| < t_1 \\ |z_2 - x_2| < t_2}} |\Theta_{t_1, t_2} f(z)| \ge |\Theta_{t_1, t_2} f(x)| > s$. For each $i= 1, 2$, let $I_i = I_i(x_i, t_i)$ be the cube centered at $x_i$ with sidelength $\ell(I_i) = 2t_i$. Then $((x_1, t_1), (x_2, t_2)) \in W_{I_1} \times W_{I_2}$. Note that for any $y \in I_1 \times I_2$, we have $|y_1 - x_1| < t_1$, $|y_2 - x_2| < t_2$, and
\begin{align*}
\mathcal{N} f(y) 
= \sup_{\substack{|z_1 - y_1| < 2s_1 \\ |z_2 - y_2| < 2s_2}} |\Theta_{s_1, s_2} f(z)|
\ge \sup_{\substack{|z_1 - y_1| < 2t_1 \\ |z_2 - y_2| < 2t_2}} |\Theta_{t_1, t_2} f(z)|
\ge \sup_{\substack{|z_1 - x_1| < t_1 \\ |z_2 - x_2| < t_2}} |\Theta_{t_1, t_2} f(z)|
> s,
\end{align*}
which means $I_1 \times I_2 \subset \Omega_s$. Thus, \eqref{EaOa} holds.

It follows from \eqref{def:Nf}, and \eqref{EaOa} that
\begin{align}\label{eq:PCar}
\|\Theta_{t_1, t_2} f\|_{L^p(\mu)}^p 
& = p \int_0^{\infty} s^{p-1} \mu(E_s) \, ds
\\ \nonumber
& \le p \int_0^{\infty} s^{p-1} \mu \Big(\bigcup_{I_1 \times I_2 \subset \Omega_s} W_{I_1} \times W_{I_2} \Big) \, ds
\\ \nonumber
& \lesssim p \|\mu\|_{\mathrm{Car}(\R^{n_1+1}_+ \times \R^{n_2+1}_+)} \int_0^{\infty} s^{p-1} |\Omega_s| \, ds
\\ \nonumber
& = \|\mu\|_{\mathrm{Car}(\R^{n_1+1}_+ \times \R^{n_2+1}_+)} \|\mathcal{N} f\|_{L^p}^p 
\\ \nonumber
& \lesssim \|\mu\|_{\mathrm{Car}(\R^{n_1+1}_+ \times \R^{n_2+1}_+)} \|M_{\mathcal{R}} f\|_{L^p}^p 
\\ \nonumber
& \lesssim \|\mu\|_{\mathrm{Car}(\R^{n_1+1}_+ \times \R^{n_2+1}_+)} \|f\|_{L^p}^p.
\end{align}
This verifies item \eqref{Car-3}.

Finally, let us turn to proving item \eqref{Car-4}. Given topological vector spaces $X$ and $Y$, let $\mathscr{B}(X, Y)$ denote the collection of all bounded linear operators from $X$ to $Y$. It was proved in \cite[p. 113]{Rudin} that if $X$ is reflexive, then $T \in \mathscr{B}(X, Y)$ is compact if and only if $\lim\limits_{k \to \infty} \|Tx_k\|_Y = 0$ whenever $x_k$ weakly converges to zero in $X$.

Let $\{f_j\}$ be a sequence weakly converges to zero in $L^p(\Rnn)$. To obtain the compactness of $\Theta_{t_1, t_2}$, it suffices to show 
\begin{align}\label{Pttcpt}
\lim_{j \to \infty} \|\Theta_{t_1, t_2} f_j\|_{L^p(\R^{n_1+1}_+ \times \R^{n_2+1}_+, \, \mu)} = 0.
\end{align}
By the uniform boundedness principle, the weak convergence of $\{f_j\}$ in $L^p(\Rnn)$ implies 
\begin{align}\label{fkLp}
\sup_j \|f_j\|_{L^p(\Rnn)} \le C_0 < \infty,
\end{align}
which yileds
\begin{align}\label{kptt}
\sup_j \|\Theta_{t_1, t_2} f_j\|_{L^{\infty}(\Rnn)} 
\le \big\|\theta_{t_1}^{(1)} \otimes \theta_{t_2}^{(2)} \big\|_{L^{p'}} \sup_j \|f_j\|_{L^p}
\lesssim t_1^{-\frac{n_1}{p}} t_2^{-\frac{n_2}{p}}.
\end{align}
Let $\varepsilon > 0$ and $\D$ be a dyadic grid. Since $\mu$ is a vanishing Carleson measure on $\R^{n_1+1}_+ \times \R^{n_2+1}_+$, there exists an increasing sequence $\{N_k\}_{k \ge 1}$ such that
\begin{align}\label{VCE}
\sup_{\Omega} \frac{1}{|\Omega|} \sum_{\substack{J_1 \times J_2 \subset \Omega \\ J_1 \times J_2 \notin \D(N_k)}} 
\iint_{W_{J_2}} \iint_{W_{J_1}}  \, d\mu(x, t)
\le 2^{-k}.
\end{align}
In addition, there exists some $k_0 \ge 1$ so that $\sum_{k>k_0} 2^{-k} \le \varepsilon$. Then we perform the decomposition
\begin{align}\label{PPP}
\|\Theta_{t_1, t_2} f_j\|_{L^p(\mu)}^p
\le \sum_{k \ge k_0} \mathscr{P}_j^k,
\end{align}
where
\begin{align*}
\mathscr{P}_j^{k_0}
& := \sum_{J_1 \times J_2 \in \D(N_{k_0})} 
\iint_{W_{J_2}} \iint_{W_{J_1}} |\Theta_{t_1, t_2} f_j(x)|^p \, d\mu(x, t),
\\
\mathscr{P}_j^k
& := \sum_{J_1 \times J_2 \in \D(N_{k+1}) \setminus \D(N_k)} 
\iint_{W_{J_2}} \iint_{W_{J_1}} |\Theta_{t_1, t_2} f_j(x)|^p \, d\mu(x, t),
\end{align*}
for each $k > k_0$. Set 
\begin{align*}
\Omega_N := \bigcup_{I_1 \times I_2 \in \D(N)} I_1 \times I_2 
\subset \big\{|x_1|, |x_2| \le (N+2) 2^N \big\}.
\end{align*}
It light of \eqref{kptt}, there holds
\begin{align*}
\mathscr{P}_j^{k_0}
\lesssim 2^{N_{k_0}(n_1+n_2)} |\Omega_{N_{k_0}}| \|\mu\|_{\mathrm{Car}(\R^{n_1+1}_+ \times \R^{n_2+1}_+)}
\lesssim 2^{3N_{k_0}(n_1+n_2)} \|\mu\|_{\mathrm{Car}(\R^{n_1+1}_+ \times \R^{n_2+1}_+)}.
\end{align*}
Thus, by the Lebesgue dominated convergence theorem, 
\begin{align}\label{Pjk0}
\lim_{j \to \infty} \mathscr{P}_j^{k_0} = 0.
\end{align}
For each $k > k_0$, if we define
\begin{align*}
\mu_k 
:= \mu \big|_{\bigcup_{J_1 \times J_2 \in \D(N_{k+1}) \setminus \D(N_k)} W_{J_1} \times W_{J_2}},
\end{align*}
then \eqref{VCE} gives
\begin{align*}
\|\mu_k\|_{\mathrm{Car}(\R^{n_1+1}_+ \times \R^{n_2+1}_+)}
\le 2^{-k},
\end{align*}
which together with \eqref{eq:PCar} and \eqref{fkLp} implies that
\begin{align}\label{Pjk}
\mathscr{P}_j^k
= \|\Theta_{t_1, t_2} f_j\|_{L^p(\R^{n_1+1}_+ \times \R^{n_2+1}_+, \mu_k)}^p
\lesssim \|\mu_k\|_{\mathrm{Car}(\R^{n_1+1}_+ \times \R^{n_2+1}_+)}
\le 2^{-k}.
\end{align}
Hence, \eqref{Pttcpt} is a consequence of \eqref{PPP}--\eqref{Pjk}.
\end{proof}

\subsection{Bi-parameter pseudo-differential operators}
Finally, in order to obtain Theorem \ref{thm:Tsig}, we establish the result below. 

\begin{theorem}\label{thm:sss}
Let $\sigma \in \mathcal{S}_{\rho, \delta}^m(\Rnn)$, where $m = (0, 0)$, $\rho = (1, 1)$, and $\delta = (\delta_1, \delta_2) \in [0, 1)^2$. Then the following hold:
\begin{list}{\rm (\theenumi)}{\usecounter{enumi}\leftmargin=1.3cm \labelwidth=1cm \itemsep=0.1cm \topsep=0.2cm \renewcommand{\theenumi}{\alph{enumi}}}
\item\label{sss-1} $T_{\sigma}$ admits a full kernel representation;

\item\label{sss-2} $T_{\sigma}$ admits a partial kernel representation;

\item\label{sss-3} $T_{\sigma}$ is bounded from $L^4(\Rnn) \times L^4(\Rnn)$ to $L^2(\Rnn)$;

\item\label{sss-4} The symbol of $\mathbb{T}$ belongs to $\mathcal{S}_{\rho, \delta}^m(\Rnn)$, where 
\begin{align*}
\mathbb{T} \in \{T_{\sigma}^{1*}, T_{\sigma}^{2*}, T_{\sigma, 1}^{1*}, T_{\sigma, 2}^{1*}, T_{\sigma, 1}^{2*}, T_{\sigma, 2}^{2*}, T_{\sigma, 1, 2}^{1*, 2*}, T_{\sigma, 2, 1}^{1*, 2*}\};
\end{align*}

\item\label{sss-5} If $\sigma \in \mathcal{K}_{\rho, \delta}^m(\Rnn)$, then $T_{\sigma}$ is compact from $L^4(\Rnn) \times L^4(\Rnn)$ to $L^2(\Rnn)$.
\end{list}
\end{theorem}

Assuming Theorem \ref{thm:sss} holds momentarily, we justify Theorem \ref{thm:Tsig} as follows. By parts \eqref{sss-1}--\eqref{sss-3} and \cite[Proposition 3.6]{LMV20}, we have 
\begin{align}\label{TsigBMO-1}
T_{\sigma}(1, 1) \in \BMO(\Rnn).
\end{align}
In the same way, we use part \eqref{sss-4} to deduce 
\begin{align}\label{TsigBMO-2}
\mathbb{T}(1, 1) \in \BMO(\Rnn),
\end{align}
where $\mathbb{T} \in \{T_{\sigma}^{1*}, T_{\sigma}^{2*}, T_{\sigma, 1}^{1*}, T_{\sigma, 2}^{1*}, T_{\sigma, 1}^{2*}, T_{\sigma, 2}^{2*}, T_{\sigma, 1, 2}^{1*, 2*}, T_{\sigma, 2, 1}^{1*, 2*}\}$. Thus, item \eqref{Tsig-1} follows from parts \eqref{sss-1}--\eqref{sss-3}, \eqref{TsigBMO-1}--\eqref{TsigBMO-2}, and \cite[Theorem 1.2]{LMV21}. In addition, item \eqref{Tsig-2} is a consequence of item \eqref{Tsig-1}, part \eqref{sss-5}, and \cite[Theorem 1.6]{CY}, while item \eqref{Tsig-3} follows from item \eqref{Tsig-1}, part \eqref{sss-5}, and Theorem \ref{thm:RdF-cpt}. This completes the proof of Theorem \ref{thm:Tsig}.

\medskip
\noindent{\bf Proof of Theorem \ref{thm:sss}.} For any $\sigma \in \mathcal{S}_{\rho,\delta}^m(\Rnn)$, we define the seminorm 
\begin{align*}
P_{\alpha, \beta, \gamma}^{m, \rho, \delta}(\sigma) 
:= \sup_{x, \, \xi, \, \eta \in \Rnn}  
\frac{\big|\partial_{x_1}^{\alpha_1} \partial_{x_2}^{\alpha_2} \partial_{\xi_1}^{\beta_1} \partial_{\xi_2}^{\beta_2} \partial_{\eta_1}^{\gamma_1} \partial_{\eta_2}^{\gamma_2} \sigma(x, \xi, \eta) \big|}{\prod_{i=1}^2 (1+ |\xi_i| + |\eta_i|)^{m_i + \delta_i |\alpha_i| - \rho_i (|\beta_i| + |\gamma_i|)}}, 
\end{align*}
for all multi-indices $\alpha$, $\beta$ and $\gamma$. If we denote 
\begin{align*}
K_{\sigma}(x, y, z) 
:= \int_{\Rnn} \int_{\Rnn} \sigma(x, \xi, \eta) e^{2\pi i [(x-y) \cdot \xi + (x-z) \cdot \eta]} \, d\xi \, d\eta,
\end{align*}
then the full kernel representation is a direct consequence of Lemma \ref{lem:KP} below.

To justify the partial kernel representation, set
\begin{align*}
K_{f_2, g_2, h_2}(x_1, y_1, z_1) 
& := \int_{\R^{n_2}} \int_{\R^{n_2}} \int_{\R^{n_2}} K_{\sigma}(x, y, z) f_2(y_2) \, g_2(z_2) \, h_2(x_2) \, dx_2 \, dy_2 \, dz_2.
\end{align*}
Given a symbol $\sigma_2$ defined on $\R^{n_2} \times \R^{n_2} \times \R^{n_2}$, let $T_{\sigma_2}^{(2)}$ be the associated one-parameter bilinear pseudo-differential operator acting on functions on $\R^{n_2}$. Then we rewrite 
\begin{align*}
K_{f_2, g_2, h_2}(x_1, y_1, z_1) 
= \big\langle T_{\sigma_{x_1, y_1, z_1}}^{(2)} (f_2, g_2), h_2 \big\rangle,
\end{align*}
where
\begin{align*}
\sigma_{x_1, y_1, z_1}(x_2, \xi_2, \eta_2)
:= \int_{\R^{n_1}} \int_{\R^{n_1}} \sigma(x, \xi, \eta) e^{2\pi i [(x_1 - y_1) \cdot \xi_1 + (x_1 - z_1) \cdot \eta_1]} \, d\xi_1 \, d\eta_1.
\end{align*}
In view of Lemma \ref{lem:sxyz} and \cite[eq. (10.18)]{CLSY}, we see that $T_{\sigma_{x_1, y_1, z_1}}^{(2)}$ is a bilinear Calder\'{o}n--Zygmund operator with constant $(|x_1 - y_1| + |x_1 - z_1|)^{-2n_1}$, and that 
\[
T_{\sigma_{x_1, y_1, z_1}}^{(2)} - T_{\sigma_{x'_1, y_1, z_1}}^{(2)} = T_{\sigma_{x_1, y_1, z_1} - \sigma_{x'_1, y_1, z_1}}^{(2)}
\] 
is also a bilinear Calder\'{o}n--Zygmund operator with constant $\frac{|x_1 - x'_1|}{(|x_1 - y_1| + |x_1 - z_1|)^{2n_1+1}}$, whenever $|x_1 - x'_1| \le \max\{|x_1 - y_1|, |x_1 - z_1|\}/2$. These imply the size condition and the H\"{o}lder condition with respect to the variable $x_1$ in the partial kernel representation on the first parameter with
\begin{align*}
C(f_2, g_2, h_2) := C_0 \|f_2\|_{L^4(\R^{n_2})} \|g_2\|_{L^4(\R^{n_2})} \|h_2\|_{L^2(\R^{n_2})}.
\end{align*}
Similarly, we obtain the H\"{o}lder conditions with respect to the variables $y_1$ and $z_1$ in the partial kernel representation on the first parameter. Symmetrically, one can show the partial kernel representation on the second parameter. 

Part \eqref{sss-3} is a consequence of the estimate
\begin{align}\label{L4L4}
\|T_{\sigma}\|_{L^4 \times L^4 \to L^2} 
\lesssim \sum_{|\alpha|, |\beta|, |\gamma| \le n_1 + n_2 +1} P_{\alpha, \beta, \gamma}^{m, \rho, \delta}(\sigma),
\end{align}
where $N_0 \ge n_1 + n_2 + 1$. To prove \eqref{L4L4}, take a function $a_i \in \mathscr{C}_c^{\infty}(\R^{2n_i})$ such that $\mathbf{1}_{B(0, 1/2)} \le a_i \le \mathbf{1}_{B(0, 1)}$, $i=1, 2$. Then we write
\begin{align}\label{def:sgsg}
\sigma = \sum_{j=1}^4 \sigma_j,
\end{align} 
where
\begin{align*}
\sigma_1(x, \xi, \eta) & := \sigma(x, \xi, \eta) a_1(\xi_1, \eta_1) a_2(\xi_2, \eta_2),
\\
\sigma_2(x, \xi, \eta) & := \sigma(x, \xi, \eta) a_1(\xi_1, \eta_1) (1-a_2(\xi_2, \eta_2)),
\\
\sigma_3(x, \xi, \eta) & := \sigma(x, \xi, \eta) (1-a_1(\xi_1, \eta_1)) a_2(\xi_2, \eta_2),
\\
\sigma_4(x, \xi, \eta) & := \sigma(x, \xi, \eta) (1-a_1(\xi_1, \eta_1)) (1-a_2(\xi_2, \eta_2)).
\end{align*}
It is easy to check that $\sigma_j \in \mathcal{S}_{\rho, \delta}^m(\Rnn)$ for each $j=1, 2, 3, 4$. By symmetry, the estimate \eqref{L4L4} follows from Lemmas \ref{lem:suppxieta}, \ref{lem:suppxieta1}, and \ref{lem:suppxieta2}.

For part \eqref{sss-4}, it suffices to treat $\mathbb{T} \in \{T_{\sigma, 1}^{1*}, T_{\sigma, 2}^{1*}, T_{\sigma, 1}^{2*}, T_{\sigma, 2}^{2*}\}$ because 
\begin{align*}
T_{\sigma}^{1*} = (T_{\sigma, 1}^{1*})_2^{1*}, \quad 
T_{\sigma}^{2*} = (T_{\sigma, 1}^{2*})_2^{2*}, \quad
T_{\sigma, 1, 2}^{1*, 2*} = (T_{\sigma, 1}^{1*})_2^{2*}, \quad \text{and} \quad
T_{\sigma, 2, 1}^{1*, 2*} = (T_{\sigma, 2}^{1*})_1^{2*}. 
\end{align*}
By symmetry, we only focus on $T_{\sigma, 1}^{1*}$. Let $b$ denote the symbol of $T_{\sigma, 1}^{1*}$. Then $K_b(x, y, z) = K_{\sigma}((y_1, x_2), (x_1, y_2), z))$ and 
\begin{align*}
b(x, \xi, \eta) 
& = \iiiint \sigma((y_1, x_2), \xi', \eta') e^{2\pi i [(y-x) \cdot \xi + (z-x) \cdot \eta]} 
\\
& \quad \times e^{2\pi i [(y_1- x_1) \cdot \xi'_1 + (x_2 - y_2) \cdot \xi'_2 + (y_1 - z_1) \cdot \eta'_1 + (x_2 - z_2) \cdot \eta'_2]} \, d\xi' \, d\eta' \, dy \, dz.
\end{align*}
By changing variables, there holds
\begin{align*}
b(x, \xi, \eta) 
& = \iiiint \sigma((x_1 + y_1, x_2), \xi', \eta') 
e^{2\pi i [y_1 \cdot (\xi_1 + \xi'_1 + \eta'_1) + y_2 \cdot (\xi_2 - \xi'_2) + z \cdot (\eta - \eta')]} \, d\xi' \, d\eta' \, dy \, dz
\\
& = \int_{\R^{n_1}} \int_{\R^{n_1}} \sigma((x_1 + y_1, x_2), (\xi'_1, \xi_2), \eta) 
e^{2\pi i y_1 \cdot (\xi_1 + \xi'_1 + \eta_1)} \, d\xi'_1 \, dy_1
\\
& = \int_{\R^{n_1}} \int_{\R^{n_1}} \sigma((x_1 + y_1, x_2), (-z_1 - \xi_1 - \eta_1, \xi_2), \eta) 
e^{-2\pi i y_1 \cdot z_1} \, dy_1 \, dz_1.
\end{align*}
Given multi-indices $\alpha = (\alpha_1, \alpha_2)$, $\beta = (\beta_1, \beta_2)$ and $\gamma = (\gamma_1, \gamma_2)$, writing 
\begin{align*}
\sigma_{x_2, \xi_2, \eta_2}(x_1, \xi_1, \eta_1) 
:= \partial_{x_2}^{\alpha_2} \partial_{\xi_2}^{\beta_2} \partial_{\eta_2}^{\gamma_2} \sigma(x, \xi, \eta),
\end{align*} 
we have
\begin{align}\label{dpab-1}
\big| \partial_{x_1}^{\alpha_1} \partial_{\xi_1}^{\beta_1} \partial_{\eta_1}^{\gamma_1}
\sigma_{x_2, \xi_2, \eta_2}(x_1, \xi_1, \eta_1) \big| 
\lesssim \prod_{i=1}^2 (1 + |\xi| + |\eta_i|)^{\delta_i |\alpha_i| - (|\beta_i| + |\gamma_i|)}
\end{align}
and
\begin{align}\label{dpab-2}
& \partial_{x_1}^{\alpha_1} \partial_{x_2}^{\alpha_2} \partial_{\xi_1}^{\beta_1} \partial_{\xi_2}^{\beta_2} 
\partial_{\eta_1}^{\gamma_1} \partial_{\eta_2}^{\gamma_2} b(x, \xi, \eta) 
\\ \nonumber
& = \int_{\R^{n_1}} \int_{\R^{n_1}}  \partial_{x_1}^{\alpha_1} \partial_{\xi_1}^{\beta_1} \partial_{\eta_1}^{\gamma_1}
\sigma_{x_2, \xi_2, \eta_2}(x_1 + y_1, -z_1 - \xi_1 - \eta_1, \eta_1) 
e^{-2\pi i y_1 \cdot z_1} \, dy_1 \, dz_1.
\end{align}
Then it follows from \eqref{dpab-1}--\eqref{dpab-2} and the proof of \cite[Theorem 2.1]{BMNT} that 
\begin{align*}
|\partial_x^{\alpha} \partial_{\xi}^{\beta} \partial_{\eta}^{\gamma} b(x, \xi, \eta)|
\lesssim \prod_{i=1}^2 (1 + |\xi| + |\eta_i|)^{\delta_i |\alpha_i| - (|\beta_i| + |\gamma_i|)},
\end{align*}
which asserts $b \in \mathcal{S}_{\rho, \delta}^m(\Rnn)$. 

It remains to show part \eqref{sss-5}. Let $\psi^i \in \mathscr{C}_c^{\infty}(\R^{n_i})$ be a smooth cut-off function such that $\mathbf{1}_{B(0, 1)} \le \psi^i \le \mathbf{1}_{B(0, 2)}$. For each $j \in \N$, set 
\begin{align}\label{def:psi}
\phi_j(x, \xi, \eta) 
= & \psi_j^1(x_1) \psi_j^2(x_2) \psi_j^1(\xi_1) \psi_j^2(\xi_2)  \psi_j^1(\eta_1) \psi_j^2(\eta_2)  
\\ \nonumber
:= & \psi^1(2^{-j} x_1) \psi^2(2^{-j} x_2) \psi^1(2^{-j} \xi_1) \psi^2(2^{-j} \xi_2) \psi^1(2^{-j} \eta_1) \psi^2(2^{-j} \eta_2).
\end{align} 
Then one has
\begin{align}\label{def:sig}
\sigma_j := \sigma \phi_j \in \mathcal{S}_{\rho, \delta}^{m}(\Rnn) 
\quad\text{ uniformly in $j \in \N$}.
\end{align}
It is enough to show  
\begin{align}\label{TT-2}
\lim_{j \to \infty} \big\|T_{\sigma_j} - T_{\sigma}\big\|_{L^4 \times L^4 \to L^2} = 0
\end{align}
and 
\begin{align}\label{TT-3}
T_{\sigma_j} : L^4(\Rnn) \times L^4(\Rnn) \to L^2(\Rnn) \text{ compactly}.
\end{align}
Observe that \eqref{TT-2} is a consequence of \eqref{def:sig}, \eqref{L4L4}, and Lemma \ref{lem:Pab}. Additionally, in view of \cite[Theorem 6.3]{CY}, \eqref{def:sig}, and \eqref{L4L4}, the estimate \eqref{TT-3} is reduced to proving 
\begin{align}\label{TFK-2}
\lim_{A \to \infty} & \sup_{\substack{\|f_1\|_{L^4} \le 1 \\ \|f_2\|_{L^4} \le 1}} 
\big\|T_{\sigma_j}(f_1, f_2) \mathbf{1}_{B(0, A)^c}\big\|_{L^2} = 0,
\end{align}
and
\begin{align}\label{TFK-3}
\lim_{|v| \to 0} & \sup_{\substack{\|f_1\|_{L^4} \le 1 \\ \|f_2\|_{L^4} \le 1}} 
\big\|\tau_v T_{\sigma_j}(f_1, f_2) - T_{\sigma_j}(f_1, f_2)\big\|_{L^2} = 0.
\end{align}
Note that \eqref{TFK-2} follows from $T_{\sigma_j}(f_1, f_2)(x) = 0$ for all $|x| > 2^{j+2}$, since 
\begin{align}\label{sgsupp}
\supp \sigma_j \subset \{|x_i|, |\xi_i|, |\eta_i| \le 2^{j+1}, i=1, 2\}.
\end{align}
In order to obtain \eqref{TFK-2}, fix an even number $N > n_1 + n_2$, and note that 
\begin{align}\label{SNN}
\sigma_j \in \mathcal{S}_{(1, 1), (0, 0)}^{(-2N, -2N)}(\Rnn) 
\end{align}
and 
\begin{align}\label{KKSS}
\|K_{\sigma_j}\|_{L^{\infty}} 
+ \|K_{\Delta_{\xi}^N \sigma_j}\|_{L^{\infty}} 
+ \|K_{\Delta_{\eta}^N \sigma_j}\|_{L^{\infty}} 
\lesssim \prod_{i=1}^2 \iint \frac{d\xi_i \, d\eta_i}{(1 + |\xi_i| + |\eta_i|)^{2N}} 
\lesssim 1. 
\end{align}
Recall that $e^{i \xi \cdot x} = (-1)^k |x|^{-2k} \Delta_{\xi}^k e^{i \xi \cdot x}$ for all $k \in \N$ and $x \neq 0$. We use integration by parts and \eqref{sgsupp} to obtain 
\begin{align*}
K_{\Delta_{\xi}^N \sigma_j}(x, y, z) 
&= \iint \Delta_{\xi}^N \sigma_j(x, \xi, \eta) 
e^{2\pi i [(x-y) \cdot \xi + (x-z) \cdot \eta]} \, d\xi \, d\eta 
\\ \nonumber 
&= \iint \Delta_{\xi}^N \big(e^{2\pi i (x-y) \cdot \xi} \big) 
e^{2\pi i (x-z) \cdot \eta} \sigma_j(x, \xi, \eta) \, d\xi \, d\eta 
\\ \nonumber
&= (2\pi |x-y|)^{2N} K_{\sigma_j}(x, y, z),
\end{align*}
and
\begin{align*}
K_{\Delta_{\eta}^N \sigma_j}(x, y, z) 
= (2\pi |x-z|)^{2N} K_{\sigma_j}(x, y, z),
\end{align*}
which yield
\begin{align}\label{KSIJ}
K_{\sigma_j}(x, y, z) 
= \frac{K_{\sigma_j}(x, y, z) + K_{\Delta_{\xi}^N \sigma_j}(x, y, z) + K_{\Delta_{\eta}^N \sigma_j}(x, y, z)}{1 + (2\pi |x-y|)^{2N} + (2\pi |x-z|)^{2N}}, 
\end{align}
Then it follows from \eqref{KKSS} and \eqref{KSIJ} that for any $v \in \Rnn$, 
\begin{align*}
|\tau_v K_{\sigma_j}(x, y, z) - K_{\sigma_j}(x, y, z)|
\lesssim \sum_{i=1}^4 K_v^i(x, y, z), 
\end{align*}
where 
\begin{align*}
K_v^1(x, y, z) 
&:= \bigg|\frac{1}{1 + (2\pi |x - v - y|)^{2N} + (2\pi |x - v - z|)^{2N}} 
\\
&\qquad\qquad- \frac{1}{1 + (2\pi |x - y|)^{2N} + (2\pi |x - z|)^{2N}}\bigg|, 
\\
K_v^2(x, y, z) 
&:= \frac{|K_{\sigma_j}(x-v, y, z) - K_{\sigma_j}(x, y, z)|}{1 + (2\pi |x-y|)^{2N} + (2\pi |x-z|)^{2N}}, 
\\
K_v^3(x, y, z) 
&:= \frac{|K_{\Delta_{\xi}^N \sigma_j}(x-v, y, z) 
- K_{\Delta_{\xi}^N \sigma_j}(x, y, z)|}{1 + (2\pi |x-y|)^{2N} + (2\pi |x-z|)^{2N}}, 
\\
K_v^4(x, y, z) 
&:= \frac{|K_{\Delta_{\eta}^N \sigma_j}(x-v, y, z) 
- K_{\Delta_{\eta}^N \sigma_j}(x, y, z)|}{1 + (2\pi |x-y|)^{2N} + (2\pi |x-z|)^{2N}}. 
\end{align*}
Thus, 
\begin{align}\label{VTTS}
\|\tau_v T_{\sigma_j}(f_1, f_2) - T_{\sigma_j}(f_1, f_2)\|_{L^2}  
\lesssim \sum_{i=1}^4 \|G_v^i\|_{L^2},
\end{align}
where
\begin{align*}
G_v^i(x) 
:= \int_{\Rnn} \int_{\Rnn} K_v^i(x, y, z) |f_1(y)| \, |f_2(z)| \, dy \, dz.
\end{align*}
To estimate $G_v^1$, define 
\begin{align*}
h_k(x) := \frac{(1+|x-v|)^{-k/2}}{(1+|x|)^{(2N-k-1)/2}},
\end{align*}
and observe that 
\begin{align*}
K_v^1(x, y, z) 
\lesssim |v| \sum_{k=0}^{2N-1} 
\frac{(1 + |x - v - y| + |x - v - z|)^{-k}}{(1 + |x - y| + |x - z|)^{2N - k - 1}}
= |v| \sum_{k=0}^{2N-1} h_k(x - y) h_k(x - z),
\end{align*}
which gives
\begin{align}\label{GJJ-1}
\|G_v^1\|_{L^2}
&\lesssim |v| \sup_k \|(h_k*|f_1|) (h_k*|f_2|)\|_{L^2}
\\ \nonumber
& \le |v| \sup_k \|h_k*|f_1|\|_{L^4} \|h_k*|f_2|\|_{L^4}
\\ \nonumber 
&\le |v| \|f_1\|_{L^4} \|f_2\|_{L^4} \sup_k \|h_k\|_{L^1}^2
\lesssim |v| \|f_1\|_{L^4} \|f_2\|_{L^4}. 
\end{align}
Moreover, by \eqref{SNN} and the mean value theorem,
\begin{align*}
&|\tau_v K_{\sigma_j}(x, y, z) - K_{\sigma_j}(x, y, z)|
\\ 
&\le \int_{\Rnn} \int_{\Rnn} |\sigma_j(x-v, \xi, \eta) - \sigma_j(x, \xi, \eta)| \, d\xi \, d\eta 
\\
& \quad + \int_{\Rnn} \int_{\Rnn} |\sigma_j(x, \xi, \eta)| \big|e^{2\pi i v \cdot (\xi + \eta)} -1 \big| \, d\xi \, d\eta
\\
&\lesssim |v| \iint |\nabla_x \sigma_j(x + \theta v, \xi, \eta)| \, d\xi \, d\eta 
+ |v| \iint |\sigma_j(x, \xi, \eta)| |\xi + \eta| \, d\xi \, d\eta 
\\
&\lesssim |v| \prod_{i=1}^2 \iint \frac{1 + |\xi_i| + |\eta_i|}{(1 + |\xi_i| + |\eta_i|)^{2N}} \, d\xi_i \, d\eta_i
\lesssim |v|, 
\end{align*}
which leads to  
\begin{align*}
G_v^2(x) 
\lesssim \iint \frac{|v| |f_1(y)| \, |f_2(z)|}{(1 + |x - y| + |x - z|)^{2N}} \, dy \, dz
\lesssim |v| \mathcal{M}_{\mathcal{R}}(f_1, f_2)(x), 
\end{align*}
and 
\begin{align}\label{GJJ-2}
\|G_v^2\|_{L^2}
\lesssim |v| \|\mathcal{M}_{\mathcal{R}}(f_1, f_2)\|_{L^2}
\lesssim |v| \|f_1\|_{L^4} \|f_2|\|_{L^4}. 
\end{align}
Likewise, there holds
\begin{align}\label{GJJ-34}
\|G_v^i\|_{L^2}
\lesssim |v| \|f_1\|_{L^4} \|f_2\|_{L^4}, \quad i=3, 4. 
\end{align}
Consequently, by \eqref{VTTS}--\eqref{GJJ-34},
\begin{align*}
\|\tau_v T_{\sigma_j}(f_1, f_2) - T_{\sigma_j}(f_1, f_2)\|_{L^2}  
\lesssim \sum_{i=1}^4 \|G_v^i\|_{L^2} 
\lesssim |v| \|f_1\|_{L^4} \|f_2\|_{L^4}, 
\end{align*}
which implies \eqref{TFK-3} as desired. 
\qed

\subsection{Estimates involving bi-parameter pseudo-differential operators}
\begin{lemma}\label{lem:KP}
Let $\sigma \in \S_{\rho, \delta}^m(\Rnn)$ with $m = (0, 0)$, $\rho = (1, 1)$, and $\delta = (\delta_1, \delta_2) \in [0, 1)^2$. If we define
\begin{align*}
\widetilde{K}_{\sigma}(x, y, z) 
:= \int_{\Rnn} \int_{\Rnn} \sigma(x, \xi, \eta) e^{2\pi i (y \cdot \xi + z \cdot \eta)} \, d\xi \, d\eta,
\end{align*}
then for all multi-indices $\alpha$, $\beta$ and $\gamma$, 
\begin{align*}
|\partial_{x_1}^{\alpha_1} \partial_{x_2}^{\alpha_2} \partial_{y_1}^{\beta_1} \partial_{y_2}^{\beta_2}
\partial_{z_1}^{\gamma_1} \partial_{z_2}^{\gamma_2} \widetilde{K}_{\sigma}(x, y, z)|
\lesssim \frac{\sum_{|\beta'|, |\gamma'| \le N_0} P_{\alpha, \beta', \gamma'}^{m, \rho, \delta}(\sigma)}{\prod_{i=1}^2 (|y_i| + |z_i|)^{2n_i + |\alpha_i| + |\beta_i| + |\gamma_i|}}, 
\end{align*}
whenever $y_1 \neq 0$ or $z_1 \neq 0$, and $y_2 \neq 0$ or $z_2 \neq 0$, where $N_i$ is an integer such that $N_0 \ge \sum_{i=1}^2 (2n_i + |\alpha_i| + |\beta_i| + |\gamma_i| + 1)$. 
\end{lemma}

\begin{proof}
Our strategy is to use a Littlewood-Paley decomposition of the symbol. For each $i=1, 2$, pick $\phi^{i} \in \mathscr{C}_c^{\infty}(\R^{2n_i})$ such that $\mathbf{1}_{B(0, 1)} \le \phi^i \le \mathbf{1}_{B(0, 2)}$. Then define $\psi_0^i(\xi_i, \eta_i) := \phi^i(\xi_i, \eta_i)$ and 
\begin{align}\label{psiij}
\psi_j^i(\xi_i, \eta_i) := \phi(2^{-j} \xi_i, 2^{-j} \eta_i) - \phi(2^{-j+1} \xi_i, 2^{-j+1} \eta_i), \quad \forall j \ge 1,
\end{align} 
which verify the following properties
\begin{align}
\label{psij-1}
& \supp(\psi_0^i)  \subset \big\{(\xi_i, \eta_i) \in \R^{2n_i}: |(\xi_i, \eta_i)| \leq 2\big\},
\\
\label{psij-2}
& \supp(\psi_j^i) \subset
\big\{(\xi_i, \eta_i) \in \R^{2n_i}: 2^{j-1} \leq |(\xi_i, \eta_i)| \leq 2^{j+1}\big\},
\\ 
\label{psij-3}
& \sum_{j=0}^{\infty} \psi_j^i(\xi_i, \eta_i) = 1, \ \ \forall \xi_i, \eta_i \in \R^{n_i}.
\end{align}
In addition, for all multi-indices $\beta$ and $\gamma$,
\begin{align}\label{e:bgpsi-0}
|\partial_{\xi_i}^{\beta} \partial_{\eta_i}^{\gamma} \psi_0^i(\xi_i, \eta_i)|
\leq C_{\beta, \gamma, N} (1 + |\xi_i| + |\eta_i|)^{-N}, \ \ \forall N \geq 0,
\end{align}
and
\begin{align}\label{e:bgpsi-j}
|\partial_{\xi_i}^{\beta} \partial_{\eta_i}^{\gamma} \psi_j^i(\xi_i, \eta_i)|
\leq C_{\beta, \gamma} \, 2^{-j(|\beta| + |\gamma|)}, \ \ \forall j \geq 1.
\end{align}
Setting
\begin{align}\label{}
\sigma_{j_1, j_2}(x, \xi, \eta) := \sigma(x, \xi, \eta) \psi_{j_1}^1(\xi_1, \eta_1) \psi_{j_2}^2(\xi_2, \eta_2),
\end{align} 
we have
\begin{align*}
& \big| \partial_x^{\alpha} \partial_{\xi}^{\beta} \partial_{\eta}^{\gamma} \sigma_{j_1, j_2}(x, \xi, \eta) \big|
= \big| \partial_{x_1}^{\alpha_1} \partial_{x_2}^{\alpha_2} 
\partial_{\xi_1}^{\beta_1} \partial_{\xi_2}^{\beta_2} 
\partial_{\eta_1}^{\gamma_1} \partial_{\eta_2}^{\gamma_2} \sigma_{j_1, j_2}(x, \xi, \eta) \big|
\\
& \lesssim \sum_{\substack{\beta'_1 + \beta''_1 = \beta_1 \\ \beta'_2 + \beta''_2 = \beta_2 \\ \gamma'_1 + \gamma''_1 = \gamma_1 \\ \gamma'_2 + \gamma''_2 = \gamma_2}} \big| \partial_{x_1}^{\alpha_1} \partial_{x_2}^{\alpha_2} \partial_{\xi_1}^{\beta'_1} \partial_{\xi_2}^{\beta'_2} \partial_{\eta_1}^{\gamma'_1} \partial_{\eta_2}^{\gamma'_2} \sigma(x, \xi, \eta) \big|
\big| \partial_{\xi_1}^{\beta''_1} \partial_{\eta_1}^{\gamma''_1} \psi_{j_1}^1(\xi_1, \eta_1) \big| 
\big| \partial_{\xi_2}^{\beta''_2} \partial_{\eta_2}^{\gamma''_2} \psi_{j_2}^2(\xi_2, \eta_2) \big| 
\\
& \lesssim \sum_{\substack{\beta'_1 + \beta''_1 = \beta_1 \\ \beta'_2 + \beta''_2 = \beta_2 \\ \gamma'_1 + \gamma''_1 = \gamma_1 \\ \gamma'_2 + \gamma''_2 = \gamma_2}} P_{\alpha, \beta', \gamma'}^{m, \rho, \delta}(\sigma) 
\prod_{i=1}^2 (1 + |\xi_i| + |\eta_i|)^{\delta_i |\alpha_i| - (|\beta'_i| + |\gamma'_i|)} 
2^{-j_i(|\beta''_i| + |\gamma''_i|)} 
\\
& \lesssim \sum_{\substack{\beta' \le \beta \\ \gamma' \le \gamma}} 
P_{\alpha, \beta', \gamma'}^{m, \rho, \delta}(\sigma) 
\prod_{i=1}^2 2^{j_i[\delta_i |\alpha_i| - (|\beta_i| + |\gamma_i|)]},
\end{align*}
which implies
\begin{align}\label{eq:abgk}
& \big|\partial_{\xi}^{\beta_0} \partial_{\eta}^{\gamma_0} 
\big[\partial_x^{\alpha} \sigma_{j_1, j_2}(x, \xi, \eta) (2\pi i \xi)^{\beta} (2\pi i \eta)^{\gamma} \big] \big|
\\ \nonumber
& \lesssim \sum_{\substack{\beta_1 + \beta_2 = \beta_0 \\ \gamma_1 + \gamma_2 = \gamma_0}} 
\big| \partial_x^{\alpha} \partial_{\xi}^{\beta_1} \partial_{\eta}^{\gamma_1} \sigma_{j_1, j_2}(x, \xi, \eta) \big|
\big| \partial_{\xi}^{\beta_2} [(2\pi i \xi)^{\beta}] \big| 
\big| \partial_{\eta}^{\gamma_2} [(2\pi i \eta)^{\gamma}] \big|
\\ \nonumber
& \lesssim \sum_{\substack{\beta' \le \beta_0 \\ \gamma' \le \gamma_0}} 
P_{\alpha, \beta', \gamma'}^{m, \rho, \delta}(\sigma) 
\prod_{i=1}^2 2^{j_i[\delta_i |\alpha_i| - (|\beta_{0, i}| + |\gamma_{0, i}|) + |\beta_i| + |\gamma_i|]}.
\end{align}
Observe that
\begin{align*}
& (2\pi i y)^{\beta_0} \cdot (2\pi i z)^{\gamma_0} \partial_x^{\alpha} \partial_y^{\beta} \partial_z^{\gamma} \widetilde{K}_{\sigma_{j_1, j_2}}(x, y, z) 
\\
& = \int \int \partial_x^{\alpha} \sigma_{j_1, j_2}(x, \xi, \eta) \, (2\pi i \xi)^{\beta} (2\pi i \eta)^{\gamma} \,
\partial_{\xi}^{\beta_0} \partial_{\eta}^{\gamma_0} \big(e^{2\pi i (y \cdot \xi + z \cdot \eta)} \big)\, d\xi \, d\eta
\\
& = (-1)^{|\beta_0| + |\gamma_0|} \int \int 
\partial_{\xi}^{\beta_0} \partial_{\eta}^{\gamma_0} \big[\partial_x^{\alpha} \sigma_{j_1, j_2}(x, \xi, \eta) (2\pi i \xi)^{\beta} (2\pi i \eta)^{\gamma} \big]\,
e^{2\pi i (y \cdot \xi + z \cdot \eta)} \, d\xi \, d\eta,
\end{align*}
and 
\begin{align*}
\prod_{i=1}^2 (|y_i| + |z_i|)^{N_i} 
\simeq \prod_{i=1}^2 |(y_i, z_i)|^{N_i} 
\lesssim \sum_{\substack{|\beta_{0, 1}| + |\gamma_{0, 1}| = N_1 \\ |\beta_{0, 2}| + |\gamma_{0, 2}| = N_2}} |y^{\beta_0} \cdot z^{\gamma_0}|,
\end{align*}
which together with \eqref{psij-3} and \eqref{eq:abgk} give
\begin{align*}
& |\partial_x^{\alpha} \partial_y^{\beta} \partial_z^{\gamma} \widetilde{K}_{\sigma}(x, y, z)|
\le \sum_{j_1, j_2 \ge 0} |\partial_x^{\alpha} \partial_y^{\beta} \partial_z^{\gamma} \widetilde{K}_{\sigma_{j_1, j_2}}(x, y, z)|
\\
& \lesssim \sum_{|\beta'|, |\gamma'| \le N_1 + N_2} 
P_{\alpha, \beta', \gamma'}^{m, \rho, \delta}(\sigma) 
\prod_{i=1}^2 \sum_{j_i \ge 0} \frac{(|y_i| + |z_i|)^{-N_i}}{2^{j_i[N_i - 2n_i - \delta_i |\alpha_i| - |\beta_i| - |\gamma_i|]}}.
\end{align*}
In the case $|y_i| + |z_i| \ge 1$, we choose $N_i > 2n_i + |\alpha_i| + |\beta_i| + |\gamma_i|$ to deduce 
\begin{align}\label{Nxyz-1}
\sum_{j_i \ge 0} \frac{(|y_i| + |z_i|)^{-N_i}}{2^{j_i[N_i - 2n_i - \delta_i |\alpha_i| - |\beta_i| - |\gamma_i|]}}
\lesssim (|y_i| + |z_i|)^{-N_i} 
\le (|y_i| + |z_i|)^{-(2n_i + |\alpha_i| + |\beta_i| + |\gamma_i|)}.
\end{align}
In the case $|y_i| + |z_i| < 1$, there exists $j_i^0 \ge 0$ such that $|y_i| + |z_i| \simeq 2^{-j_i^0}$. Picking $N_i = 0$, we have
\begin{align}\label{Nxyz-2}
& \sum_{0 \le j_i \le j_i^0} \frac{(|y_i| + |z_i|)^{-N_i}}{2^{j_i[N_i - 2n_i - \delta_i |\alpha_i| - |\beta_i| - |\gamma_i|]}}
\lesssim 2^{j_i^0 (2n_i + \delta_i |\alpha_i| + |\beta_i| + |\gamma_i|)}
\\ \nonumber
& \le 2^{j_i^0 (2n_i + |\alpha_i| + |\beta_i| + |\gamma_i|)}
\simeq (|y_i| + |z_i|)^{-(2n_i + |\alpha_i| + |\beta_i| + |\gamma_i|)}.
\end{align}
If we take $N_i > 2n_i + |\alpha_i| + |\beta_i| + |\gamma_i|$, then 
\begin{align}\label{Nxyz-3}
& \sum_{j_i \ge j_i^0} \frac{(|y_i| + |z_i|)^{-N_i}}{2^{j_i[N_i - 2n_i - \delta_i |\alpha_i| - |\beta_i| - |\gamma_i|]}}
\lesssim \frac{(|y_i| + |z_i|)^{-N_i}}{2^{j_i^0[N_i - 2n_i - \delta_i |\alpha_i| - |\beta_i| - |\gamma_i|]}}
\\ \nonumber
& \le \frac{(|y_i| + |z_i|)^{-N_i}}{2^{j_i^0[N_i - 2n_i - |\alpha_i| - |\beta_i| - |\gamma_i|]}}
\simeq (|y_i| + |z_i|)^{-(2n_i + |\alpha_i| + |\beta_i| + |\gamma_i|)}.
\end{align}
Gathering all estimates \eqref{Nxyz-1}--\eqref{Nxyz-3}, we conclude the proof.
\end{proof}

\begin{lemma}\label{lem:sxyz}
Let $\sigma \in \S_{\rho, \delta}^m(\Rnn)$ with $m = (0, 0)$, $\rho = (1, 1)$, and $\delta = (\delta_1, \delta_2) \in [0, 1)^2$. If we define 
\begin{align*}
\widetilde{\sigma}_{x_1, y_1, z_1}(x_2, \xi_2, \eta_2)
:= \int_{\R^{n_1}} \int_{\R^{n_1}} \sigma(x, \xi, \eta) e^{2\pi i (y_1 \cdot \xi_1 + z_1 \cdot \eta_1)} \, d\xi_1 \, d\eta_1,
\end{align*}
then
$\widetilde{\sigma}_{x_1, y_1, z_1} \in \mathcal{S}_{1, \delta_2}^0(\R^{n_2})$ with 
\begin{align*}
& |\partial_{x_1}^{\alpha_1} \partial_{x_2}^{\alpha_2} \partial_{y_1}^{\beta_1} \partial_{\xi_2}^{\beta_2}
\partial_{z_1}^{\gamma_1} \partial_{\eta_2}^{\gamma_2} \widetilde{\sigma}_{x_1, y_1, z_1}(x_2, \xi_2, \eta_2)|
\\
& \lesssim \sum_{|\beta'_1|, |\gamma'_1| \le N_0} P_{\alpha, (\beta'_1, \beta_2), (\gamma'_1, \gamma_2)}^{m, \rho, \delta}(\sigma) 
\frac{(1+ |\xi_2| + |\eta_2|)^{\delta_2 |\alpha_2| - (|\beta_2| + |\gamma_2|)}}{(|y_1| + |z_1|)^{2n_1 + |\alpha_1| + |\beta_1| + |\gamma_1|}}
\end{align*}
whenever $y_1 \neq 0$ or $z_1 \neq 0$, where $N_0$ is an integer such that $N_0 \ge 2n_1 + |\alpha_1| + |\beta_1| + |\gamma_1| + 1$. 
\end{lemma}

\begin{proof}
Using the function $\psi_{j_1}^1$ in \eqref{psiij} and setting $\sigma^{j_1}(x, \xi, \eta) := \sigma(x, \xi, \eta) \psi_{j_1}^1(\xi_1, \eta_1)$, we have
\begin{align*}
& \big| \partial_x^{\alpha} \partial_{\xi}^{\beta} \partial_{\eta}^{\gamma} \sigma^{j_1}(x, \xi, \eta) \big|
= \big| \partial_{x_1}^{\alpha_1} \partial_{x_2}^{\alpha_2} 
\partial_{\xi_1}^{\beta_1} \partial_{\xi_2}^{\beta_2} 
\partial_{\eta_1}^{\gamma_1} \partial_{\eta_2}^{\gamma_2} \sigma^{j_1}(x, \xi, \eta) \big|
\\
& \lesssim \sum_{\substack{\beta'_1 + \beta''_1 = \beta_1 \\ \gamma'_1 + \gamma''_1 = \gamma_1}} 
\big| \partial_{x_1}^{\alpha_1} \partial_{x_2}^{\alpha_2} \partial_{\xi_1}^{\beta'_1} \partial_{\xi_2}^{\beta_2} \partial_{\eta_1}^{\gamma'_1} \partial_{\eta_2}^{\gamma_2} \sigma(x, \xi, \eta) \big|
\big| \partial_{\xi_1}^{\beta''_1} \partial_{\eta_1}^{\gamma''_1} \psi_{j_1}^1(\xi_1, \eta_1) \big|
\\
& \lesssim \sum_{\substack{\beta'_1 + \beta''_1 = \beta_1 \\ \gamma'_1 + \gamma''_1 = \gamma_1}}  
P_{\alpha, (\beta'_1, \beta_2), (\gamma'_1, \gamma_2)}^{m, \rho, \delta}(\sigma) 
(1 + |\xi_1| + |\eta_1|)^{\delta_1 |\alpha_1| - (|\beta'_1| + |\gamma'_1|)} 
2^{-j_1(|\beta''_1| + |\gamma''_1|)} 
\\
&\qquad\qquad \times (1 + |\xi_2| + |\eta_2|)^{\delta_2 |\alpha_2| - (|\beta_2| + |\gamma_2|)}
\\
& \lesssim \sum_{\substack{\beta'_1 \le \beta_1 \\ \gamma'_1 \le \gamma_1}} 
P_{\alpha, (\beta'_1, \beta_2), (\gamma'_1, \gamma_2)}^{m, \rho, \delta}(\sigma) 
2^{j_1[\delta_1 |\alpha_1| - (|\beta_1| + |\gamma_1|)]}
(1 + |\xi_2| + |\eta_2|)^{\delta_2 |\alpha_2| - (|\beta_2| + |\gamma_2|)},
\end{align*}
which gives
\begin{align}\label{eq:bgk}
& \big| \partial_{\xi_1}^{\beta'_1} \partial_{\eta_1}^{\gamma'_1} 
\big[\partial_x^{\alpha} \partial_{\xi_2}^{\beta_2} \partial_{\eta_2}^{\gamma_2} 
\sigma^{j_1}(x, \xi, \eta) (2\pi i \xi_1)^{\beta_1} (2\pi i \eta_1)^{\gamma_1} \big] \big|
\\ \nonumber
& \lesssim \sum_{\substack{\beta''_1 \le \beta'_1 \\ \gamma''_1 \le \gamma'_1}} 
\big| \partial_x^{\alpha} \partial_{\xi_1}^{\beta''_1} \partial_{\xi_2}^{\beta_2} 
\partial_{\eta_1}^{\gamma''_1} \partial_{\eta_2}^{\gamma_2} \sigma^{j_1}(x, \xi, \eta) \big| 
\big| \partial_{\xi_1}^{\beta'_1 - \beta''_1} (2\pi i \xi_1)^{\beta_1} \big| 
\big| \partial_{\eta_1}^{\gamma'_1 - \gamma''_1} (2\pi i \eta_1)^{\gamma_1} \big| 
\\ \nonumber
& \lesssim \sum_{\substack{\beta''_1 \le \beta'_1 \\ \gamma''_1 \le \gamma'_1}} 
P_{\alpha, (\beta''_1, \beta_2), (\gamma''_1, \gamma_2)}^{m, \rho, \delta}(\sigma) 
\frac{(1 + |\xi_2| + |\eta_2|)^{\delta_2 |\alpha_2| - (|\beta_2| + |\gamma_2|)}}{2^{j_1(|\beta'_1| + |\gamma'_1| - \delta_1 |\alpha_1| - |\beta_1| - |\gamma_1|)}}.
\end{align}
Note that
\begin{align*}
& (2\pi i y_1)^{\beta'_1} (2\pi i z_1)^{\gamma'_1} 
\partial_{x_1}^{\alpha_1} \partial_{x_2}^{\alpha_2} 
\partial_{y_1}^{\beta_1} \partial_{\xi_2}^{\beta_2} 
\partial_{z_1}^{\gamma_1} \partial_{\eta_2}^{\gamma_2} 
\widetilde{\sigma}^{j_1}_{x_1, y_1, z_1}(x_2, \xi_2, \eta_2) 
\\
& = \iint \partial_{x_1}^{\alpha_1} \partial_{x_2}^{\alpha_2} 
\partial_{\xi_2}^{\beta_2} \partial_{\eta_2}^{\gamma_2} \sigma_{j_1}(x, \xi, \eta)
(2\pi i \xi_1)^{\beta_1} (2\pi i \eta_1)^{\gamma_1} 
\partial_{\xi_1}^{\beta'_1} \partial_{\eta_1}^{\gamma'_1} 
\big(e^{2\pi i (y_1 \cdot \xi_1 + z_1 \cdot \eta_1)} \big)\, d\xi_1 \, d\eta_1
\\
& = (-1)^{|\beta'_1| + |\gamma'_1|} \iint
\partial_{\xi_1}^{\beta'_1} \partial_{\eta_1}^{\gamma'_1} 
\big[\partial_x^{\alpha} \partial_{\xi_2}^{\beta_2} \partial_{\eta_2}^{\gamma_2} 
\sigma_{j_1}(x, \xi, \eta) (2\pi i \xi_1)^{\beta_1} (2\pi i \eta_1)^{\gamma_1} \big]\,
e^{2\pi i (y_1 \cdot \xi_1 + z_1 \cdot \eta_1)} \, d\xi_1 \, d\eta_1,
\end{align*}
which along with \eqref{psij-3} and \eqref{eq:bgk} implies
\begin{align*}
& |\partial_{x_1}^{\alpha_1} \partial_{x_2}^{\alpha_2} 
\partial_{y_1}^{\beta_1} \partial_{\xi_2}^{\beta_2} 
\partial_{z_1}^{\gamma_1} \partial_{\eta_2}^{\gamma_2} 
\widetilde{\sigma}_{x_1, y_1, z_1}(x_2, \xi_2, \eta_2)|
\le \sum_{j_1 \ge 0} |\partial_{x_1}^{\alpha_1} \partial_{x_2}^{\alpha_2} 
\partial_{y_1}^{\beta_1} \partial_{\xi_2}^{\beta_2} 
\partial_{z_1}^{\gamma_1} \partial_{\eta_2}^{\gamma_2} 
\widetilde{\sigma}^{j_1}_{x_1, y_1, z_1}(x_2, \xi_2, \eta_2)|
\\
& \lesssim \sum_{\substack{\beta''_1 \le \beta'_1 \\ \gamma''_1 \le \gamma'_1}} 
P_{\alpha, (\beta''_1, \beta_2), (\gamma''_1, \gamma_2)}^{m, \rho, \delta}(\sigma) 
\sum_{j_1 \ge 0}
\frac{(|y_1| + |z_1|)^{-|\beta'_1| - |\gamma'_1|} (1 + |\xi_2| + |\eta_2|)^{\delta_2 |\alpha_2| - (|\beta_2| + |\gamma_2|)}}{2^{j_1(\beta'_1| + |\gamma'_1| - 2n_1 - \delta_1 |\alpha_1| - |\beta_1| - |\gamma_1|)}}.
\end{align*}
Then invoking \eqref{Nxyz-1}--\eqref{Nxyz-3}, we obtain the estimate as desired.
\end{proof}

\begin{lemma}\label{lem:suppxieta}
Let $a \in \S_{\rho, \delta}^m(\Rnn)$ with $m = (0, 0)$, $\rho = (1, 1)$, and $\delta = (\delta_1, \delta_2) \in [0, 1)^2$. Assume that $a(x, \xi, \eta)$ has compact support in $(\xi, \eta)$ uniformly with respect to $x$. Then $T_a$ is bounded from $L^{p_1}(\Rnn) \times L^{p_2}(\Rnn)$ to $L^p(\Rnn)$ for all $\frac1p = \frac{1}{p_1} + \frac{1}{p_2}$ with $p_1, p_2 \in [1, \infty]$. Moreover,
\begin{align*}
\|T_a\|_{L^{p_1} \times L^{p_2} \to L^p} 
\lesssim \sup_{\substack{x, \xi, \eta \in \Rnn \\ |\beta|, |\gamma| \le n_1 + n_2 +1}}
|\partial_{\xi}^{\beta} \partial_{\eta}^{\gamma} a(x, \xi, \eta)| 
\lesssim \sum_{|\beta|, |\gamma| \le n_1 + n_2 +1} P_{0, \beta, \gamma}^{m, \rho, \delta}(a).
\end{align*}
\end{lemma}

\begin{proof}
Note that for any $N \ge 1$,
\begin{align*}
(1 + |x|)^N 
= \sum_{0 \le k \le N} c_k |x|^k 
\lesssim \sum_{|\beta| \le N} |x^{\beta}|,
\end{align*}
which gives
\begin{align*}
& (1 + |x-y|)^N (1 + |x-z|)^N |K_a(x, y, z)| 
\\
& \lesssim \sum_{|\beta|, |\gamma| \le N}
\bigg|\iint a(x, \xi, \eta) (x-y)^{\beta} (x-z)^{\gamma} e^{2\pi i [(x-y) \cdot \xi + (x-z) \cdot \eta]}  d\xi \, d\eta\bigg|
\\
& \simeq \sum_{|\beta|, |\gamma| \le N}
\bigg|\iint a(x, \xi, \eta) \partial_{\xi}^{\beta} \partial_{\eta}^{\gamma} \big(e^{2\pi i [(x-y) \cdot \xi + (x-z) \cdot \eta]} \big)  d\xi \, d\eta\bigg|
\\
& = \sum_{|\beta|, |\gamma| \le N}
\bigg|\iint \partial_{\xi}^{\beta} \partial_{\eta}^{\gamma} a(x, \xi, \eta) e^{2\pi i [(x-y) \cdot \xi + (x-z) \cdot \eta]} \,  d\xi \, d\eta\bigg|
\\
& \lesssim \sup_{\substack{x, \xi, \eta \\ |\beta|, |\gamma| \le N}}
|\partial_{\xi}^{\beta} \partial_{\eta}^{\gamma} a(x, \xi, \eta)| 
=: A_N.
\end{align*}
Set $\varphi(x) := (1 + |x|)^{-N}$. By H\"{o}lder's and Young's inequalities, there holds
\begin{align*}
\|T_a(f_1, f_2)\|_{L^p} 
& \lesssim A_N \|(\varphi * f_1) (\varphi * f_2)\|_{L^p}
\le A_N \|\varphi * f_1\|_{L^{p_1}} \|\varphi * f_2\|_{L^{p_2}}
\\
& \le A_N \|\varphi\|_{L^1}^2 \|f_1\|_{L^{p_1}} \|f_2\|_{L^{p_2}}
\lesssim A_N \|f_1\|_{L^{p_1}} \|f_2\|_{L^{p_2}},
\end{align*}
provided $N > n_1 + n_2$.
\end{proof}

\begin{lemma}\label{lem:suppxieta1}
Let $a \in \S_{\rho, \delta}^m(\Rnn)$ with $m = (0, 0)$, $\rho = (1, 1)$, and $\delta = (\delta_1, \delta_2) \in [0, 1)^2$. Assume that $a(x, \xi, \eta)$ has compact support in $(\xi_1, \eta_1)$ uniformly with respect to $(x, \xi_2, \eta_2)$. Then $T_a$ is bounded from $L^4(\Rnn) \times L^4(\Rnn)$ to $L^2(\Rnn)$, and 
\begin{align*}
\|T_a\|_{L^4 \times L^4 \to L^2} 
\lesssim \sum_{|\alpha|, |\beta|, |\gamma| \le n_1 + n_2 +1} P_{\alpha, \beta, \gamma}^{m, \rho, \delta}(a).
\end{align*}
\end{lemma}

\begin{proof}
We may assume that $\supp a(x, (\cdot, \xi_2), (\cdot, \eta_2)) \subset \{ |(\xi_1, \eta_1)| \le 1\}$. Let $a'_2 \in \mathscr{C}_c^{\infty}(\R^{2n_2})$ be such that $\mathbf{1}_{B(0, 1/2)} \le a'_2 \le \mathbf{1}_{B(0, 1)}$. Write $a = a_1 + a_2$, where
\begin{align*}
a_1(x, \xi, \eta) := a(x, \xi, \eta) a'_2(\xi_2, \eta_2) \quad \text{and} \quad
a_2(x, \xi, \eta) := a(x, \xi, \eta) (1-a'_2(\xi_2, \eta_2)).
\end{align*}
It suffices to bound $T_{a_2}$ because of Lemma \ref{lem:suppxieta} and $a_1 \in \mathcal{S}_{\rho, \delta}^m(\Rnn)$ with
\begin{align*}
\sum_{|\beta|, |\gamma| \le n_1 + n_2 +1} P_{0, \beta, \gamma}^{m, \rho, \delta}(a_1)
\lesssim \sum_{|\beta|, |\gamma| \le n_1 + n_2 +1} P_{0, \beta, \gamma}^{m, \rho, \delta}(a).
\end{align*}
Take $\phi \in \mathscr{C}_c^{\infty}(\R^{2n_2})$ such that
\begin{align*}
\supp(\phi) \subset \{1/3 \le |(\xi_2, \eta_2)| \le 1\} 
\quad \text{and} \quad
\sum_{j \ge 0} \phi(2^{-j} \xi_2, 2^{-j} \eta_2) = 1, \forall |(\xi_2, \eta_2)| \ge 1/2.
\end{align*}
Then we split
\begin{align}\label{sspee-1}
a_2(x, \xi, \eta) 
= \sum_{j \ge 0} a_2(x, \xi, \eta) \phi(2^{-j} \xi_2, 2^{-j} \eta_2) 
=: \sum_{j \ge 0} a_2^j(x, \xi, \eta).
\end{align}
Pick a function $\psi \in \mathscr{C}^{\infty}(\R^{n_2})$ so that $\mathbf{1}_{B(0, \frac{1}{16})} \le \widehat{\psi} \le \mathbf{1}_{B(0, \frac18)}$, and 
\begin{align}\label{sspee-2}
a_2^j(x, \xi, \eta) = p_j(x, \xi, \eta) + q_j(x, \xi, \eta),
\end{align}
where
\begin{align*}
p_j(x, \xi, \eta) & := \int_{\R^{n_2}} a_2^j((x_1, x_2 - y_2), \xi, \eta) \, 2^{j n_2} \psi(2^j y_2)  \, dy_2,
\\
q_j(x, \xi, \eta) & := \int_{\R^{n_2}} \big[a_2^j(x, \xi, \eta) - a_2^j((x_1, x_2 - y_2), \xi, \eta) \big] \, 2^{j n_2} \psi(2^j y_2) \, dy_2.
\\
\end{align*}
Define 
\begin{align*}
\widetilde{p}_j(x, \xi, \eta) & := p_j((x_1, 2^{-j} x_2), (\xi_1, 2^j \xi_2), (\eta_1, 2^j \eta_2)),
\\
\widetilde{q}_j(x, \xi, \eta) & := q_j((x_1, 2^{-j} x_2), (\xi_1, 2^j \xi_2), (\eta_1, 2^j \eta_2)).
\end{align*}
It is easy to verify that 
\begin{align}
\label{suppp-1}
& \supp a_2^j(x, \cdot, \cdot) \subset \{ |(\xi_1, \eta_1)| \le 1, \, 2^j/3 \le |(\xi_2, \eta_2)| \le 2^j \},
\\
\label{suppp-2}
& \supp p_j(x, \cdot, \cdot) \subset \{ |(\xi_1, \eta_1)| \le 1, \, 2^j/3 \le |(\xi_2, \eta_2)| \le 2^j \},
\\ 
\label{suppp-3}
& \supp q_j(x, \cdot, \cdot) \subset \{ |(\xi_1, \eta_1)| \le 1, \, 2^j/3 \le |(\xi_2, \eta_2)| \le 2^j \},
\\
\label{suppp-4}
& \supp \widetilde{p}_j(x, \cdot, \cdot)
\subset \{ |(\xi_1, \eta_1)| \le 1, \, 2^j/3 \le |(\xi_2, \eta_2)| \le 2^j \}, 
\\
\label{suppp-5}
& \supp \widetilde{q}_j(x, \cdot, \cdot) 
\subset \{ |(\xi_1, \eta_1)| \le 1, \, 2^j/3 \le |(\xi_2, \eta_2)| \le 2^j \},
\\
\label{suppp-6}
& \partial_x^{\alpha} \partial_{\xi}^{\beta} \partial_{\eta}^{\gamma} a_2^j(x, \xi, \eta) 
\lesssim 2^{j(\delta_2 |\alpha_2| - |\beta_2| - |\gamma_2|)} A_0,
\\
\label{suppp-7}
& \partial_x^{\alpha} \partial_{\xi}^{\beta} \partial_{\eta}^{\gamma} p_j(x, \xi, \eta)
\lesssim 2^{j(\delta_2 |\alpha_2| - |\beta_2| - |\gamma_2|)} A_0,
\\
\label{suppp-8}
& \partial_x^{\alpha} \partial_{\xi}^{\beta} \partial_{\eta}^{\gamma} q_j(x, \xi, \eta)
\lesssim 2^{j[\delta_2 (|\alpha_2| +1) - |\beta_2| - |\gamma_2| - 1]} A_0,
\\
\label{suppp-9}
& \partial_x^{\alpha} \partial_{\xi}^{\beta} \partial_{\eta}^{\gamma} \widetilde{p}_j(x, \xi, \eta)
\lesssim 2^{-j (1 - \delta_2) |\alpha_2|} A_0 
\le A_0,
\\
\label{suppp-10}
& \partial_x^{\alpha} \partial_{\xi}^{\beta} \partial_{\eta}^{\gamma} \widetilde{q}_j(x, \xi, \eta)
\lesssim 2^{- j (1 - \delta_2)(1 + |\alpha_2|)} A_0 
\le 2^{- j (1 - \delta_2)} A_0,
\end{align}
where
\begin{align*}
A_0 := \sum_{\substack{|\alpha'| \le |\alpha| +1 \\ |\beta'| \le |\beta| \\ |\gamma'| \le |\gamma}} 
P_{\alpha', \beta', \gamma'}^{m, \rho, \delta}(a)
\quad \text{and} \quad 
A_1 := \sum_{|\alpha|, |\beta|, |\gamma| \le n_1 + n_2 +1} 
P_{\alpha, \beta, \gamma}^{m, \rho, \delta}(a).
\end{align*}
Thus, it follows from \eqref{suppp-4}--\eqref{suppp-5}, \eqref{suppp-9}--\eqref{suppp-10}, and Lemma \ref{lem:suppxieta} that
\begin{align}
\label{TpTq-1}
\|T_{p_j}\|_{L^{p_1} \times L^{p_2} \to L^p} 
& = \|T_{\widetilde{p}_j}\|_{L^{p_1} \times L^{p_2} \to L^p} 
\lesssim A_1,
\\
\label{TpTq-2}
\|T_{q_j}\|_{L^{p_1} \times L^{p_2} \to L^p} 
& = \|T_{\widetilde{q}_j}\|_{L^{p_1} \times L^{p_2} \to L^p}
\lesssim 2^{-j(1-\delta_2)} A_1,
\end{align}
which gives 
\begin{align}\label{sspee-3}
\sum_{j \ge 0} \|T_{q_j}\|_{L^{p_1} \times L^{p_2} \to L^p} 
\lesssim A_1,
\end{align}
for all $\frac1p = \frac{1}{p_1} + \frac{1}{p_2}$ with $p_1, p_2 \in [1, \infty]$.

Observe that the estimate \eqref{TpTq-1} is too rough to guarantee the series $\sum_{j \ge 0} \|T_{p_j}\|_{L^4 \times L^4 \to L^2}$ converges. To overcome this problem, we need some delicate calculations. Let $\widehat{h}$ denote the Fourier transform of $h$ with respect to $x_2$. Then
\begin{align}\label{Fourier}
\widehat{T_{p_j} (f, g)}(x_1, \xi'_2) 
& = \iint \widehat{p}_j ((x_1, \xi'_2 - \xi_2 - \eta_2), \xi, \eta) 
\widehat{f}(\xi) \widehat{g}(\eta) e^{2\pi i x_1 \cdot (\xi_1 + \eta_1)} \, d\xi \, d\eta
\\ \nonumber
& = \iint \Phi(x, \xi, \eta, \xi'_2) \phi(2^{-j} \xi_2, 2^{-j} \eta_2) 
\widehat{f}(\xi) \widehat{g}(\eta) e^{2\pi i x_1 \cdot (\xi_1 + \eta_1)} \, d\xi \, d\eta,
\end{align}
where $\Phi(x, \xi, \eta, \xi'_2) := \widehat{a}_2 ((x_1, \xi'_2 - \xi_2 - \eta_2), \xi, \eta) \widehat{\psi}(2^{-j} (\xi'_2 - \xi_2 - \eta_2))$. Note that $\supp \widehat{T_{p_j} (f, g)}(x_1, \cdot) \subset \{2^{j-3} \le |\xi'_2| < 2^{j+1}\}$, which along with Parseval's formula implies 
\begin{align*}
\bigg\| \sum_{j \ge 0} T_{p_{4j+m}} f \bigg\|_{L^2}^2 
= \sum_{j \ge 0} \| T_{p_{4j+m}} f \|_{L^2}^2, \quad m=0, 1, 2, 3.
\end{align*}
To handle the case $m=0$, setting 
\begin{align*}
E_j := \{|\xi_2| \ge 1/2, 2^{j-3} \le |\xi_2| \le 2^{j+1}\} 
\quad \text{and} \quad
\widehat{f}_j(\xi) := \widehat{f}(\xi) \mathbf{1}_{E_j}(\xi_2),
\end{align*}
we have
\begin{align*}
\widehat{f}(\xi) \mathbf{1}_{\{|\xi_2| \ge 1/2\}} 
= \sum_{j \ge 0} \widehat{f}_{4j}(\xi),
\end{align*}
and by \eqref{TpTq-1},
\begin{align*}
\sum_{j \ge 0} \|T_{p_{4j}} (f, g)\|_{L^2}^2 
& = \sum_{j \ge 0} \|\widehat{T_{p_{4j}} (f, g)}\|_{L^2}^2 
= \sum_{j \ge 0} \bigg\| \sum_{k \ge 0} \widehat{T_{p_{4j}} (f_{4k}, g)} \bigg\|_{L^2}^2 
\\
& = \sum_{j \ge 0} \|\widehat{T_{p_{4j}} (f_{4j}, g)} \|_{L^2}^2 
= \sum_{j \ge 0} \|T_{p_{4j}} (f_{4j}, g) \|_{L^2}^2 
\\
& \lesssim A_1^2 \sum_{j \ge 0} \|f_{4j} \|_{L^2}^2 \|g\|_{L^{\infty}}^2 
= A_1^2 \sum_{j \ge 0} \|\widehat{f_{4j}} \|_{L^2}^2 \|g\|_{L^{\infty}}^2  
\\
& \le A_1^2 \|\widehat{f} \|_{L^2}^2 \|g\|_{L^{\infty}}^2
= A_1^2 \|f\|_{L^2}^2 \|g\|_{L^{\infty}}^2,
\end{align*}
provided that $\langle \widehat{f}_{4k}, \widehat{f}_{4j} \rangle = 0$ for any $k \ne j$. Hence, 
\begin{align}\label{sspee-4}
\bigg\| \sum_{j \ge 0} T_{p_{4j}} \bigg\|_{L^2 \times L^{\infty} \to L^2}
\lesssim A_1.
\end{align}
Symmetrically, 
\begin{align}\label{sspee-5}
\bigg\| \sum_{j \ge 0} T_{p_{4j}} \bigg\|_{L^{\infty} \times L^2 \to L^2}
\lesssim A_1.
\end{align}
Then interpolating between \eqref{sspee-4} and \eqref{sspee-5} gives
\begin{align}\label{sspee-6}
\bigg\| \sum_{j \ge 0} T_{p_{4j}} \bigg\|_{L^4 \times L^4 \to L^2}
\lesssim A_1.
\end{align}
Similarly, 
\begin{align}\label{sspee-7}
\bigg\| \sum_{j \ge 0} T_{p_{4j+m}} \bigg\|_{L^4 \times L^4 \to L^2}
\lesssim A_1, \quad m=1, 2, 3.
\end{align}
As a consequence of \eqref{sspee-1}, \eqref{sspee-2}, and \eqref{sspee-6}--\eqref{sspee-7}, we conclude
\begin{align*}
\|T_{a_2}\|_{L^4 \times L^4 \to L^2}
\le \sum_{m=0}^3 \bigg\| \sum_{j \ge 0} T_{p_{4j+m}} \bigg\|_{L^4 \times L^4 \to L^2}
+ \sum_{j \ge 0} \|T_{q_j}\|_{L^4 \times L^4 \to L^2}
\lesssim A_1.
\end{align*}
The proof is complete. 
\end{proof}

\begin{lemma}\label{lem:suppxieta2}
Let $\sigma_4$ be given in \eqref{def:sgsg}. Then $T_{\sigma_4}$ is bounded from $L^4(\Rnn) \times L^4(\Rnn)$ to $L^2(\Rnn)$, and 
\begin{align*}
\|T_{\sigma_4}\|_{L^4 \times L^4 \to L^2} 
\lesssim \sum_{|\alpha|, |\beta|, |\gamma| \le n_1 + n_2 +1} P_{\alpha, \beta, \gamma}^{m, \rho, \delta}(\sigma).
\end{align*}
\end{lemma}

\begin{proof}
Since the proof is similar to that of $T_{a_2}$ in Lemma \ref{lem:suppxieta1}, we only present the main modifications.  For each $i=1, 2$, pick $\phi_i \in \mathscr{C}_c^{\infty}(\R^{2n_i})$ so that
\begin{align*}
\supp(\phi_i) \subset \{1/3 \le |(\xi_i, \eta_i)| \le 1\} 
\quad \text{and} \quad
\sum_{j \ge 0} \phi_i(2^{-j} \xi_i, 2^{-j} \eta_i) = 1, \forall |(\xi_i, \eta_i)| \ge 1/2,
\end{align*}
which allows us to write
\begin{align*}
\sigma_4(x, \xi, \eta) 
= \sum_{j_1, j_2 \ge 0} \sigma_4^{j_1, j_2}(x, \xi, \eta)
:= \sum_{j_1, j_2 \ge 0} \sigma_4(x, \xi, \eta) \phi_1(2^{-j_1} \xi_1, 2^{-j_1} \eta_1) \phi_2(2^{-j_2} \xi_2, 2^{-j_2} \eta_2).
\end{align*}
Let $\psi_i \in \mathscr{C}^{\infty}(\R^{n_i})$ so that $\mathbf{1}_{B(0, \frac{1}{16})} \le \widehat{\psi}_i \le \mathbf{1}_{B(0, \frac18)}$, and $\sigma_4^{j_1, j_2} = p_{j_1, j_2} + q_{j_1, j_2}$,
where
\begin{align*}
p_{j_1, j_2}(x, \xi, \eta) & := \int_{\Rnn} \sigma_4^{j_1, j_2}(x-y, \xi, \eta) 
2^{j n_1} \psi_1(2^j y_1) \, 2^{j n_2} \psi_2(2^j y_2) \, dy,
\\
q_{j_1, j_2}(x, \xi, \eta) & := \int_{\Rnn} \big[\sigma_4^{j_1, j_2}(x, \xi, \eta) - \sigma_4^{j_1, j_2}(x - y, \xi, \eta) \big] \, 
2^{j n_1} \psi_1(2^j y_1) \, 2^{j n_2} \psi_2(2^j y_2) \, dy.
\end{align*}
Define 
\begin{align*}
\widetilde{p}_j(x, \xi, \eta) := p_j(2^{-j} x, 2^j \xi, 2^j \eta)
\quad \text{and} \quad
\widetilde{q}_j(x, \xi, \eta) := q_j(2^{-j} x, 2^j \xi, 2^j \eta).
\end{align*}
Unlike to \eqref{Fourier}, in the current situation, $\widehat{h}$ denotes the Fourier transform of $h$ with respect to $x_1$ and $x_2$. Then
\begin{align*}
\widehat{T_{p_j} (f, g)}(\xi') 
& = \iint \widehat{p}_j (\xi' - \xi - \eta, \xi, \eta) 
\widehat{f}(\xi) \widehat{g}(\eta) e^{2\pi i x \cdot (\xi + \eta)} \, d\xi \, d\eta
\\
& = \iint \Phi(\xi, \eta, \xi'_2) \phi_1(2^{-j} \xi_1, 2^{-j} \eta_1) \phi_2(2^{-j} \xi_2, 2^{-j} \eta_2) 
\widehat{f}(\xi) \widehat{g}(\eta) e^{2\pi i x \cdot (\xi + \eta)} \, d\xi \, d\eta,
\end{align*}
where 
\begin{align*}
\Phi(\xi, \eta, \xi'_2) 
:= \widehat{\sigma}_4 (\xi' - \xi - \eta, \xi, \eta) \widehat{\psi}_1(2^{-j} (\xi'_1 - \xi_1 - \eta_1)) \widehat{\psi}_2(2^{-j} (\xi'_2 - \xi_2 - \eta_2)).
\end{align*} 
More details are left to the reader. 
\end{proof}

\begin{lemma}\label{lem:Pab}
Let $m = (m_1, m_2) \in \R^2$, $\rho = (\rho_1, \rho_2) \in (0, 1]^2$, $\delta = (\delta_1, \delta_2) \in [0, 1]$. Let $\sigma \in \mathcal{K}_{\rho, \delta}^m(\Rnn)$ and define $\sigma_j$ as in \eqref{def:sig}. Then for all multi-indices $\alpha$, $\beta$, and $\gamma$, 
\begin{align*}
\lim_{j \to \infty} P_{\alpha, \beta, \gamma}^{m, \rho, \delta}(\sigma_j - \sigma) = 0.  
\end{align*}
\end{lemma}

\begin{proof}
Fix multi-indices $\alpha$, $\beta$, and $\gamma$. Leibniz's rule implies
\begin{align*}
&|\partial_x^{\alpha} \partial_{\xi}^{\beta} \partial_{\eta}^{\gamma} (\sigma - \sigma_j)(x, \xi, \eta)|
\lesssim \sum_{\substack{\alpha' + \alpha'' = \alpha \\ \beta' + \beta'' = \beta \\ 
\gamma' + \gamma'' = \gamma}} 
S_{\alpha', \beta', \gamma'}^{\alpha'', \beta'', \gamma''}
\\ 
&=: \sum_{\substack{\alpha' + \alpha'' = \alpha \\ \beta' + \beta'' = \beta \\ 
\gamma' + \gamma'' = \gamma}} 
|\partial_x^{\alpha'} \partial_{\xi}^{\beta'} \partial_{\eta}^{\gamma'} \sigma(x, \xi, \eta)| 
|\partial_x^{\alpha''} \partial_{\xi}^{\beta''} \partial_{\eta}^{\gamma''} (1-\phi_j)(x, \xi, \eta)|.   
\end{align*}
If $|\alpha''| = |\beta''| = |\gamma''| = 0$, then by the fact that $\supp(1-\phi_j) \subset \{|x| + |\xi| + |\eta| \ge 2^j\}$, 
\begin{align*}
S_{\alpha', \beta', \gamma'}^{\alpha'', \beta'', \gamma''}
\lesssim C_{\alpha, \beta, \gamma}(x, \xi, \eta) 
\prod_{i=1}^2 (1+ |\xi_i| + |\eta_i|)^{m_i + \delta_i |\alpha_i| - \rho_i (|\beta_i| + |\gamma_i|)} 
\mathbf{1}_{\{|x| + |\xi| + |\eta| \ge 2^j\}}. 
\end{align*}
If $|\alpha''| \neq 0$ and $|\beta''| = |\gamma''| = 0$, then 
\begin{align*}
|\partial_x^{\alpha''} (1-\phi_j)(x, \xi, \eta)| 
\le |\partial_{x_1}^{\alpha''_1} \psi_j^1(x_1)| |\partial_{x_2}^{\alpha''_2} \psi_j^2(x_2)|
\lesssim 2^{-j |\alpha''|} \mathbf{1}_{\{2^j \le |x_1|, |x_2| \le 2^{j+1}\}}, 
\end{align*}
which gives  
\begin{align*}
S_{\alpha', \beta', \gamma'}^{\alpha'', \beta'', \gamma''}
& \lesssim C_{\alpha', \beta, \gamma}(x, \xi, \eta) 
\prod_{i=1}^2 (1+ |\xi_i| + |\eta_i|)^{m_i + \delta_i |\alpha'_i| - \rho_i (|\beta_i| + |\gamma_i|)} 
\mathbf{1}_{\{2^j \le |x_i| \le 2^{j+1}\}}
\\
&\le C_{\alpha', \beta, \gamma}(x, \xi, \eta) 
\prod_{i=1}^2 (1+ |\xi_i| + |\eta_i|)^{m_i + \delta_i |\alpha_i| - \rho_i (|\beta_i| + |\gamma_i|)} 
\mathbf{1}_{\{|x| + |\xi| + |\eta| \ge 2^j\}}. 
\end{align*}
If $|\beta''| \neq 0$ or $|\gamma''| \neq 0$, then we may assume that $|\beta''| \neq 0$ and 
\begin{align*}
&|\partial_x^{\alpha''} \partial_{\xi}^{\beta''} \partial_{\eta}^{\gamma''} (1-\phi_j)(x, \xi, \eta)| 
\\
&\le |\partial_x^{\alpha''} \psi_j^1 \otimes \psi_j^2(x)| 
|\partial_{\xi}^{\beta''} \psi_j^1 \otimes \psi_j^2(\xi)| 
|\partial_{\eta}^{\gamma''} \psi_j^1 \otimes \psi_j^2(\eta)| 
\\
& \lesssim 2^{-j (|\alpha''| + |\beta''| + |\gamma''|)} 
\mathbf{1}_{\{2^j \le |\xi| \le 2^{j+2}\}} 
\mathbf{1}_{\{|\eta| \le 2^{j+2}\}} 
\\
& \lesssim \prod_{i=1}^2 (1+|\xi_i| + |\eta_i|)^{-(|\beta''_i| + |\gamma''_i|)} 
\, \mathbf{1}_{\{|x| + |\xi| + |\eta| \ge 2^j\}},
\end{align*}
where we have used that $1 + |\xi_i| + |\eta_i| \le 1 + 2^{j+2} + 2^{j+2} < 2^{j+4}$. Hence, 
\begin{align*}
S_{\alpha', \beta', \gamma'}^{\alpha'', \beta'', \gamma''}
& \lesssim C_{\alpha', \beta', \gamma'}(x, \xi, \eta) 
\prod_{i=1}^2 (1+|\xi_i| + |\eta_i|)^{m_i + \delta_i |\alpha'_i| - \rho_i (|\beta'_i| + |\gamma'_i|)} 
\\
&\qquad\times \prod_{i=1}^2 (1+|\xi_i| + |\eta_i|)^{-\rho (|\beta''_i| + |\gamma''_i|)} 
\, \mathbf{1}_{\{|x| + |\xi| + |\eta| \ge 2^j\}} 
\\
&\le C_{\alpha', \beta', \gamma'}(x, \xi, \eta) 
\prod_{i=1}^2 (1+|\xi_i| + |\eta_i|)^{m_i + \delta_i |\alpha_i| - \rho_i (|\beta_i| + |\gamma_i|)} 
\, \mathbf{1}_{\{|x| + |\xi| + |\eta| \ge 2^j\}}. 
\end{align*}
Consequently, we conclude
\begin{align*}
P_{\alpha, \beta, \gamma}^{m, \rho, \delta}(\sigma_j - \sigma)
\lesssim \sum_{\substack{\alpha' \le \alpha \\ \beta' \le \beta \\ \gamma' \le \gamma}} 
C_{\alpha', \beta', \gamma'}(x, \xi, \eta) \, \mathbf{1}_{\{|x| + |\xi| + |\eta| \ge 2^j\}},
\end{align*} 
where the implicit constant is independent of $j$. By the fact $\lim_{|x| + |\xi| + |\eta| \to \infty} C_{\alpha', \beta', \gamma'}(x, \xi, \eta) = 0$ for all multi-indices $\alpha'$, $\beta'$, and $\gamma'$, we assert that $\lim_{j \to \infty} P_{\alpha, \beta, \gamma}^{m, \rho, \delta}(\sigma_j - \sigma) = 0$. 
\end{proof}

\end{document}